\documentclass{amsart}

\usepackage{graphicx}
\usepackage{pstricks}

\newtheorem{thm}{Theorem}[section]

\newtheorem{defn}[thm]{Definition}
\newtheorem{lemma}[thm]{Lemma}

\newtheorem{cor}[thm]{Corollary}
\newtheorem{remark}[thm]{Remark}
\newtheorem{example}[thm]{Example}

\usepackage{amsmath}
\usepackage{amsxtra}
\usepackage{amscd}
\usepackage{amsthm}
\usepackage{amsfonts}
\usepackage{amssymb}
\usepackage{eucal}
\newcommand{\rt}{{\rm t}}

\newcommand{\cP}{\mathcal P}
\newcommand{\bM}{{\mathbf M}}

\newcommand{\Z}{{\mathbb Z}}
\newcommand{\Q}{{\mathbb Q}}
\newcommand{\R}{{\mathbb R}}

\newcommand{\bm}{{\mathbf m}}

\newcommand{\bx}{{\mathbf x}}
\newcommand{\by}{{\mathbf y}}

\newcommand{\al}{{\alpha}}
\newcommand{\wG}{\widetilde{G}}

\numberwithin{equation}{section}

\begin{document}
\title{Q-systems, heaps, paths and cluster positivity}
\author{Philippe Di Francesco} 
\address{Institut de Physique Th\'eorique du Commissariat \`a l'Energie Atomique, 
Unit\'e de Recherche associ\'ee du CNRS,
CEA Saclay/IPhT/Bat 774, F-91191 Gif sur Yvette Cedex, 
FRANCE. e-mail: philippe.di-francesco@cea.fr}

\author{Rinat Kedem}
\address{Department of Mathematics, University of Illinois MC-382, Urbana, IL 61821, U.S.A. e-mail: rinat@illinois.edu}
\date{\today}
\begin{abstract}
We consider the cluster algebra associated to the $Q$-system for $A_r$
as a tool for relating $Q$-system solutions to all possible sets of
initial data.  Considered as a discrete integrable dynamical system,
we show that the conserved quantities are partition functions of hard
particles on certain weighted graphs determined by the choice of
initial data. This allows us to interpret the solutions of the system
as partition functions of Viennot's heaps on these graphs, or as
partition functions of weighted paths on a dual graphs.  The
generating functions take the form of finite continued fractions. In
this setting, the cluster mutations correspond to local rearrangements
of the fractions which leave their final value unchanged.  Finally,
the general solutions of the $Q$-system are interpreted as partition
functions for strongly non-intersecting families of lattice paths on
target lattices. This expresses all cluster variables as manifestly
positive Laurent polynomials of any initial data, thus proving the
cluster positivity conjecture for the $A_r$ $Q$-system. We also give
the relation to domino tilings of deformed Aztec diamonds with
defects.
\end{abstract}

\maketitle
\tableofcontents

\section{Introduction}
The $A_r$ $Q$-system is a recursion relation satisfied by characters
of special irreducible finite-dimensional modules of the Lie algebra
$A_r$ \cite{KR}. It is a discrete integrable dynamical system.  On the
other hand, the relations of the $Q$-system are mutations in a cluster
algebra defined in
\cite{Ke07}.  One of the goals of this paper is to prove the positivity
property of the corresponding cluster variables, by using the
integrability property.

We do this by solving the discrete integrable system, which can be
mapped to several different types of statistical models: path models,
heaps on graphs, or domino tilings. The choice initial conditions for
the recursion relations determines the specific model, one for each
choice of initial data. The Boltzmann weights are Laurent monomials in
the initial data. Our construction gives an explicit solution of the
$Q$-system as a function of any possible set of initial conditions.

It is a conjecture of Fomin and Zelevinsky \cite{FZI} that the cluster
variables at any seed $x$ of a cluster algebra, expressed as a
function of the cluster variables of any other seed $y$ of the
algebra, are Laurent polynomials {with non-negative
coefficients}. In the language of dynamical systems, the choice of
seed variables $y$ corresponds to the choice of initial conditions.


The methods used in this paper appear to be new in the context of the
positivity conjecture. They have the advantage that
they give explicit solutions, and extend immediately to $T$-systems
\cite{DFK09a} and to certain integrable non-commutative cluster algebras
introduced by Kontsevich \cite{Kon,DFK09b}.

\subsection{The $Q$-system}
First, let us recall some definitions.
Let $I_r=\{1,...,r\}$, and consider the family of commutative variables
$\{Q_{\alpha,n} ~:~ \alpha\in I_r, n\in \Z\}$, related by a recursion
relation of the form:
\begin{equation}\label{qsys}
Q_{\alpha,n+1}Q_{\alpha,n-1}=Q_{\alpha,n}^2-Q_{\al+1,n}Q_{\al-1,n},\quad
Q_{0,n}=Q_{r+1,n}=1,\quad
(n\in \Z,~ \alpha\in I_r).
\end{equation}
A solution of this system is specified by giving a set of {\em
  boundary conditions}. 

The original $A_r$ $Q$-system \cite{KR} is the recursion relation
\eqref{qsys} (where $n\geq 1$), together with the boundary conditions
\begin{equation}\label{KRboundary}
Q_{\alpha,0} = 1, \quad  Q_{\alpha,1}={\rm ch} V(\omega_\alpha), \quad \alpha\in I_r.
\end{equation}
Here, $V(\omega_\alpha)$ is one of the $r$ fundamental representations
of the Lie algebra $\mathfrak{sl}_{r+1}$.  The solutions, written as
functions of the initial variables, are the characters of the $A_r$
Kirillov-Reshetikhin modules \cite{KR}:
$$Q_{\al,n}={\rm ch}V(n \omega_\alpha).$$

The boundary conditions \eqref{KRboundary} are singular, in that
$Q_{\alpha,-1}=0$. All solutions are therefore polynomials in
$\{Q_{\al,1} : \al\in I_r\}$ (Lemma 4.2 of \cite{DFK08}).  In this
paper, we relax this boundary condition.  Moreover, as we are
interested in the positivity property of cluster algebras, we
renormalize the variables in \eqref{qsys}.\footnote{Alternatively, we
can introduce coefficients in the second term (see Appendix A in
\cite{DFK08}). The two approaches are equivalent in this case.}
Let
\begin{equation}
\label{redef}
R_{\alpha,n}= \epsilon_\alpha Q_{\alpha,n}, \qquad 
\epsilon_\alpha=e^{i\pi \sum_{\beta=1}^r
(C^{-1})_{\alpha,\beta}}=
e^{i\pi\alpha(r+1-\alpha)/2}
\end{equation}
where $C$ the Cartan matrix of $A_r$. Then
\begin{equation}
\label{renom}
R_{\alpha,n+1}R_{\alpha,n-1}=R_{\alpha,n}^2+R_{\al+1,n} R_{\al-1,n},
\quad R_{0,n}=R_{r+1,n}=1, \quad (\alpha\in I_r, ~ n\in \Z).
\end{equation}

In the rest of the paper we work only with the renormalized variables
in \eqref{redef}, and henceforth we refer to Equation \ref{renom} as
the $A_r$ $Q$-system.

\subsection{$Q$-systems as Cluster algebras}

Instead of specifying the boundary conditions of the form
\eqref{KRboundary}, we may specify much more general boundary conditions by
picking a set $2r$ variables in a consistent manner, and setting them
to be formal variables. Any solution is then a function of
these formal variables. In order to explain what we mean by ``a
consistent manner'', we can use the formulation of these recursion
relations as mutations in a cluster algebra.

It was shown in \cite{Ke07} that the system \eqref{renom} can be
expressed in terms of a cluster algebra \cite{FZI}:
\begin{thm}\label{Qcluster}\cite{Ke07} The
  equations \eqref{renom} for each $\al$ and $n$ are mutations
  in the cluster algebra without
  coefficients, of rank $2r$, defined by the seed $(\bx_0,B_0)$, where
  $\bx_0$ is a cluster variable and $B_0$ is an exchange matrix, with
\begin{equation}\label{initseed}
\bx_0 = (R_{1,0},...,R_{r,0};R_{1,1},...,R_{r,1}),\quad
B_0=\left(\begin{array}{cc} 0 & -C \\ C & 0 \end{array}\right),
\end{equation}
where $C$ is the Cartan matrix of $A_r$.
\end{thm}

A cluster algebra of rank $n$ without coefficients is defined as
follows \cite{FZI}.  Consider an an $n$-regular tree $\mathbb T_n$,
with nodes connected via labeled edges. Each node is attached to $n$
edges with distinct labels $1,...,n$. To each node $t$, we associate a
seed consisting of a cluster variable and an exchange matrix.
The cluster variable $\mathbf x_t$ has $n$ components,
$
\mathbf x_t = (x_{t,0},...,x_{t,n}),
$
and the exchange matrix $B_t$ is an $n\times n$ skew-symmetric
matrix. 

Cluster variables at connected
nodes are related by mutations given by the exchange matrix.
Suppose node $t$ is connected to node $t'$ by an edge labeled $k$, then the
seeds at these nodes are related by a mutation
$\mu_k$. The effect of this mutation is as follows:
\begin{eqnarray*}
\mu_k(x_{t,j})&=&x_{t',j} =\left\{ \begin{array}{ll}
 x_{t,j}^{-1}\left( \prod_{i=1}^n x_{t,i}^{[(B_t)_{ij}]_+} +
\prod_{i=1}^n x_{t,i}^{[-(B_t)_{ij}]_+}\right), & k=j\\
x_{t,j} & \hbox{otherwise}.\end{array}\right.
\end{eqnarray*}
Here, $[n]_+={\rm max}(0,n)$. 
The matrices $B_t$ and $B_{t'}$ are related via the mutation $\mu_k$
as follows.
\begin{eqnarray}
\mu_k((B_t))_{ij} &=& \left\{ \begin{array}{ll}
-(B_t)_{ij} & \hbox{if $i=k$ or $j=k$,}\nonumber \\
(B_{t'})_{ij} = (B_{t})_{ij} + {\rm sign}((B_t)_{ik}) [
(B_{t})_{ik}(B_t)_{kj}]_+ & \hbox{otherwise}\end{array}\right. \label{Bmut}
\end{eqnarray}
The cluster algebra is the commutative algebra over $\Q$ of the
cluster variables.

In the case of the cluster algebra defined in Theorem \ref{Qcluster},
there is a subgraph $\mathcal G_r$ of $\mathbb T_{2r}$, which includes
the node with seed \eqref{initseed}, and is the maximal subgraph with
the property that
the mutation of the cluster variables along each edge of the
graph is one of the 
equations \eqref{renom}.
The union of the cluster variables over all nodes of the graph
is the set $\{R_{\alpha,n}:\alpha\in I_r, n\in
\Z\}$.\footnote{Note that this graph is different from the graphs described in
\cite{Ke07,DFK08}, which were minimal rather than maximal.}

To describe the seeds in $\mathcal G_r$, note that the mutation of the
variable $R_{\al,n-1}$ into $R_{\al,n+1}$
uses the variables 
$\{R_{\al,n-1}, R_{\al-1,n}, R_{\al,n}, R_{\al+1,n}\}$. If these are
entries of the cluster variable $\bx_t$ at node $t$, then it is
possible to apply the mutation $t\to t'$ which sends
$R_{\al,n-1}\to R_{\al,n+1}$.
With our indexing convention, this is the mutation $\mu_\al$ if $n$ is
odd, $\mu_{\al+r}$ if $n$ is even.  Such a mutation will be called a
``forward'' mutation, as it increases the value of the index $n-1\to
n+1$.  (Cluster mutations are involutions, hence we also have
``backward'' mutations which send $n+1\to n-1$.)

We see that any cluster variable in $\mathcal G_r$ has the properties:
(i) The variable consists of pairs the form
$(R_{\al,m_\al},R_{\al,m_\al+1})$, $\al\in I_r$ and (ii)
$m_{\al-1}-1\leq m_\al \leq m_{\al-1}+1$ for all $\al\in I_r$.  Define
a path on the two-dimensional lattice in the $(n,\al)$-plane, by
connecting the points $\bM=\{(m_\al,\al)\}_{\al=1}^r$. Then condition
(ii) implies that $\bM$ is a Motzkin path, with steps of type $(1,1)$,
$(-1,1)$ and $(0,1)$. Therefore we have

\begin{lemma}\label{motseedgen}
The nodes of ${\mathcal G}_r$ are in bijection with
Motzkin paths on the square lattice with $r-1$ steps of the form
$(1,1)$, $(-1,1)$ or $(0,1)$, connecting vertices with integer
coordinates $(n,\al)$ such that $1\leq \al \leq r$.
\end{lemma}

We define $\bx_\bM$ to be the seed with the variables
$\{R_{\al,m_\al+i}:i=0,1, \al\in I_r \}$. 
Let $\bM_0$ denote the Motzkin path with $m_\al=0$, $\al\in I_r$.
Then $\bx_0=\bx_{\bM_0}$. Forward mutations $\mu_\al$ and
$\mu_{\al+r}$ act on Motzkin paths by increasing or
decreasing one of the indices $m_\al$. If the resulting path
is also a Motzkin path, such mutations are guaranteed to be of the type
\eqref{renom}. 


The graph $\mathcal G_r$ is an infinite strip, and we display a section of it for $r=3$
in Figure \ref{fig:slfourR}, where we have
identified the nodes sharing the same cluster variables. Most of the
results of this paper can be reduced to working with the fundamental
domain of this strip.
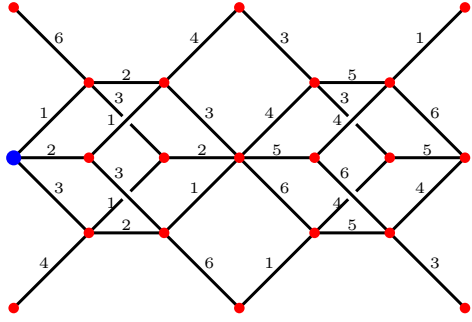
\begin{figure}
\begin{center}
\psset{unit=1mm,linewidth=.4mm,dimen=middle}

\begin{pspicture}(0,0)(60,40)
\psline(0,0)(10,10)(0,20)(10,30)(0,40)
\psline(10,10)(20,20)(10,30)
\psline[linewidth=1.5mm,linecolor=white](20,10)(10,20)(20,30)
\psline(20,10)(10,20)(20,30)
\psline(10,10)(20,10)
\psline(20,30)(10,30)
\psline(0,20)(10,20)
\psline(30,0)(20,10)(30,20)(20,30)(30,40)(40,30)(30,20)(40,10)(30,0)
\psline(40,10)(50,20)(40,30)
\psline[linewidth=1.5mm,linecolor=white](50,10)(40,20)(50,30)
\psline(50,10)(40,20)(50,30)
\psline(20,20)(40,20)
\psline(60,0)(50,10)(60,20)(50,30)(60,40)
\psline(40,10)(50,10)
\psline(50,20)(60,20)
\psline(40,30)(50,30)
\multips(0,0)(30,0){3}{\pscircle*[linecolor=red](0,0){.7}}
\multips(0,40)(30,0){3}{\pscircle*[linecolor=red](0,0){.7}}
\multips(0,20)(10,0){7}{\pscircle*[linecolor=red](0,0){.7}}
\multips(10,10)(10,0){2}{\pscircle*[linecolor=red](0,0){.7}}
\multips(40,10)(10,0){2}{\pscircle*[linecolor=red](0,0){.7}}
\multips(10,30)(10,0){2}{\pscircle*[linecolor=red](0,0){.7}}
\multips(40,30)(10,0){2}{\pscircle*[linecolor=red](0,0){.7}}
\rput(4,6){\tiny 4}
\rput(6,16){\tiny 3}
\rput(5,21){\tiny 2}
\rput(4,26){\tiny 1}
\rput(6,36){\tiny 6}
\rput(15,11){\tiny 2}
\rput(15,31){\tiny 2}
\rput(13,14){\tiny 1}
\rput(13,25){\tiny 1}
\rput(14,18){\tiny 3}
\rput(14,28){\tiny 3}
\rput(24,36){\tiny 4}
\rput(26,26){\tiny 3}
\rput(25,21){\tiny 2}
\rput(24,16){\tiny 1}
\rput(26,6){\tiny 6}
\rput(34,6){\tiny 1}
\rput(36,16){\tiny 6}
\rput(35,21){\tiny 5}
\rput(34,26){\tiny 4}
\rput(36,36){\tiny 3}
\rput(45,11){\tiny 5}
\rput(45,31){\tiny 5}
\rput(43,14){\tiny 4}
\rput(43,25){\tiny 4}
\rput(44,18){\tiny 6}
\rput(44,28){\tiny 3}
\rput(54,36){\tiny 1}
\rput(56,26){\tiny 6}
\rput(55,21){\tiny 5}
\rput(54,16){\tiny 4}
\rput(56,6){\tiny 3}
\rput(0,20){\pscircle*[linecolor=blue](0,0){1}}
\end{pspicture}
\end{center}
\caption{\small The subgraph of $Q$-system mutations for $A_3$. The blue dot stands for the node $0$.}
\label{fig:slfourR}
\end{figure}

\subsection{Outline of the paper}

The Laurent property \cite{FZLaurent} guarantees that the cluster
variables are Laurent polynomials when expressed as functions of the
cluster variables of any seed in the cluster algebra. It is a general
conjecture \cite{FZI} that the coefficients of the monomials in these
Laurent polynomials are non-negative integers.
In this paper, we prove positivity
for the cluster variables in $\mathcal G_r$.

The $Q$-system is a way of relating  the solutions of
\eqref{renom} to all possible sets of initial data.
Mutations of the cluster algebra allow us to move within the set of
possible initial data, as long as this initial data is of the form
$\bx_\bM$.  Our aim is to give an explicit combinatorial 
description of $R_{\al,n}$, for each choice of initial data
$\bx_\bM$ within this set, as the partition function for a statistical
model with positive Boltzmann weights.

We proceed using the following steps:

\noindent{\bf Step 1.} We give a combinatorial description of
$\{R_{1,n},~ n\in\Z\}$ as functions of the initial data $\bx_0$.  All
other variables $\{R_{\al,n},~ \al>1\}$ are discrete Wronskians of
$\{R_{1,n+j},~ 1-\al\leq j\leq \al-1\}$. The variables $\{R_{1,n},~ n\in
\Z\}$ satisfy a linear recursion relation with constant coefficients,
which are the conserved quantities of the $Q$-system. Here, the
integrability of the system plays a crucial role.

The conserved quantities are partition functions of hard particles on
a weighted graph $G_r$, with weights which are Laurent monomials in
the cluster variables $\bx_0$. Therefore the generating function
$F_1^{(r)}(t)=\sum_{n\geq 0}t^n R_{1,n}$ is equal to the partition
function of Viennot's heaps on $G_r$. Thus, there is a simple
expression for $F_1^{(r)}(t)$ as a finite continued fraction.

Elementary rearrangements of the continued fraction allow us to
express $F_1^{(r)}(t)$ as the generating function for heaps on
different graphs with different weights. Our goal is then to prove
that these rearrangements are mutations of the initial data
$\bx_{\bM}$.

\noindent{\bf Step 2.} We use a heap-path correspondence to re-express
$\{R_{1,n}(\bx_0)\}_{n\in \Z_+}$ as partition
functions for weighted paths on a dual graph ${\wG}_r$.

\noindent{\bf Step 3.} 
For each Motzkin path $\bM$, we construct a graph $\Gamma_\bM$ and
edge weights $y_e(\bM)$ which are Laurent monomials in the cluster
variables $\bx_\bM$. The variables $\{R_{1,n}\}_{n\in
\Z_+}$ are expressed as the generating functions for weighted paths on
$\Gamma_\bM$. This proves the positivity conjecture for $R_{1,n}$ and
the claim that continued fraction rearrangements are cluster mutations
of initial data.

\noindent{\bf Step 4.} We use the discrete Wronskian expression for
$R_{\al,n}$ to interpret these variables as partition functions for
families of {\it strongly} non-intersecting paths on $\Gamma_\bM$.
This is done by a generalization of the Lindstr\"om-Gessel-Viennot
\cite{LGV1,LGV2} determinant formula for the counting of
non-intersecting lattice paths.  This implies the positivity property
for $R_{\al,n}$ as functions of $\bx_\bM$.

The paper is organized as follows.
In Section \ref{props}, we show that the $Q$-system amounts to a
discrete Wronskian equation for the $R_{1,n}$'s, and that $R_{\al,n}$
with $\al>1$ are expressed as $\al\times \al$ discrete Wronskians of
$R_{1,n}$. We deduce that $R_{1,n}$ satisfies a linear recursion
relation with constant coefficients, interpreted as the conserved
quantities of the $Q$-system.

In Section \ref{cons}, the conserved quantities  are interpreted 
as partition functions for hard particles on a certain target graph
$G_r$, with vertex weights which depend on the initial data
$\bx_0$ of \eqref{initseed}.
We rephrase the linear recursion relation for
$R_{1,n}$ in terms of generating functions which are simple rational
functions.  

In Section \ref{posit} we re-interpret these rational
functions as
partition functions for Viennot's heaps \cite{HEAPS1} on the
same target graph $G_r$.  This gives a first proof of the positivity
of $R_{1,n}$ as a Laurent polynomial of the initial data $\bx_0$
\eqref{initseed}. We also give an explicit expression for the
generating function of $R_{1,n}$'s, as a finite continued fraction.
We show how a simple rearrangement lemma for fractions allows, by
iterative use, to rewrite the partition function for heaps on $G_r$ as
a partition function for heaps on other graphs, with weights
accordingly transformed. The proof that these rearrangements
correspond to mutations of the cluster variables appears in Section
\ref{pathinter}.

In Section \ref{heappa} we reformulate the heap partition function in
terms of weighted paths on ``dual" target graphs, the language of
which is more amenable to mutations. For each Motzkin path, we
construct a corresponding graph with edge weights which are monomials
in $\bx_\bM$. Partition functions for paths on these graphs are the
cluster variables $R_{1,n}$.
For completeness, we also explain how to construct the dual graphs for
the heap models from these graphs.

In Section \ref{pathinter} we prove that 
$R_{1,n}$ is the partition function for
weighted paths on a target graph $\Gamma_\bM$ with weights determined
by the initial data $\bx_\bM$. Using the transfer matrix formulation,
we prove the statement that fraction
rearrangements correspond to mutations of the initial data.
This leads to the main positivity theorem for the cluster
variables $R_{1,n}$, in terms of any initial data $\bx_\bM$.

The extension of this result to $R_{\al,n}$, $\al>1$ is presented in
Section \ref{secpaths}, where 
$R_{\al,n}$ is interpreted
\`a la Lindstr\"om-Gessel-Viennot as the partition function for
families of strongly non-intersecting 
weighted paths on the same target graph as for $R_{1,n}$.  This
completes the statistical-mechanical 
interpretation of all the solutions of the $Q$-system, and proves the
positivity of all the cluster variables  
of the subgraph corresponding to the $Q$-system, when expressed in
terms of cluster  
variables at any of its nodes. 

In Section \ref{asympto} we discuss the limiting case of
$A_{\infty/2}$ and present various exact and asymptotic path
enumeration results, which correspond to picking initial data $x_\bM$
with entries all equal to $1$.

Finally, since the $Q$-system is a limit of the $T$-system, aka the
octahedron equation, in Section \ref{aztec} we give the relation of
our results to the known results on the octahedron equation.  We show
how to relate weighted domino tilings of the Aztec diamond to our
weighted non-intersecting families of lattice paths, and to interpret
the result of cluster mutations in that language as weighted tilings
of suitably deformed Aztec diamonds, by means of dominos and also
pairs of square ``defects".

The appendices include explicit examples of our constructions for the
case $A_3$.

\vskip.15in

\noindent{\bf Acknowledgements:} 
We thank M. Bergvelt, C. Krattenthaler, A. Postnikov, H. Thomas, 
and particularly S. Fomin for many interesting discussions.
We thank the organizers of the semester ``Combinatorial Representation Theory"
and the Mathematical Science Research Institute, Berkeley, CA, USA
for hospitality, as well as the organizers of the program ``Combinatorics and Statistical Physics" and the
Erwin Schr\"odinger International Institute for Mathematical Physics, Vienna, Austria. R.K. acknowledges the
hospitality of the Institut des Hautes Etudes Scientifiques, Bures-sur-Yvette, France.
P. D.F. acknowledges the support of the ENIGMA research training network MRTN-CT-2004-5652,
the ANR program GIMP, and the ESF program MISGAM.
R. K.'s research is supported by NSF grants DMS 0500759 and 0802511.

\section{Properties of the $Q$-system}\label{props}
\subsection{The fundamental domain of seeds}
Our goal is give explicit expressions for $R_{\al,n}$ as a function of
any initial seed data $\bx_{\bM}$.  First, we use the symmetries of the
$Q$-system to enable us to restrict our attention to a finite
fundamental domain of initial seeds, parametrized by Motzkin paths
which have a minimum node at $0$.

There are three obvious symmetries of the system. A symmetry $\sigma:
\{R_{\al,n}\}\to \{R_{\al,n}\}$ of
\eqref{renom} is a map with the property that
$R_{\al,n}=f(\bx)$ then $\sigma(R_{\al,n})=f(\sigma(\bx))$.
Equation \eqref{renom} is invariant under
$\sigma(R_{\alpha,n})=R_{\alpha,-n+1}$. Therefore,
\begin{lemma}\label{firstrem} ``Time reversal'':
\begin{eqnarray}\label{rem1}
 {\rm if}\  R_{\al,n}&=&f(R_{1,0}, R_{2,0},...,R_{r,0};R_{1,1},R_{2,1},...,R_{r,1}),
\nonumber\\
{\rm then}\ 
  R_{\al,-n+1}&=&f(R_{1,1},R_{2,1},...,R_{r,1};R_{1,0}, R_{2,0},...,R_{r,0}).
\end{eqnarray}
for all $n\in \Z$ and $\al\in I_r$.
\end{lemma}
We also have the reflection symmetry of \eqref{renom}, $\sigma(R_{\al,n})=R_{r+1-\al,n}$. Therefore, 
\begin{lemma}\label{secrem}
\begin{eqnarray}\label{rem2}
{\rm if}\ R_{\al,n}&=&f(R_{1,0},R_{2,0},...,R_{r,0};R_{1,1},R_{2,1},...,R_{r,1}) \nonumber\\
{\rm then}\ R_{r+1-\al,n}&=&
f(R_{r,0},R_{r-1,0},...,R_{1,0};R_{r,1},R_{r-1,1},...,R_{1,1}).
\end{eqnarray}
for all $n\in \Z$ and $\al\in I_r$.
\end{lemma}
Finally, we have the translational invariance,
$\sigma(R_{\al,n})=R_{\al,n+k}$:
\begin{lemma}\label{clustinv}
\begin{eqnarray}\label{transinv}
{\rm if}\ R_{\al,n}&=&
f(R_{1,0},R_{2,0},...,R_{r,0};R_{1,1},R_{2,1},...,R_{r,1}),
\nonumber\\
{\rm then}\ R_{\al,n+k}&=&
f(R_{1,k},R_{2,k},...,R_{r,k};R_{1,k+1},R_{2,k+1},...,R_{r,k+1}).
\end{eqnarray}
for all $n,k\in \Z$ and $\al\in I_r$.
\end{lemma}
This lemma may also be viewed
as a special case
of the substitution property of deformed $Q$-systems defined in
\cite{DFK}. 

For any Motzkin path $\bM=(m_1,...,m_r)$, let
$f_{\al,n}^{(\bM)}(\bx)$ denote the function of $\bx$ such that
$$
R_{\al,n}=f_{\al,n}^{(\bM)}(\bx_\bM).
$$
Let $\bM+k=(m_1+k,...,m_r+k)$.
Then  Equation \eqref{transinv} can be written as
\begin{equation}\label{invtrans}
R_{\al,n}=f_{\al,n}^{(\bM_0)}(\bx_{\bM_0})=f_{\al,n-k}^{(\bM_0)}(\bx_{\bM_0+k}). 
\end{equation}
More generally, 
\begin{equation}\label{fundaMiam}
R_{\al,n+k}=f_{\al,n+k}^{(\bM)}(\bx_{\bM})=f_{\al,n+k}^{(\bM_0)}(\bx_{\bM_0})
=f_{\al,n}^{(\bM_0)}(\bx_{\bM_0+k})
=f_{\al,n}^{(\bM)}(\bx_{\bM+k}).
\end{equation}
Therefore,
\begin{thm}\label{inverposit}
Let $R_{\al,n}=f_{\al,n}^{(\bM)}(\bx_\bM)$, where $f_{\al,n}^{(\bM)}$
is a positive Laurent polynomial of $\bx_\bM$, $\al\in I_r$,
$n\in \Z$. Then as a function of
$\bx_{\bM+k}$, $k\in \Z$, $R_{\al,n}$ is also a positive Laurent polynomial,
$R_{\al,n}=f_{\al,n}^{(\bM+k)}(\bx_{\bM+k})$.
\end{thm}
We can thus restrict our attention to fundamental domain:
\begin{figure}
\centering
\includegraphics[width=11.cm]{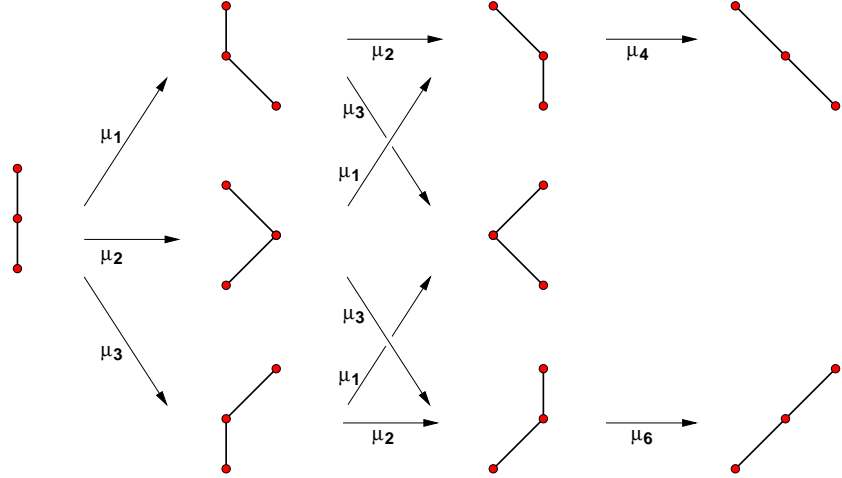}
\caption{\small The Motzkin paths for the 9 seeds in the fundamental domain 
for the $A_3$ case, and the
mutations relating them. The leftmost vertices of each Motzkin path lie on the vertical
axis $x=0$.}\label{fig:nine}
\end{figure}

\begin{defn}\label{motseed}
The fundamental domain ${\mathcal F}_r$ for the $A_r$ $Q$-system is
indexed by the Motzkin paths with $r-1$ steps of the form
$\{(m_\al,\al)\}_{\al\in I_r}$, with Min$_{\al\in I_r}(m_\al)=0$.
\end{defn}
There are exactly $3^{r-1}$ Motzkin paths in the fundamental
domain. 
As an example, Figure \ref{fig:nine} shows the $9$ Motzkin
paths of the fundamental domain for $A_3$ and the mutations relating
them, {\em c.f.} the graph ${\mathcal G}_3$ in Figure \ref{fig:slfourR}.

\subsection{Discrete Wronskians}
For any matrix $M$, let $M_{i_1,...,i_k}^{j_1,...,j_m}$ be the matrix
obtained by it by removing rows $i_1,...,r_k$ and columns $j_1,...,j_m$.

\subsubsection{Pl\"ucker relations}
Let $T$ be an
$n\times (n+k)$-matrix. The Pl\"ucker relations for the minors of $T$ are
\begin{equation}\label{plucker}
\vert T^{a_1,...,a_k}\vert \,\vert T^{b_1,...,b_k}\vert=\sum_{p=1}^k
\vert T^{b_p,a_2...,a_k}\vert  
\, \vert T^{b_1,...,b_{p-1},a_1,b_{p+1},...,b_k}\vert .
\end{equation}
In particular, when $k=2$,
\begin{equation}\label{specplu}
  \vert T^{a_1,a_2}\vert \, \vert T^{b_1,b_2}\vert =\vert T^{b_1,a_2}\vert \, 
  \vert T^{a_1,b_2}\vert+\vert T^{b_2,a_2}\vert \, \vert T^{b_1,a_1}\vert.
\end{equation}
Let $n=r+1$, $a_2=j_1$, $b_2=j_2$, and $(T)_{i,a_1}=\delta_{i,i_1}$,
$(T)_{i,b_1}=\delta_{i,i_2}$. Then Equation \eqref{specplu} gives the
Desnanot-Jacobi formula for the minors of the matrix $M=T^{a_1,b_1}$ of size
$(r+1)\times (r+1)$:
\begin{equation}\label{dj}
\vert M\vert \, 
\vert M_{i_1,i_2}^{j_1,j_2}\vert= 
\vert M_{i_1}^{j_1}\vert \, \vert M_{i_2}^{j_2}\vert
-\vert M_{i_1}^{j_2}\vert \, \vert M_{i_2}^{j_1}\vert .
\end{equation}

\subsubsection{Wronskian formula for $R_{1,n}$}
Using \eqref{renom}, it is possible to eliminate the variables
$\{R_{\alpha,n}\}_{\alpha> 1}$ in favor of $\{R_{1,n}\}$.  The
remaining equations determine $\{R_{1,n}\}$ in terms of the initial data.
As a consequence, $R_{1,n}$ satisfies a linear recursion relation with
constant coefficients.  This can then be extended trivially to all
$R_{\alpha,n}$.

Define the $\al\times \al$ matrix $M_{\al,n}$ with $(M_{\alpha,n})_{i,j}=
R_{1,n+i+j-1-\alpha}$. That is,
\begin{equation}\label{discretW} 
M_{\al,n}=
\left( \begin{matrix}
R_{1,n-\alpha+1} & R_{1,n-\alpha+2} &\cdots & R_{1,n} \cr
R_{1,n-\alpha+2} & R_{1,n-\alpha+3} &\cdots & R_{1,n+1} \cr
\vdots & \vdots & \ddots & \vdots\cr
R_{1,n}&R_{1,n+1}& \cdots &R_{1,n+\alpha-1}
\end{matrix}
\right) .
\end{equation}
and define the discrete Wronskian determinant to be $W_{\al,n}=|M_{\al,n}|$.

\begin{lemma}\label{Rwronsk}
We have $R_{\alpha,n}=W_{\alpha,n}$.
\end{lemma}
\begin{proof}
Applying the Desnanot-Jacobi formula \eqref{dj} to the $(\al+1)\times (\al+1)$ matrix 
$M$ with entries
$$(M)_{i,j}=R_{1,n+i+j-\al-2}$$ 
with the choice of rows $i_1=1,i_2=\al+1$, and columns $j_1=1$,
$j_2=\al+1$, we have
\begin{equation}\label{desnajaco} 
W_{\alpha,n+1}W_{\alpha,n-1}=W_{\alpha,n}^2+W_{\alpha+1,n}
W_{\alpha-1,n}
\end{equation}
for any sequence $R_{1,n}$, and any $\alpha\geq 1$, with the convention that $W_{0,n}=1$
for all $n$. The sequence $W_{\al,n}$ is the unique solution to eq.\eqref{desnajaco}
such that $W_{0,n}=1$ and $W_{1,n}=R_{1,n}$. Comparing this to the Q-system \eqref{renom}, we deduce that
$R_{\alpha,n}= W_{\alpha,n}$, $\al=1,...,r$, and the Lemma follows.
\end{proof}
The boundary condition $R_{r+1,n}=1$ yields
the following polynomial
relation for $R_{1,n}$:
\begin{cor}
\begin{equation}\label{finalA}
W_{r+1,n}=\left\vert \begin{matrix}
R_{1,n-r} & R_{1,n-r+1} &\cdots & R_{1,n} \cr
R_{1,n-r+1} & R_{1,n-r+2} &\cdots & R_{1,n+1} \cr
\vdots & \vdots & \ddots & \vdots\cr
R_{1,n}&R_{1,n+1}& \cdots &R_{1,n+r}
\end{matrix}
\right\vert =1.
\end{equation}
\end{cor}

\subsubsection{Integrals of motion}
The determinant $W_{r+1,n}$ is a discrete version of the
Wronskian determinant $W(f_1,...,f_r)=\det_{i,j}(f_i^{(j-1)})$. In the
theory of linear differential equations, the Wronskian of $r$ linearly
independent solutions to an $r$-th order linear differential equation
is a constant. This is proved by differentiating the Wronskian and
noting that a linear combination of its columns vanishes, due to the
differential equation.  Conversely, if the Wronskian is a (non-zero)
constant (so that its columns are linearly independent), there exists
a vanishing linear combination between the column vectors of its
derivative, namely $f_i^{(r)}=\sum_{j=1}^{r-1} a_j f_i^{(j-1)}$, where
the $f$'s are a linearly independent set of solutions of these
equations.

\begin{thm}\label{alphaone} 
The variables $\{R_{1,n}\}_{n\in\Z}$ satisfy a linear recursion
relation involving $r+2$ terms:
\begin{equation}\label{forec}
\sum_{m=0}^{r+1} (-1)^{m}c_{r+1-m} R_{1,n+m} =0, \quad n\in\Z,
\end{equation}
with the coefficients $c_0=c_{r+1}=1$, and with $c_1,c_2,...,c_r$ some
constant (independent of $n$) coefficients determined by the initial
conditions. 
\end{thm}
\begin{proof}

In analogy with the continuous situation, consider the discrete
derivative $W_{r+1,n+1}-W_{r+1,n}=0$.  
Since $W_{r+1,n}$ and
$W_{r+1,n+1}$ have $r$ identical columns ($(W_{r+1,n+1})_{i,j}=(W_{r+1,n})_{i+1,j}$, $i\in I_r$),
\begin{equation}\label{difwron}
W_{r+1,n+1}-W_{r+1,n}= \det_{1\leq i,j\leq r+1}
\left( R_{1,n+i+j-r-1} -(-1)^{r}\delta_{j,r+1} R_{1,n+i-r-1}\right) =0.
\end{equation}
As a consequence, there exists a non-trivial linear combination of the
columns of this difference which vanishes. From the form of the
entries of these columns (in which the indices are shifted by $-1$
relative to each other) the coefficients of this linear combination
are independent of $n$.
\end{proof}



\section{Conserved quantities and Hard Particles}\label{cons}

\subsection{Conserved quantities of the Q-system}
Since $W_{r+1,n}=1$, it is
a conserved quantity, i.e. it is independent of $n$. More generally,
we claim that there are $r+1$ linearly independent conserved
quantities, and therefore the $Q$-system is a discrete integrable
system in the Liouville sense.

\begin{thm}\label{conserved}
The following polynomials
$$
c_{i-1}= |(M_{r+2,n})_{r+2}^{r+2-i}|,\quad i=0,...,r
$$
where $c_{r+1}=c_0=1$, are independent of $n$, and $c_0,...,c_{r}$ are the
linearly independent conserved quantities of the $A_r$ $Q$-system.
\end{thm}
\begin{proof}
This follows from the fact that $W_{r+2,n}=0$, as a consequence of the
boundary condition $W_{r+1,n}=1$ and the $Q$-system relation. The
conserved quantities are the minors of the expansion of this
determinant with respect to the last row, as in Equation \eqref{forec}.
We get only $r+1$ linearly independent minors, since $c_0=W_{r+1,k-1}=1=
W_{r+1,k}=c_{r+1}$.
\end{proof}

\begin{example}\label{aonecval}
For $r=1$, we have
$$
c_1= \left\vert \begin{matrix}
R_{1,k-2} & R_{1,k}\\
R_{1,k-1} & R_{1,k+1}
\end{matrix}\right\vert= {R_{1,1}\over R_{1,0}}+{1\over R_{1,0}R_{1,1}}+{R_{1,0}\over R_{1,1}} .
$$
Using the Q-system for $A_1$ to eliminate $R_{1,k+2}=(R_{1,k+1}^2+1)/R_{1,k-1}$ and
$R_{1,k-2}=(R_{1,k+1}^2+1)/R_{1,k-1}$, we get the conservation law:
$$
c_1={R_{1,k}\over R_{1,k-1}}+{1\over R_{1,k-1}R_{1,k}} +{R_{1,k-1}\over R_{1,k}}=
{R_{1,1}\over R_{1,0}}+{1\over R_{1,0}R_{1,1}}+{R_{1,0}\over R_{1,1}} .
$$
This is a two-term recursion relation in $k$, whereas the $Q$-system is
a three-term recursion. The former is an explicit discrete ``first integral" of the latter.
\end{example}

Another way of understanding the conserved quantities of Theorem
\ref{conserved} is via the translational invariance of the cluster
algebra, expressed in Lemma \ref{clustinv}. We get the following
immediate
\begin{cor}
The quantities $c_i$, expressed in terms
of the seed $\bx_0$, are conserved, namely:
\begin{equation}\label{idenco}
c_i(R_{1,0},\ldots,R_{r,0};R_{1,1},\ldots,R_{r,1})=c_i(R_{1,k},\ldots,R_{r,k};R_{1,k+1},\ldots,R_{r,k+1})
\end{equation}
for all $k\in \Z$ and for $i=0,1,2,...,r$.
\end{cor}

\subsection{Recursion relations for discrete Wronskians with defects}

We now derive explicit relations between the $c_i$ and the initial
data. 
It is useful to work in the context of the $A_{\infty/2}$ $Q$-system, which is
obtained from the $A_r$ system by relaxing the boundary condition
$R_{r+1,n}=1$:
\begin{equation}\label{ainfty}
r_{\al,n+1}r_{\al,n-1}=r_{\al,n}^2+r_{\al+1,n}r_{\al-1,n}, \ \ \
r_{0,n}=1, n\in \Z, \al\geq 1.
\end{equation}
Again, as in Lemma \ref{Rwronsk}, we have
$$
r_{\al,n}:=\det_{1\leq i,j\leq \al}(r_{1,n+i+j-1-\al}). 
$$

Define the $\al \times \al$ Wronskians with a ``defect" at position
$\al-m$ ($0\leq m\leq \al$, $n\in \Z$):
\begin{equation}\label{wronskdefect}
c_{\al,m,n}=\det_{1\leq i,j\leq \al} \, r_{1,n+i+s_{\al-m}(j)-\al-1} , \qquad s_m(j)=\left\{ 
\begin{matrix}
j & {\rm if}\ j\leq m\\
j+1 & {\rm if}\ j>m
\end{matrix}
\right.
\end{equation}
where $c_{0,0,n}=1$ for all $n$. Note that $c_{\al,0,n}=r_{\al,n}$ and
$c_{\al,\al,n}=r_{\al,n+1}$. In addition, $c_i= c_{r+1,i,k}$ if
$r_{\al,n}=R_{\al,n}$, that is, when we impose the condition
$r_{r+1,n}=1$.

\begin{lemma}\label{recucj}
The $c_{\al,m,n}$ satisfy the following recursion relation:
\begin{equation}\label{cjrecu}
r_{\al-1,n}r_{\al-1,n+1}c_{\al,m,n}=r_{\al,n}r_{\al-1,n+1}c_{\al-1,m,n}+r_{\al-1,n}r_{\al,n+1}
c_{\al-1,m-1,n}+r_{\al,n}r_{\al,n+1}c_{\al-2,m-1,n}.
\end{equation}
\end{lemma}
\begin{proof}
Applying \eqref{dj} to the $\al\times \al$ matrix $M$ with entries
$M_{i,j}=r_{1,n+i+j-\al}$, $i,j\in\{1,\ldots, \alpha\}$, with $i_1=1$,
$i_2=\al$, $j_1=\al-m$ and $j_2=\al$,
\begin{equation}\label{simplif}
r_{\al,n+1}c_{\al-2,m-1,n}+r_{\al-1,n+1}c_{\al-1,m,n}=r_{\al-1,n}c_{\al-1,m,n+1}.
\end{equation}
Using this to simplify the sum of the first and last terms on the
r.h.s. of \eqref{cjrecu}, must prove that
\begin{equation}\label{onemore}
r_{\al-1,n+1}c_{\al,m,n}=r_{\al,n+1}c_{\al-1,m-1,n}+r_{\al,n}c_{\al-1,m,n+1}.
\end{equation}
Multiplying \eqref{onemore} by $r_{\al,n+1}$, and using
\eqref{ainfty}, 
\begin{eqnarray*}
&&r_{\al,n+1}r_{\al-1,n+1}c_{\al,m,n}-(r_{\al,n+2}r_{\al,n}-r_{\al+1,n+1}r_{\al-1,n+1})
c_{\al-1,m-1,n}-r_{\al,n+1}r_{\al,n}c_{\al-1,m,n+1}\\
&=&r_{\al-1,n+1}(r_{\al,n+1} c_{\al,m,n}+r_{\al+1,n+1}c_{\al-1,m-1,n})
-r_{\al,n}(r_{\al,n+2}c_{\al-1,m-1,n}+r_{\al,n+1}c_{\al-1,m,n+1})\\
&=&r_{\al,n}(r_{\al-1,n+1}c_{\al,m,n+1}-r_{\al,n+2}c_{\al-1,m-1,n}-r_{\al,n+1}c_{\al-1,m,n+1}) = 0,
\end{eqnarray*}
where we have used again \eqref{simplif}, with $\al\to\al+1$,
to simplify the second line.  

The last
equation follows from 
\eqref{specplu}, using the $\al\times \al+2$
matrix $T$ with entries $T_{i,1}=\delta_{i,\al}$ and
$T_{i,j}=r_{1,n+i+j-\al-1}$ ($2\leq j\leq \al+2$ and $1\leq i\leq
\al$), with $a_1=1$, $a_2=2$, $b_1=\al+2-m$ and
$b_2=\al+2$. Then eq.\eqref{specplu} becomes
$$
r_{\al,n+2}c_{\al-1,m-1,n}=c_{\al,m,n+1}r_{\al-1,n+1}-r_{\al,n+1}c_{\al-1,m,n+1}
$$
which is the desired relation. 
\end{proof}

Define
$$
v_{\al,n}={r_{\al,n}\over r_{\al-1,n}}, \ \ \al=1,2,... \ \ n\in \Z.
$$
Equation \eqref{cjrecu} can be written as
\begin{equation}\label{recastcj}
c_{\al,m,n}=v_{\al,n}c_{\al-1,m,n}+v_{\al,n+1}
c_{\al-1,m-1,n}+v_{\al,n}v_{\al,n+1}c_{\al-2,m-1,n}.
\end{equation}
Together with the initial conditions 
$c_{0,0,n}=1$ and $c_{1,0,n}=v_{1,n}$, $c_{1,1,n}=v_{1,n+1}$,  
\eqref{recastcj} gives $c_{\al,m,n}$ as a polynomial in the variables
$\{v_{k,n},v_{k,n+1}\}_{1\leq k\leq \al}$, of total degree $\al$. 
In particular, we have
\begin{eqnarray}\label{bdryvalc}
c_{\al,0,n}=v_{1,n}v_{2,n} \cdots v_{\al,n},\quad
c_{\al,\al,n}=v_{1,n+1}v_{2,n+1} \cdots v_{\al,n+1}.
\end{eqnarray}

Next we introduce the quantities which we call {\em weights}, for
reasons which will become clear below:
\begin{equation}\label{qty}
y_{2\al-1,n}={v_{\al,n+1}\over v_{\al,n}}={r_{\al-1,n}r_{\al,n+1}\over r_{\al,n}r_{\al-1,n+1}}, 
\qquad  y_{2\al,n}={v_{\al+1,n+1}\over
v_{\al,n}}={r_{\al-1,n}r_{\al+1,n+1}\over r_{\al,n}r_{\al,n+1}},\quad
\al\geq 1.
\end{equation}
We define
\begin{equation}\label{renormC}
C_{\al,m,n}={c_{\al,m,n}\over v_{1,n}v_{2,n}\cdots v_{\al,n}}.
\end{equation}
Then \eqref{recastcj} becomes
\begin{thm}\label{diswronrec}
The quantities $C_{\al,m,n}$ of \eqref{renormC} satisfy
\begin{equation}\label{recastCj}
C_{\al,m,n}=C_{\al-1,m,n}+y_{2\al-1,n}C_{\al-1,m-1,n}+y_{2\al-2,n}C_{\al-2,m-1,n},
\end{equation}
with the $y$'s as in eq.\eqref{qty}.
\end{thm}

Together with the initial condition $C_{\al,0,n}=1$, this gives
$\{C_{\al,m,n}\}$ 
as polynomials of homogeneous degree $m$ in $y_{k,n}$'s, with $1\leq
k\leq 2\al-1$. 
\begin{example}
The first few $C$'s for $\al=0,1,2,3$ read:
\begin{eqnarray*}
C_{0,0,n}&=&1,\ \
C_{1,0,n}=1,\ \
C_{1,1,n}=y_{1,n},\\
C_{2,0,n}&=&1, \ \
C_{2,1,n}=y_{1,n}+y_{2,n}+y_{3,n}, \ \
C_{2,2,n}=y_{1,n}y_{3,n}, \\
C_{3,0,n}&=&1,\ \ C_{3,1,n}=y_{1,n}+y_{2,n}+y_{3,n}+y_{4,n}+y_{5,n} ,\\
C_{3,2,n}&=&y_{1,n}y_{3,n}+y_{1,n}y_{4,n}+y_{1,n}y_{5,n}+y_{2,n}y_{5,n}+y_{3,n}y_{5,n},\ \
C_{3,3,n}=y_{1,n}y_{3,n}y_{5,n} .
\end{eqnarray*}
\end{example}

We apply the above results to the conserved quantities of the $Q$-system
of Theorem \ref{conserved}. We identify $R_{\al,n}=r_{\al,n}$ by
imposing the boundary condition
$r_{r+1,n}=1$ for all $n$. Then $v_{r+1,n}=1/r_{r,n}$ and
\begin{equation}\label{unitcond}
v_{1,n}v_{2,n}... v_{r+1,n} =1.
\end{equation}
Therefore,
\begin{equation}\label{identic}
c_i=c_{r+1,i,k}=C_{r+1,i,k}=C_{r+1,i,0} .
\end{equation}
In particular, we recover $c_0=c_{r+1}=1$ from the explicit
expressions for $c_{r+1,0,n}$ and $c_{r+1,r+1,n}$ of
Equation \eqref{bdryvalc}, together with
\eqref{unitcond}.  We note that $C_{r+1,i,k}=C_{r+1,i,0}$ are
independent of $k$ for all $j=0,1,...,r+1$, in other words we have the
conservation laws: $c_j(\{y_{\al,k}\})=c_j(\{y_{\al,0}\})$ for all
$k\in\Z$.

The identification \eqref{identic} gives an expression for $c_i$ in
terms of the initial data $\bx_0$:  By iterative use of the recursion
relations of Theorem \ref{diswronrec} for $n=0$ and $1\leq \al\leq
r+1$, we get expressions for $C_{r+1,m,0}$ as polynomials of homogeneous
degree $m$ in the weights $y_{\al,0}$, $1\leq \al\leq 2r+1$. 
These involve only the entries of $\bx_0$.

\subsection{Conserved quantities as hard particle partition functions}

The recursion relations of the previous section lead directly to an
interpretation of the quantities $c_i$ as partition functions of hard
particles on a graph, with weights which depend only on $\bx_0$.

\begin{figure}
\centering
\includegraphics[width=10.cm]{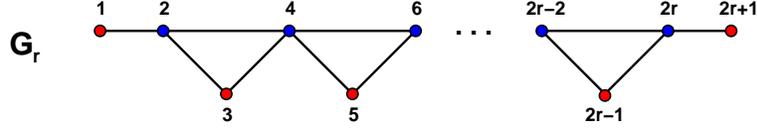}
\caption{\small The graph $G_r$,
with $2r+1$ vertices labeled $i=1,2,...,2r+1$.}
\label{fig:graphgr}
\end{figure}

Let $G_r$ be the graph of Figure \ref{fig:graphgr}.  When $r=1$, $G_1$
is the chain with 3 vertices.  To each vertex labeled $i$ in the
graph, we assign the weight $y_i$.

Let $G$ be a graph with vertices labeled by the index set $I$.

\begin{defn}\label{harpatic}
A hard particle configuration $C$ on $G$ is a subset of $I$ such that
$i,j\in C$ only if there is no edge connecting vertices $i$ and $j$ in
$G$.
\end{defn}

Let ${\rm HP}(G)$ be the set of all hard particle configurations on $G$.

If we assign a weight $y_i$ to each vertex $i\in I$, the {\em weight}
of a configuration $C$ is
$w(C)=\prod_{i\in C} y_i$.
The partition function of hard particles on $G$ is
\begin{equation}
Z^{(G)}(\{y_i\})=\sum_{C\in {\rm HP}(G)} w(C) .
\end{equation}
If we limit the summation to the set of configurations fixed
cardinality $j$, we have the $j$-particle partition function
$Z_j^{(G)}$.

In the particular case $G=G_r$ of Figure \ref{fig:graphgr}, we have
the partition function of $j$ hard particles on $G_r$, denoted by
$Z_j^{(G_r)}:=Z_j^{(G_r)}(y_1,...,y_{2r+1})$.
These satisfy recursion relations, coming from the structure of $G_r$.

\begin{thm}\label{recurelahpgr}
The partition functions $Z_j^{(G_r)}$ satisfy the following recursion
relations:
\begin{equation}\label{recuhardpart}
Z_j^{(G_r)}=Z_j^{(G_{r-1})} +y_{2r+1} Z_{j-1}^{(G_{r-1})} + y_{2r} Z_{j-1}^{(G_{r-2})}
\end{equation}
\end{thm}
\begin{proof}
Depending on the occupation numbers of the last two vertices, three
situations may occur:
\begin{enumerate}
\item Vertices $2r$ and $2r+1$ are both vacant. This is a
configuration of $j$ hard particles
on $G_{r-1}$, obtained by erasing these two vertices and their adjacent edges.

\item Vertex $2r+1$ is occupied, and hence the vertex $2r$ is empty. Such configurations
are those of $j-1$ hard particles on $G_{r-1}$.

\item The vertex $2r$ is occupied, and hence  vertices $2r+1$,
$2r-1$ and $2r-2$ are empty. Such configurations are those of $j-1$
hard particles on $G_{r-2}$ obtained by erasing the vertices
$2r-2,2r-1,2r,2r+1$ and their incident edges.
\end{enumerate}
Each of these occupation states gives rise to one of the terms on the
right hand side of the equation.
\end{proof}

\begin{figure}
\centering
\includegraphics[width=14.cm]{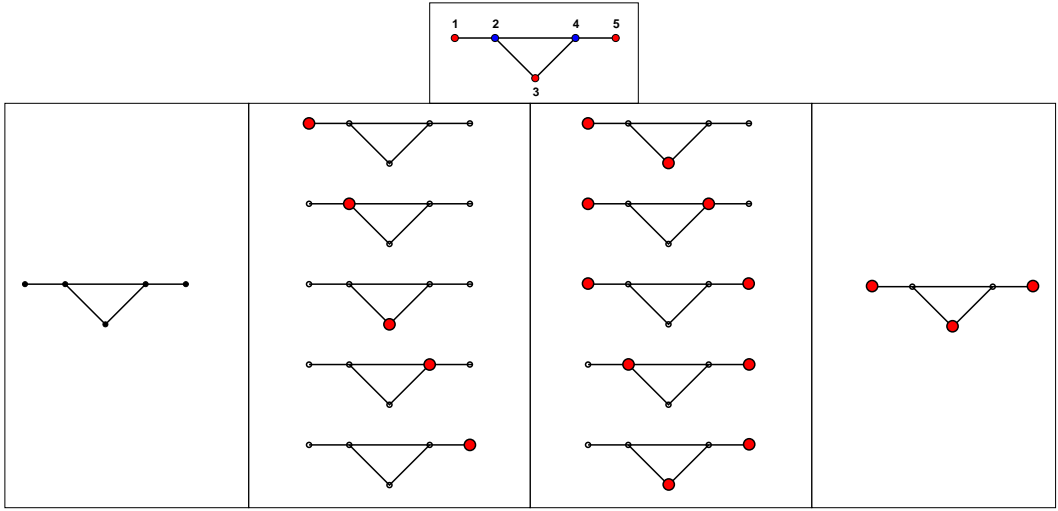}
\caption{\small The configurations of hard particles on $G_2$, with, from left to right, 
$0$, $1$, $2$ or $3$ particles. }
\label{fig:confatwo}
\end{figure}

\begin{example}\label{exgtwo}
For $r=2$, the hard particle model on $G_2$
has the partition function
\begin{equation}\label{hardgtwo}
Z^{(G_2)}(y_1,y_2,y_3,y_4,y_5)=1+(y_1+y_2+y_3+y_4+y_5)
+(y_1 y_3+y_1y_4+y_2y_5+y_3y_5+y_1y_5)+(y_1y_3y_5) ,
\end{equation}
where the various terms correspond to the configurations depicted in Fig.\ref{fig:confatwo}.
\end{example}

\begin{thm}\label{conshp}
The $j$-th conserved quantity $c_j$ of the Q-system for $A_r$ is equal
to the partition function
$Z_j^{(G_r)}(y_{1,n},...,y_{2r+1,n})$ for
$j$ hard particles on the graph $G_r$, with vertex weights $y_i\equiv y_{i,n}$ defined in 
eq.\eqref{qty}, for $i=1,2,...,2r+1$, and for any choice of $n\in \Z$.
\end{thm}
\begin{proof}
The equations \eqref{recastCj} and \eqref{recuhardpart} are identical upon taking
$y_i=y_{i,n}$. Moreover the initial conditions are also identical,
as $Z_0^{(G_r)}=C_{r+1,0,n}=1$. We deduce that $Z_j^{(G_r)}=C_{r+1,j,n}=c_j$,
independently of $n$.
\end{proof}

\begin{cor}\label{hpcor}
The conserved quantities $c_j$ can be expressed in terms $\bx_0$
as the partition functions for $j$ hard
particles on $G_r$, with vertex weights
\begin{equation}\label{qtyzero}
y_{2\al-1}= {R_{\al-1,0}R_{\al,1}\over R_{\al,0}R_{\al-1,1}},\ \ 1\leq \al\leq r+1, \qquad 
y_{2\al}={R_{\al-1,0}R_{\al+1,1}\over R_{\al,0}R_{\al,1}}, \ \ 1\leq
\al\leq r. 
\end{equation}
\end{cor}

\begin{example}
In the case $G_2$ 
of Example \ref{exgtwo}, we have 
(with $y_1y_3y_5=1$): $c_0=c_3=1$, and
\begin{eqnarray}
c_1&=&{R_{2,0}\over R_{2,1}}+{R_{1,1}\over R_{1,0}}+{R_{1,0}R_{2,1}\over R_{2,0}R_{1,1}}
+{R_{2,1}\over R_{1,0}R_{1,1}}+{R_{1,0}\over R_{2,0}R_{2,1}}\label{coneatwo}\\
c_2&=&{R_{1,0}\over R_{1,1}}+{R_{2,1}\over R_{2,0}}+{R_{2,0}R_{1,1}\over R_{1,0}R_{2,1}}
+{R_{1,1}\over R_{2,0}R_{2,1}}+{R_{2,0}\over R_{1,0}R_{1,1}}\label{ctwoatwo}
\end{eqnarray}
The two integrals of motion of the $A_2$ $Q$-system correspond to writing $c_1$ and $c_2$
with the substitutions $R_{\al,i}\to R_{\al,i+n-1}$, $i=0,1$, $\al=1,2$. These yield a system of recursion
relations involving only indices $n$ and $n-1$, as opposed to the original $Q$-system, which involves
the indices $n-1,n$ and $n+1$.
\end{example}

\subsection{Generating functions}

\subsubsection{A generating function for $R_{1,n}$}

It is useful to introduce generating functions. Define
\begin{equation}
F_1^{(r)}(t)=\sum_{n\geq 0} R_{1,n} t^n 
\end{equation}
\begin{thm}\label{expressratio}
We have the relation
\begin{equation}\label{genA}
F_1^{(r)}(t)= {\sum_{j=0}^r (-1)^j d_j t^j \over 
\sum_{j=0}^{r+1} (-1)^j c_j t^j } , 
\qquad d_j=\sum_{i=0}^j R_{1,j-i} (-1)^{j-i} c_i.
\end{equation}
\begin{proof}
Consider the product of series 
$\Big(\sum_{i=0}^\infty R_{1,i} t^i \Big)\Big( \sum_{j=0}^{r+1} (-1)^j c_j t^j \Big)$.
Then all terms of order $r+1$ or higher in $t$ vanish, due to Theorem \ref{alphaone}. We are
left with the terms of order $0,1,...,r$, the term of order $j$ being exactly $(-1)^j d_j$.
\end{proof}
\end{thm}

\begin{example}
For $r=1$, we have $c_1={R_{1,1}\over R_{1,0}}+{1\over R_{1,0}R_{1,1}}+{R_{1,0}\over R_{1,1}}$
from Example \ref{aonecval}, and $d_0=R_{1,0}$, 
$d_1=c_1 R_{1,0}-R_{1,1}={R_{1,0}^2+1\over R_{1,1}}$, hence
\begin{equation}\label{sltwof}
F_1^{(1)}(t)=R_{1,0}{1 -\left({R_{1,0}\over R_{1,1}}+{1\over R_{1,0}R_{1,1}}\right)t \over 
1-\left({R_{1,1}\over R_{1,0}}+{1\over R_{1,0}R_{1,1}}+{R_{1,0}\over R_{1,1}}\right) t +t^2 }
\end{equation}
\end{example}

\subsubsection{Generating function and hard particles}

\begin{thm}\label{fratio}
\begin{equation}\label{ratioF}
F_1^{(r)}(t)=R_{1,0} {Z^{(G_r)}(0,-ty_2,-ty_3,...,-ty_{2r+1})\over 
Z^{(G_r)}(-ty_1,-ty_2,-ty_3,...,-ty_{2r+1})}
\end{equation}
with $y_i$ as in \eqref{qtyzero}.
\end{thm}
\begin{proof}
In the expression \eqref{genA}, the denominator is the partition
function $Z^{(G_r)}(-ty_1,...,-ty_{2r+1})$, according to Corollary
\ref{hpcor}.
The numerator of $F_1^{(r)}(t)/R_{1,0}$ is  $\sum_{j=0}^{r} (-t)^j d_j/R_{1,0}$,
where
$$
{d_j\over R_{1,0}}=\sum_{i=0}^j (-1)^{j-i}c_i {R_{1,j-i}\over R_{1,0}} .
$$
We proceed as for the $c_j$. First, we relax the condition that
 $R_{r+1,n}=1$, hence work with
the $r_{\al,n}$, the solutions of \eqref{ainfty}. 
Define
\begin{equation}\label{defD}
D_{\al,m,n}=\sum_{i=0}^m (-1)^{m-i}C_{\al,i,n} {r_{1,m-i}\over r_{1,0}} .
\end{equation}
Then $d_j/r_{1,0}=D_{r+1,j,n}=D_{r+1,j,0}$,
independently of $n$ when we impose the 
condition $r_{r+1,n}=1$, due to \eqref{identic}.
Substituting the recursion relations \eqref{recastCj} into this expression,
we obtain an analogous recursion relation for $D_{\al,m,0}$:
\begin{eqnarray*}
D_{\al,m,0}&=&\sum_{i=0}^m (-1)^{m-i} {r_{1,m-i}\over r_{1,0}}(y_{2\al-1}C_{\al-1,i-1,0}
+y_{2\al-2}C_{\al-2,i-1,0}+C_{\al-1,i,0}) \\
&=&y_{2\al-1}D_{\al-1,m-1,0}+y_{2\al-2}D_{\al-2,m-1,0}+D_{\al-1,m,0},
\end{eqnarray*}
with $C_{\al,-1,0}=0$ and $y_k:=y_{k,0}$.  The initial values of $D$
for $\al=1$ are $D_{1,0,0}=C_{1,0,0}=1$ and
$D_{1,1,0}=C_{1,1,0}-y_{1}C_{1,0,0}=y_{2}+y_{3}$. Both coincide with
the values of $C_{1,0,0}$ and $C_{1,1,0}$, respectively, when
restricted to $y_{1}=0$.  As the recursion relation for $D_{\al,m,0}$
is identical to that for $C_{\al,m,0}\vert_{y_{1}=0}$, we deduce that
$D_{\al,m,0}=C_{\al,m,0}\vert_{y_{1}=0}$ for all $\al,m$. This
relation remains true after imposing the condition \eqref{unitcond}.
Therefore
$$
d_j=R_{1,0} \, D_{r+1,j,0}=R_{1,0}\, C_{r+1,j,0}(0,y_2,...,y_{2r+1})
$$
with the $y$'s as in Corollary \ref{hpcor}. We deduce that the
numerator of $F_1^{(r)}(t)/R_{1,0}$ is equal to the denominator of
\eqref{genA}, restricted to the value $y_1=0$.
\end{proof}

\subsubsection{Translational invariance} From the translational invariance property of Lemma \ref{clustinv},  
we may easily deduce an invariance property for the generating
function $F_1^{(r)}(t)$.  Let us first write $F_1^{(r)}(t)$ as an
explicit expression
$F_1^{(r)}(t)=\Phi((R_{\al,0})_{\al=1}^r;(R_{\al,1})_{\al=1}^r;t)$
involving only the initial data $\bx_0$.

\begin{thm}\label{transtheo}
The generating function $\Phi$ satisfies the following translation property:
$$
\Phi(R_{1,0},\ldots,R_{r,0};R_{1,1},\ldots,R_{r,1};t)=\sum_{n=0}^{k-1} R_{1,n}t^n 
+t^k \Phi(R_{1,k},\ldots,R_{r,k};R_{1,k+1},\ldots,R_{r,k+1};t)
$$
for all $k\geq 0$.
\end{thm}
\begin{proof}
We write $F_1^{(r)}(t)=\sum_{n=0}^{k-1}R_{1,n}t^n +t^k \sum_{n\geq 0} R_{1,n+k}t^n$,
and apply eq.\eqref{invtrans} to all $R_{1,n+k}$ in the second term.
\end{proof}

\section{Positivity: a heap interpretation}\label{posit}
We now have an expression for the generating function of $R_{1,n}$
($n\geq 0$) in terms of the ratio of two partition functions of hard
particles.  Positivity of the terms $R_{1,n}$, when expressed in terms
of the fundamental cluster variables $\bx_0$, follows from a theorem
relating this ratio to the partition function of heaps. This
interpretation also allows us to find an explicit formula for the
generating function of cluster variables.

\subsection{Heaps}\label{heapsec}

\begin{figure}
\centering
\includegraphics[width=8.cm]{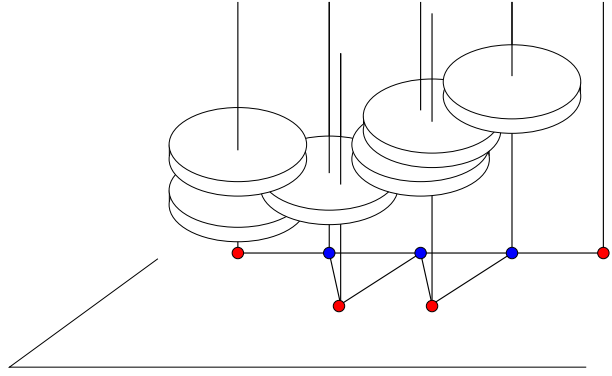}
\caption{\small A heap on the graph $G_3$. Solid discs are piled up above each vertex of $G_3$.
Diameters are such that only adjacent vertex discs may overlap.}\label{fig:heap}
\end{figure}

Given a graph $G$, heaps on $G$ are defined as follows
 (see \cite{HEAPS1} and the beautiful expository article
\cite{HEAPS2}).  

The graph $G$ is represented in the $xy$ plane in $\R^3$, and we attach
half-lines parallel to the positive $z$-axis, originating at each
vertex of $G$.  (See Figure \ref{fig:heap}). The
vertices of $G$ are endowed with a partial ordering.

On each half-line above vertex $i$, we stack an arbitrary number of
discs of radius $R_i$ and thickness $a$, with a hole at their center,
so that they can freely slide along the half-lines (gravity points
in the negative $z$-direction). The disc radii $R_i$ are such that the
distance between any pair of adjacent vertices $i,j$ of $G$ is
$<R_i+R_j$, while the distance between any pair of non-adjacent
vertices $i,j$ is $>R_i+R_j$. Thus, the order in which the various
discs are stacked matters only on neighboring half-lines (i.e. with
connected projections on $G$), but not on distant ones.

For a given stack of discs, its {\it foreground} is
the set of discs that touch the $xy$-plane.  A stack is
said to be {\em admissible} if it is empty or if its foreground is reduced
to one disc, positioned at a vertex of smallest order. (Such
configurations are also called ``pyramids" in the heap jargon
\cite{HEAPS1})
We will call such admissible stacks of discs {\em heaps} on
$G$. 

To each heap $h$ on $G$, we associate a weight $W(h)=\prod_{d\in h}
W(d)$, where the product extends over all discs $d$ of the heap, and
where the weight $W(d)=z_i$ if $d$ is stacked above vertex $i$. The
partition function for heaps on $G$ is
\begin{equation}
\Phi^{(G)}(\{z_i\})=\sum_{h\ {\rm heap}\ {\rm on}\ G} W(h) 
\end{equation}

\begin{thm}\label{heapsforar}
\begin{equation}\label{ratiophi}
\Phi^{(G)}(\{z_i\})={Z^{(G)}(0,-z_2,-z_3,\ldots )\over Z^{(G)}(-z_1,-z_2,-z_3,\ldots )}
\end{equation}
\end{thm}
\begin{proof}
This follows from the general theory of heaps \cite{HEAPS1} \cite{HEAPS2}.
We write
\begin{eqnarray}\qquad
\Phi^{(G)}(\{z_i\}) \ Z^{(G)}(-z_1,-z_2,-z_3,\ldots )=\sum_{(h,c)} W(h)\,w(c)=Z^{(G)}(0,-z_2,-z_3,\ldots )
=\sum_{c' \ {\rm with}\atop {\rm vertex}\ 1 \ {\rm empty} }
w(c')\label{incluexclu}
\end{eqnarray}
where the sums extend over pairs $(h,c)$ made of a heap $h$ and 
a configuration $c$ of hard particles on $G$,
and configurations $c'$ of hard 
particles on $G$ such that the vertex 1
remains unoccupied.

We define an involution $\varphi$
between pairs $(h,c)$ which reverses the sign of $W(h)\,w(c)$. Let
$c\circ h$ be the heap obtained by replacing the particles of $c$ by
discs, and by adding those discs on top of $h$.  We define the
background of any heap $h$ to be the foreground of the heap obtained
by flipping the configuration upside-down. In particular, the
background of $c\circ h$ contains $c$.  Let $d$ be the disc in the
background of $c\circ h$ that has the smallest vertex index $i$.  Then
if $i$ is occupied in $c$, we form $c'$ by removing the particle at
$i$, and $h'$ by adding a disc on top of the $i$ vertex. If $i$ is
unoccupied in $c$, then we form $c'$ by adding a particle at $i$, and
$h'$ by removing the top disc at $i$.  Finally if $h$ is empty and the
vertex $1$ is unoccupied in $c$, then we leave the pair
unchanged. These rules define an involution $\varphi(h,c)=(h',c')$. By
construction, we have $W(h)\,w(c)=-W(h')\,w(c')$, if $(h',c')\neq
(h,c)$.  Hence the distinct pairs $((h,c),(h',c'))$ image of
one-another under $\varphi$ contribute zero to the sum on the
l.h.s. of eq. \eqref{incluexclu}, and we are left only with the
contribution of the fixed points of the involution.  The latter
correspond to the situation where $h=\emptyset$ and $c$ has no
particle at the vertex $1$, producing the r.h.s. of
\eqref{incluexclu}.
\end{proof}

\subsection{Positivity from heaps}

Applying Theorem \ref{heapsforar} to the case of $G=G_r$, 
and comparing the expressions \eqref{ratioF} and \eqref{ratiophi},
we arrive at the main theorem of this section, 
which allows to interpret the $R_{1,n}$ as partition functions
for heaps.

\begin{thm}\label{diskette}
The solution $R_{1,n}$ to the Q-system for $n\geq 0$ is, 
up to a multiplicative factor $R_{1,0}$,
the partition function for configurations
of heaps of $n$ discs on $G_r$ with weights $z_i=t y_i$, with $y_i$ as in \eqref{qtyzero}. 
\end{thm}
\begin{proof}
Using Theorem
\ref{fratio},
we may rewrite $F_1^{(r)}(t)/R_{1,0}$ exactly in the form of the r.h.s.
of \eqref{ratiophi},
with $G=G_r$ and the weights $z_i=t y_i$.
We deduce that $F_1^{(r)}(t)/R_{1,0}$
is the generating function for heaps on $G_r$ with weight $t y_i$ per disc above the vertex $i$.
The coefficient of $t^n$ in the corresponding series corresponds to heap 
configurations with exactly $n$ discs.
\end{proof}

\begin{cor}\label{positrone}
For all $n\in\Z$,
$R_{1,n}$ is a positive Laurent polynomial of the initial seed $\bx_0$.
\end{cor}
\begin{proof} 
By Theorem \ref{diskette}, $R_{1,n}/R_{1,0}$ for $n\geq 0$ is a sum over heap configurations, each 
contributing a weight made of a product of $y_i$'s. The resulting sum is therefore a
manifestly positive polynomial of the $y_i$'s, themselves products of ratios of initial data.
For $n<0$, the result follows by applying the Lemma \ref{firstrem}.
\end{proof}

\subsection{Continued fraction expressions for generating functions}

The heap interpretation allows us to write an explicit expression for
$F_1^{(r)}(t)$ as a rational  function of $t$.

\begin{figure}
\centering
\includegraphics[width=12.cm]{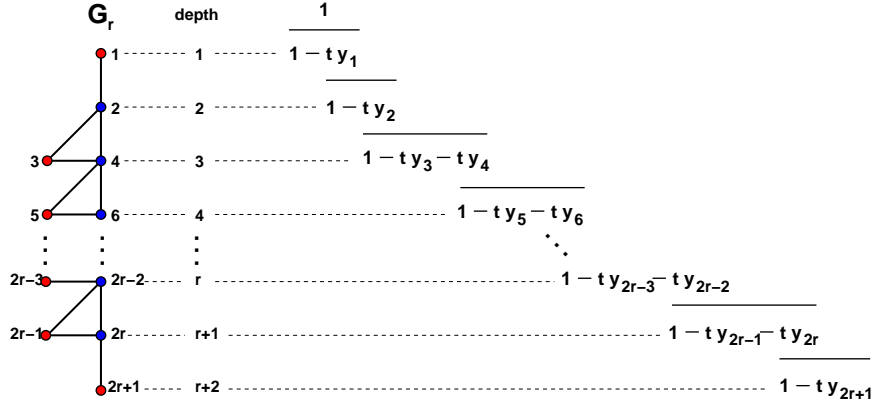}
\caption{\small The graph $G_r$, arranged into a hierarchical structure. The depth of the vertices
is indicated on the center, and the corresponding terms of the continued fraction
for the generating function $F_1^{(r)}(t)$ are displayed on the right.}\label{fig:depth}
\end{figure}

\begin{thm}\label{posita}
\begin{equation}\label{contia}
F_1^{(r)}(t)={R_{1,0}\over 1-t{y_1\over 1-t {y_2\over 1-t y_3
-t {y_4\over 1-t y_5-t {y_6\over \ddots {{}\over 1-ty_{2r+1}}}}}}}
\end{equation}
where $y_i$ are defined in \eqref{qtyzero}.
\end{thm}
\begin{proof}
Define a partial ordering on the vertices of $G_r$ by their geodesic
distance from vertex 1.  Any nonempty heap on $G_r$ is constructed by
repeating the following two steps:
\begin{enumerate}
\item Stack one disc above vertex $1$.
\item Construct a heap on the graph $G_r'=G_r\setminus \{1\}$ (the
graph $G_r$ without the vertex $1$ and its incident edges).
\end{enumerate}
Let $c(t)$ be the generating function for heaps on
$G_r'$, then clearly
$
{F_1^{(r)}(t)/ R_{1,0}}={1/( 1-t {y_1 c(t)})}.
$
Similarly, $c(t)$ is found via a similar construction of heaps on
$G_r'$. Clearly, $c(t)=1/(1-t y_2 d(t))$, where $d(t)$ is the
generating function for heaps on the graph $G_r''=G_r'\setminus \{2\}$.

To find $d(t)$, we note that $G_r''$ has two minimal vertices, but
they are connected, so that a heap on $G_r''$ may have as its
foreground either node $3$ or node $4$ but not both. Thus,
$d(t)={1/( 1-t y_3-t y_4 e(t))}$, where $e(t)$ is the generating
function for heaps on $G_r''\setminus\{3,4\}$. This procedure is
iterated (see Figure \ref{fig:depth}), resulting in Equation
\eqref{contia}.
\end{proof}

When expanded as a power series in $t$, \eqref{contia} has manifestly positive
coefficients which are Laurent polynomials in the $y_i$'s.

\subsection{Continued fraction rearrangements}
One can rewrite the continued fraction expression for
 $F_1^{(r)}(t)$ in various ways using two simple Lemmas:
\begin{lemma}\label{firstmov}
For all $a$, $b$ there is an identity of power series of $t$:
\begin{equation}
{1\over 1-t {a\over 1-t b}}=1+t{a\over 1-t a -t b}.
\end{equation}
\end{lemma}
and
\begin{lemma}\label{secmov}
For all $a$, $b$, $c$, $d$
such that $c\neq 0$ and $a+b\neq 0$,
\begin{equation}
a+{b\over 1-t {c\over 1- t d}}={a'\over 1-t{b' \over 1-t c'-t d}}
\end{equation}
where
\begin{equation}\label{transforat}
a'=a+b,\quad  b'={b c\over a+b},\quad   c'={a c\over a+b},\quad
a={a' c'\over b'+c'},\quad  b={a' b'\over b'+c'},\quad  c=b'+c'
\end{equation}
\end{lemma}

For example, applying Lemma \ref{firstmov} to the expression \eqref{contia}
with $a=y_1$ and $b=y_2/(1-t y_3\cdots )$, we have
\begin{equation}\label{positA}
F_1^{(r)}(t)=R_{1,0}\Big( 1 + t{y_1 \over 1- t y_1 -{t y_2\over 1-t y_3
-{t y_4 \over \ddots {{} \over 1-t y_{2r-1}-{t y_{2r}\over 1-t
      y_{2r+1} }}}}}\quad \Big)
\end{equation}
where $y_i$ are defined in \eqref{qtyzero}.

\begin{example}
For $r=1$, we have 
\begin{equation}\label{aoneflat}
F_1^{(1)}(t)=R_{1,0}\Big( 1 +{t{R_{1,1}\over R_{1,0}}\over 
1- t{{R_{1,1}\over R_{1,0}} - {{t\over R_{1,0}R_{1,1}}\over 1-t{R_{1,0}\over R_{1,1}} }}} \quad\Big)
={R_{1,0}\over 1-t {{R_{1,1}\over R_{1,0}}\over 1-t {{1\over R_{1,0}R_{1,1}}\over 1-t {R_{1,0}\over R_{1,1}}}}}
\end{equation}
\end{example}

\begin{figure}
\centering
\includegraphics[width=6.cm]{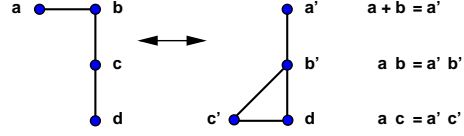}
\caption{\small The local transformation of Lemma \ref{secmov}.}\label{fig:secmov}
\end{figure}
Figure \ref{fig:secmov} is a graphical interpretation of Lemma
\ref{secmov}. It shows that the generating function of
heaps on $G$ can be rewritten as a generating function for heaps on
$G'$ with new weights. Note that $a,b,c,d$ and the associated nodes
might stand for composite generating functions and the associated
subgraphs.

Starting from the initial seed $\bx_0$ and its associated graph $G_r$,
repeated application of Lemmas
\ref{firstmov} and \ref{secmov} on the
expression \eqref{positA} produces new expressions for $F_1^{(r)}(t)$
with manifestly positive series expansions in $t$. 

Writing the expression in \eqref{positA} as
$F_1^{(r)}(t)\equiv F_{\bM_0}(\bx_0;t)$,
let $\bx_M$ be another seed in the fundamental domain. Then
we claim that there is a sequence
$\Lambda$ of applications Lemmas \ref{firstmov} and
\ref{secmov}, such that $\Lambda( F_1^{(r)}(t)) = F_\bM(\bx_\bM)$.
It turns out that we can generate in this way all the mutations of
$\bx_0$ within the fundamental domain ${\mathcal F}_r$. The proof of
this statement appears in Section \ref{pathinter}.  Here, we illustrate this
result with the
example of $A_2$.  The example of $A_3$ is given in Appendix
\ref{appendixa}.



\begin{example}\label{sectatwo}
Consider the case of $A_2$. We present all 
the rearrangements of $F_1^{(2)}(t)$ which
correspond to mutations of the weights within the fundamental domain.

\begin{figure}
\centering
\includegraphics[width=12.cm]{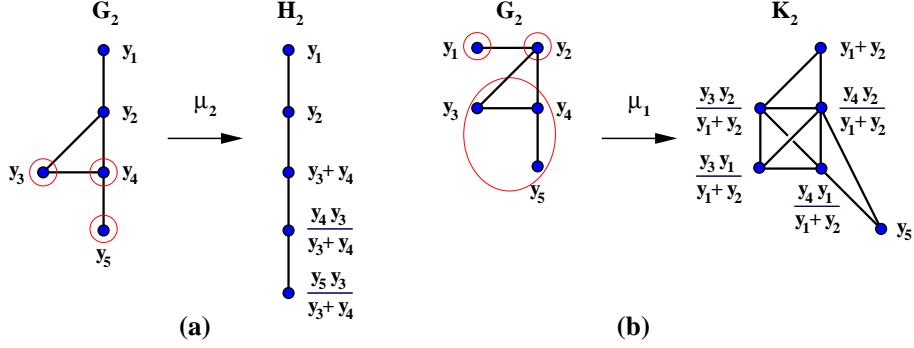}
\caption{\small The graph heap formulation of the application of Lemma \ref{secmov} 
on the generating function $F_1^{(r)}(t)/R_{1,0}$ (a), and on the generating function $\Phi(t)$ such that 
$F_1^{(r)}(t)=R_{1,0}+t R_{1,1} \Phi(t)$ (b). The circled
vertices correspond to $a,b,c$ in the Lemma \ref{secmov}. 
}\label{fig:ktwo}
\end{figure}

\begin{enumerate}
\item $F_{\bM_0}(\bx_0)$ with $\bx_0=(R_{1,0},R_{2,0};R_{1,1}R_{2,1})$:
Equation \ref{positA} is
\begin{equation}\label{exponeAtwo}
F_1^{(2)}(t)=R_{1,0} +t {R_{1,1}\over 1-t y_1 -{t y_2 \over 
1-t y_3 -
{t y_4\over 1-t y_5}}}
\end{equation}
with
\begin{equation}\label{seedone}
y_1= {R_{1,1}\over R_{1,0}},\ \ y_2={R_{2,1}\over R_{1,0}R_{1,1}},\ \ 
y_3={R_{1,0}R_{2,1}\over R_{2,0}R_{1,1}}, \ \ y_4={R_{1,0}\over R_{2,0}R_{2,1}} ,\ \ 
y_5={R_{2,0}\over R_{2,1}}.
\end{equation}

\item $F_{\mu_2(\bM_0)}(\mu_2(\bx_0))$, where
 $\mu_2(\bx_0)=(R_{1,0},R_{2,2};R_{1,1}R_{2,1})$: Use Lemma
\ref{secmov} with $a=y_3$,
$b=y_4$, $c=y_5$ and
$d=0$. Then
\begin{equation}\label{exptwoAtwo}
F_1^{(2)}(t)=R_{1,0} +t {R_{1,1}\over 1-t x_1 -{t  x_2\over 
1-{t x_3 \over 
1-{t x_4\over 1-t x_5}}}}={R_{1,0} \over 1-{t x_1 \over 1-{t x_2 \over 
1-{t x_3\over 
1-{t x_4\over 1-t x_5}}}}}
\end{equation}
where the second expression follows from application of Lemma
\ref{firstmov} with $a=x_1$. Here
\begin{eqnarray}
x_1&=& y_1= {R_{1,1}\over R_{1,0}},\quad
x_2= y_2={R_{2,1}\over R_{1,0}R_{1,1}}\nonumber \\
x_3&=&a'=a+b=y_3+y_4=
{R_{1,0}(R_{2,1}^2+R_{1,1})\over R_{2,1}R_{1,1}R_{2,0}}={R_{1,0}R_{2,2}\over R_{2,1}R_{1,1}}\nonumber \\
x_4&=&b'={b c\over a'}={y_4y_5\over y_3+y_4}=
{R_{1,0}\over R_{2,1}^2 a'}={R_{1,1}\over R_{2,1}R_{2,2}}\nonumber \\
x_5&=&c'={a c\over a'}={y_3y_5\over y_3+y_4}={R_{1,0}\over R_{1,1} a'}={R_{2,1}\over R_{2,2}} \label{seedtwo}.
\end{eqnarray}
by use of the $Q$-system.
This is an expression for the generating function in terms of
$(R_{1,0},R_{2,2};R_{1,1}R_{2,1})=\mu_2 (\bx_0)$. It has a manifestly
positive series expansion in $t$.

Figure \ref{fig:ktwo} (a) illustrates this transformation.
The
generating function is interpreted in terms of that for heaps
on $H_2$:
$$
F_1^{(2)}(t)=R_{1,0}  {Z^{(H_2)}(0,-tx_2,-tx_3,-tx_4,-tx_5)\over 
Z^{(H_2)}(-tx_1,-tx_2,-tx_3,-tx_4,-tx_5)}
$$

\item $F_{\mu_1(\bM_0)}(\mu_1(\bm_0))$, where
 $\mu_1(\bx_0)=(R_{1,2},R_{2,0};R_{1,1}R_{2,1})$:
Use Lemma \ref{secmov} on \eqref{exponeAtwo}, with 
$a=y_1$,  $b=y_2$, 
$c=y_3+y_4/(1-t y_5)={R_{1,0}\over R_{2,0}R_{2,1}R_{1,1}} \Big(R_{2,1}^2+
R_{1,1}/(1-t{R_{2,0}\over R_{2,1}})\Big)$, 
and $d=0$ (see  Fig.\ref{fig:ktwo} (b)):
\begin{equation}\label{expfourAtwo}
F_1^{(2)}(t)=R_{1,0}+t {R_{1,1} \over 1-{t z_1 \over 1-t {z_2+{z_6\over 1-t z_5}
\over 1- tz_3-t{z_4\over 1-t z_5}}}}
\end{equation}
with
\begin{eqnarray}
z_1&=&a'=a+b=y_1+y_2={R_{1,1}^2+R_{2,1}\over R_{1,0}R_{1,1}}={R_{1,2}\over R_{1,1}}\nonumber \\
z_2+{z_6\over 1-t z_5}&=&b'={b c\over a+b}=y_2\Big(y_3+{y_4\over 1-t y_5}\Big)
={1\over R_{2,0}R_{1,2}R_{1,1}}\left(R_{2,1}^2+{R_{1,1}\over 1-t {R_{2,0}\over R_{2,1}}}\right)\nonumber  \\
\qquad z_3+{z_4\over 1-t z_5}&=&c'={a c\over a+b}=y_1\Big(y_3+{y_4\over 1-t y_5}\Big)
= {R_{1,1}\over R_{2,0}R_{2,1}R_{1,2}}\left(R_{2,1}^2+{R_{1,1}\over 1-t {R_{2,0}\over R_{2,1}}}\right)\label{seedthree}.
\end{eqnarray}
where we have used the $Q$-system.
This yields positivity of all $R_{1,n}$ in terms of $(R_{1,2},R_{2,0};R_{1,1}R_{2,1})=\mu_1(\bx_0)$,
by noting that $R_{1,0}=(R_{1,1}^2+R_{2,1})/R_{1,2}$.

Figure \ref{fig:ktwo} (b) illustrates this transformation.
\eqref{expfourAtwo} may be rewritten in terms of the generating function of heaps on $K_2$ as:
$$
F_1^{(2)}(t)=R_{1,0}+R_{1,1} t
{Z^{(K_2)}(0,-tz_2,-tz_3,-tz_4,-tz_5,-tz_6)\over
Z^{(K_2)}(-tz_1,-tz_2,-tz_3,-tz_4,-tz_5,-tz_6)}.
$$
Note that the graph $K_2$ has one more node than $G_2$ and $H_2$, but
the weights depend on the same number ($4$) of free
parameters, since $z_1z_3z_5=1$, and
$z_2z_4=z_3z_6$.

\end{enumerate}

\begin{figure}
\centering
\includegraphics[width=10.cm]{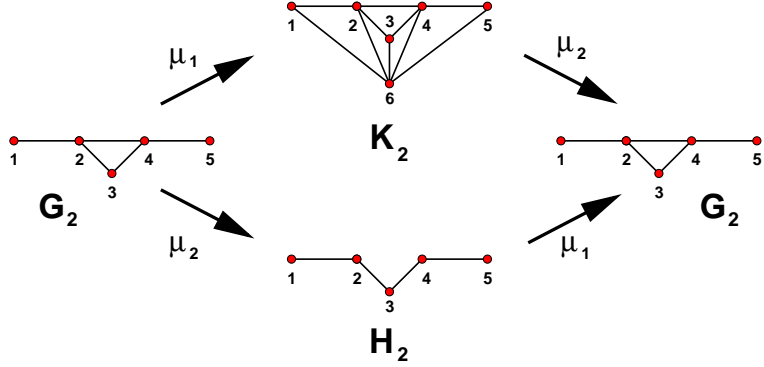}
\caption{\small The graphs encoding the $R_{1,n}$'s for the case $A_2$, and the corresponding
mutations of cluster variables.}\label{fig:slthree}
\end{figure}

The rearranged expressions for $F_1^{(2)}(t)$ considered above have
been related to mutations of cluster variables in the fundamental
domain ${\mathcal F}_2$, as well as to configurations of heaps on particular
graphs, see Figure \ref{fig:slthree}.

\end{example}

\section{Path generating functions}\label{heappa}
In order to prove the claim of the last section, it is simpler to work with
a path interpretation. There exist certain bijections between the
partition function of heaps on a graph $G$, and the partition function
of paths on an associated weighted rooted graph
$\widetilde{G}$. Such bijections are standard in the theory of heaps
\cite{HEAPS1}. 

We will establish this bijection in the case of $G_r$. Then, we
construct bijections between Motzkin paths representing cluster
variables in the fundamental domain, and weighted graphs, on which the path
partition function gives the generating function $F_1^{(r)}(t)$, in terms
of the new cluster variables. For completeness, we give the bijection with
the related graphs for heaps at the end of this section.

\subsection{From heaps on $G_r$ to paths on $\widetilde{G}_r$}\label{gpaths}
We start with the graph $G_r$ defined in Figure \ref{fig:graphgr}.
\begin{figure}
\centering
\includegraphics[width=13.cm]{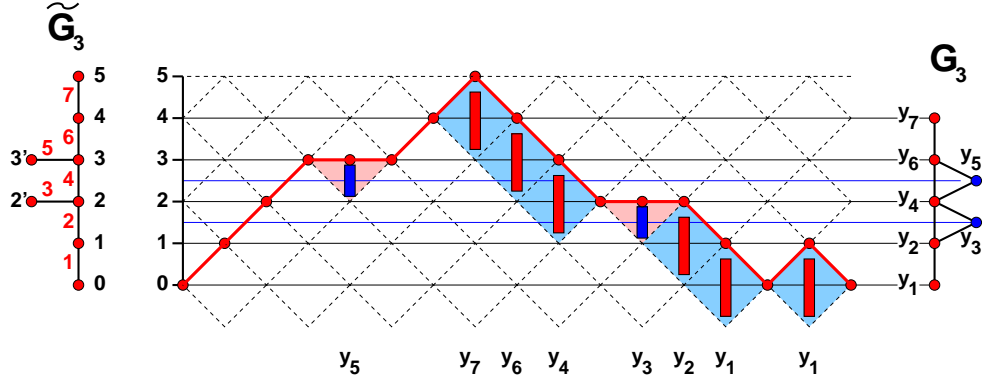}
\caption{\small A heap on $G_3$ with 8 discs, and the corresponding
lattice path of length $16$. The heap has small discs above vertices
$3$ and $5$ of $G_3$, and large discs above the other vertices.  The
large discs are in bijection with the descending steps, and the small
discs are in bijection with the horizontal steps.  This path has
weight $y_1^2y_2y_3y_4y_5y_6y_7$.  On the left is the path graph
${\tilde G}_3$, with labeled vertices (in black) and labeled edges (in
gray).}\label{fig:grheapath}
\end{figure}

\begin{defn}
The graph ${\wG}_r$ is a vertical chain
of $r+3$ vertices labeled $0,1,2,...,r+2$, with $r+2$ vertical edges
$(i,i+1)$  ($0\leq i\leq r+1$), together with $r-1$ vertices
$2',3',...,r'$ and $r-1$ horizontal edges $(i,i')$ ($2\leq i \leq r$). It is rooted at its
bottom vertex $0$.
(See the left of Fig.\ref{fig:grheapath} for the $r=3$ example.) 
\end{defn}

The graph ${\wG}_r$ is ``dual'' to $G_r$, in the sense that its
edges are in bijection with the vertices of $G_r$. We may denote the edges
by the same labels, $i=1,2,...,2r+1$.

\begin{defn}\label{wGrweights}
The weights of $\wG_r$ are defined as follows. There is a weight $y_i$
for each step along the edge $i$ which goes towards the root vertex
$0$, and a weight $1$ to all others.  That is, the step $i'\to i$ has
weight $y_{2i-1}$, the step $i\to i-1$ has weight $y_{2i}$ if
$1<i<r+1$, step $1\to 0$ has weight $y_1$, and step $r+2\to r+1$ has
weight $y_{2r+1}$.
\end{defn}

Paths along $\wG_r$, from the root to the root, also referred to as $\wG_r$-paths
below,
can be represented on a two-dimensional lattice as
follows. Paths of length $2n$ start at the point $(0,0)$, end at the
point $(2n,0)$, and cannot go below 0 or above $r+2$. They may contain
steps of the type $(i,j)\to (i+1,j\pm 1)$ (if $j\pm1$ is in the range
$0,...,r+2$), and steps $(i,j)\to (i+2,j)$ ($2\leq j\leq r$).

\begin{thm}\label{grheapath}
The heaps on $G_r$ with $n$ discs are in bijection with the
$\wG_r$-paths of length $2n$, and their partition functions
are equal, with weights as in Definition \ref{wGrweights}.
\end{thm}
\begin{proof}
Given a heap on the graph $G_r$, we associate a large disc to each
vertex along the ``backbone'' $\{1,2,4,...,2r-2,2r,2r+1\}$. We
associate a small disc to the other vertices, $\{3,5,...,2r-1\}$. The
overlap between discs is given by the graph $G_r$, namely two discs
overlap if and only if the corresponding vertices of $G_r$ are
connected via an edge. These discs are represented by bars in the
plane in Figure
\ref{fig:grheapath}. 

We draw the bars on the faces of a tilted square lattice, in such a
way that they are in bijection with the descending (for a large disc)
or horizontal (for a small disc) steps of the two-dimensional
representation of a $\wG_r$-path.  In this picture, the descending
steps of the path are the north-east edges of the square faces in
which the large discs sit, while the horizontal steps are the
horizontal diagonals of the square faces in which the small discs sit.

There is a unique way of placing the discs with this constraint.  We
decompose the heap into successive shells. Each shell has faces
containing the discs placed from the bottom to the top and from
left to right, the spaces in-between being covered with only up
steps. In this way, each large disc corresponds to a descending step
and each small one to a double horizontal step of the path.  The
correspondence of weights is clear: descending steps receive the
weights of the corresponding large discs, while the horizontal steps
receive the weights of the corresponding small discs.

Conversely, given a $\wG_r$-path, we associate bijectively a large disc
to each descending step and a small disc to each horizontal step.  The
resulting configuration is a heap over $G_r$, as the path always ends
with a descending step corresponding to a disc over vertex 1 of $G_r$.
\end{proof}

\begin{cor}\label{grpathrone}
For $n\geq 0$, the solution $R_{1,n}$ of \eqref{renom},
expressed in terms of the variables $\bx_0$, is $R_{1,0}$ times the
partition function of $\wG_r$-paths of length $2n$, 
with weights as in Definition \ref{wGrweights}.
\end{cor}

\begin{figure}
\centering
\includegraphics[width=12.cm]{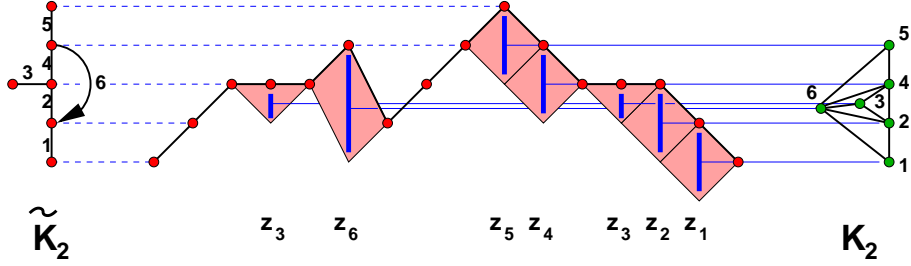}
\caption{\small The bijection between heaps on $K_2$ and paths on 
the ``dual" graph ${\tilde K}_2$. The discs of the heap give rise to polygons, whose
top right edges are exactly the descending steps of the path. We have indicated the weights
$z_i$ of the various discs, that are transferred to the descending steps of the path. We have
indicated the vertex labels on $K_2$, and the corresponding 
dual edge labels on ${\tilde K}_2$.}\label{fig:bijpatheap}
\end{figure}

The partition function for heaps $F_1^{(r)}$ on other the graphs
corresponding to fraction rearrangements can also be written in terms
of paths on graphs. For example, it is easy to see that the path graph
$\widetilde{H}_r$ corresponding to the heap graph $H_r$ (the chain
with $2r+1$ vertices) is a vertical chain with $2r+2$ vertices, with
edge $(i,i+1)$ of $\widetilde{H}_r$ corresponding to vertex $i+1$ in
$H_r$. The edge weights for descending edges are the same as the
corresponding vertex. Ascending edges have weight 1.

\begin{example}
Consider the case of $A_2$ (see Fig.\ref{fig:slthree}). We have a path
formulation of the partition function in terms of paths on $\wG_2$ as
a function of $\bx_0$. Paths on the graph $\widetilde{H}_2$ correspond
to the generating function in terms of $\mu_2(\bx_0)$.
The path
formulation for the cluster variable
$\mu_1(\bx_0)$, corresponding to
heaps on $K_2$, is the partition function for graphs on the graph
$\widetilde{K}_2$ in Figure \ref{fig:bijpatheap}.
Note that the edge labeled 6 is an oriented edge (with weight $z_6$).
\end{example}





In general, it is quite complicated to work out the direct bijection
between $G$ and $\wG$. Instead, we introduce a bijection between
Motzkin paths in the fundamental domain and path graphs, and a
bijection between the same Motzkin paths and heap graphs. This
establishes an identification between path and heap partition
functions in the cases of interest.

\subsection{Motzkin paths and path graphs}\label{seedmotz}

\subsubsection{The graphs $\Gamma_\bM$}\label{targraweight}\label{tardef} 

For any Motzkin path $\bM$ in the fundamental domain, we associate a
graph $\Gamma_\bM$ by (1)
decomposing $\bM$ into ``strictly descending" pieces; (2) associating
a subgraph to each descending piece and (3) gluing the subgraphs.
\begin{enumerate}
\item {Decomposition of $\bM$:}
Each path $\bM$, consists of {\it strictly descending}
pieces $\bm_1,\bm_2,...,\bm_p$, where $\bm_i=\{(x,y),
(x-1,y+1),...\}$. (We consider a single vertex to be a descending
piece).  These are separated by $p-1$ {\em ascending} steps of type
(i) the step (1,1) in the plane, or (ii) the step (0,1) in the plane.

\item {Graphs for $\bm_i$:}
For each strictly descending piece $\bm$ with $k$ vertices, we
define a $\Gamma(k):=\Gamma_{\bm}$ as follows.
When $k=1$, $\bm$ is a single vertex, and
$\Gamma(1)$ is a chain of four vertices, as shown in the top left of
Figure \ref{fig:strictfive}.

\begin{figure}
\centering
\includegraphics[width=10.cm]{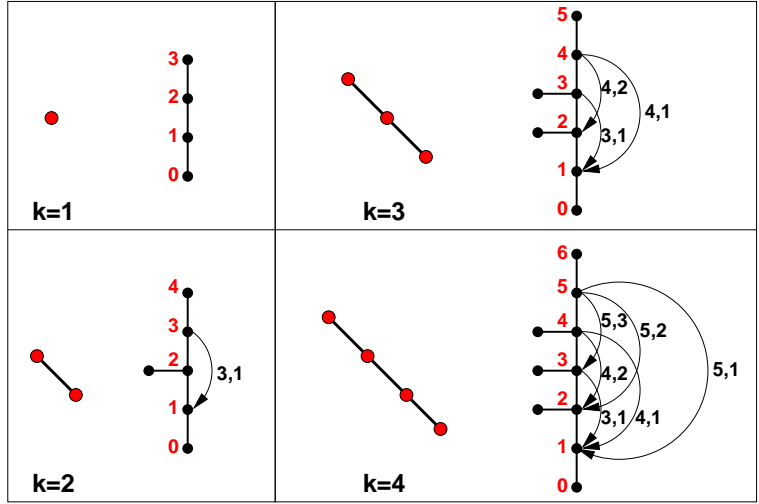}
\caption{\small The graphs $\Gamma(k)$ corresponding to strictly descending
Motzkin paths with $k=1,2,3,4$ vertices. We have indicated the extra edge labels
(in black) and the vertex labels (gray) along the vertical chain 
for the case when $i=j=0$.}\label{fig:strictfive}
\end{figure}

When $\bm$ contains $k\geq 2$ vertices, the graph
$\Gamma_\bm=\Gamma(k)$ consists of a ``skeleton'' $\wG_k$,
plus extra oriented edges $j\to i$ if
$j-i>1$, $j<k+2$ and $i>0$.  There is a total of $k(k-1)/2$ extra
oriented edges in $\Gamma(k)$, which we label by the vertices they
connect. See Figure
\ref{fig:strictfive}. 

Each $\Gamma(k)$ has four distinguished vertices denoted by
$(t,t',b',b)=(k+2,k+1,1,0)$.

\begin{figure}
\centering
\includegraphics[width=11.cm]{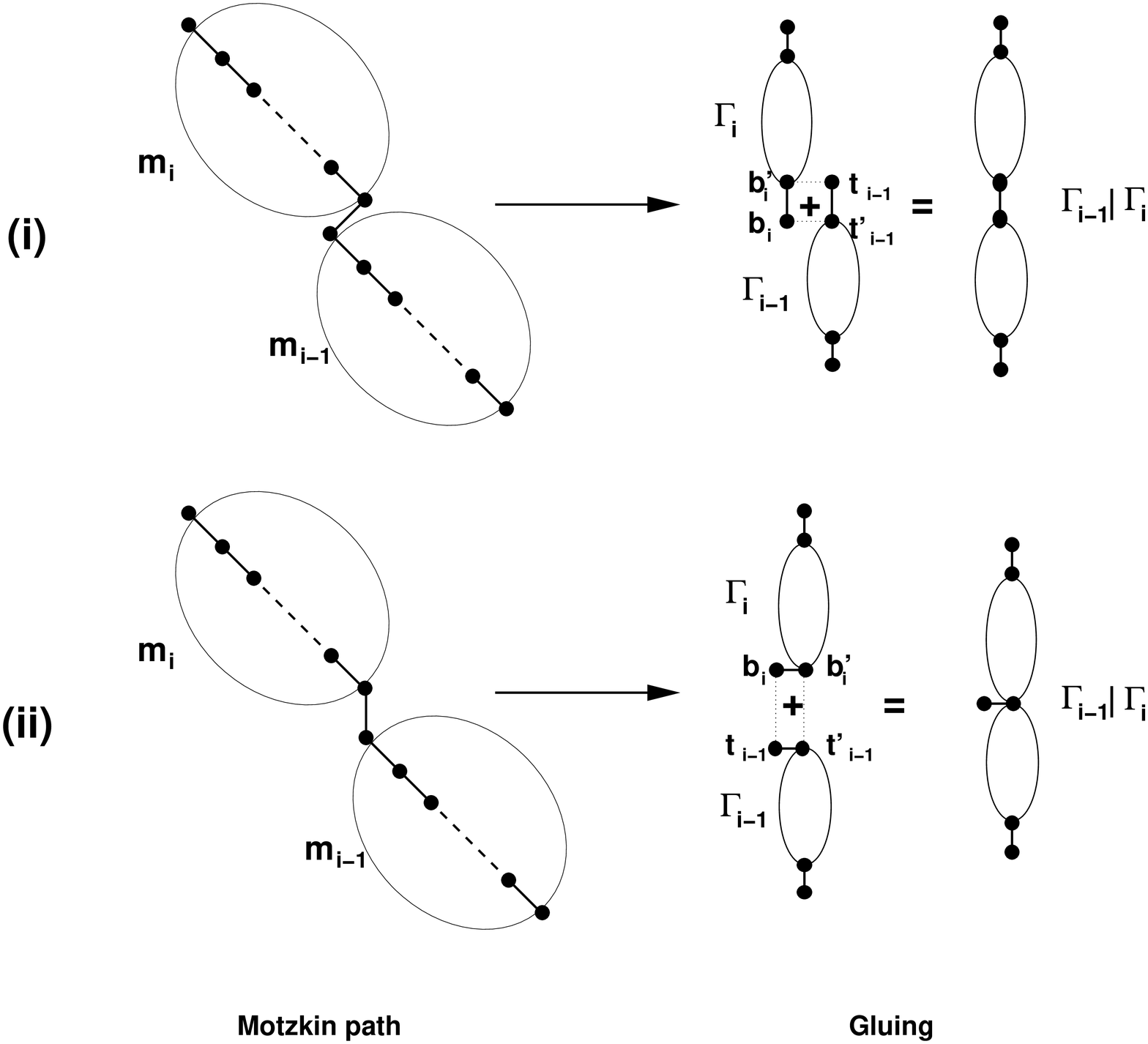}
\caption{\small The graphs $\Gamma_{\bm_{i-1}}$ and $\Gamma_{\bm_i}$ are
glued together
according to whether the pieces $\bm_{i-1}$
and $\bm_i$ are separated by a step of type (i) or (ii).
}\label{fig:gluemotz}
\end{figure}

\item {Gluing the subgraphs:}
Ascending steps of $\bM$ are of type (i) $(1,1)$ or (ii) $(0,1)$. The
graph $\Gamma_\bM$ is obtained
by gluing the graphs $\Gamma_{\bm_i}$, $i=1,2,...,p$ as follows.
Denote the edge $(t',t)$ of $\Gamma_{\bm_i}$ by $(t'_i,\rt_i)$, and so forth.
We then make the following identifications of vertices and edges
(See Figure \ref{fig:gluemotz}):

\begin{itemize}
\item {\bf Type (i):} 
Identify $\rt_{i-1}$ with $b'_i$ and $t'_{i-1}$ with $b_i$. 
The common edge is represented vertically.

\item {\bf Type (ii):} Identify $\rt_{i-1}$ with $b_i$ and 
$t'_{i-1}$ with $b'_i$.
The common edge is represented horizontally.
\end{itemize}

We denote this gluing procedure by the symbol ``$\vert$". Thus,
\begin{equation}\label{gluglu}
\Gamma_\bM=
\Gamma_{\bm_1}\vert\Gamma_{\bm_2}\vert \cdots \vert \Gamma_{\bm_k}.
\end{equation}
Finally, all vertices are renumbered sequentially from bottom to top.
\end{enumerate}

\begin{figure}
\centering
\includegraphics[width=6.cm]{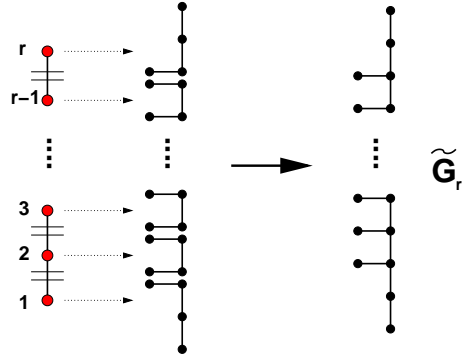}
\caption{\small The graph $\Gamma_{\bM_0}$ corresponding to the vertical
Motzkin path $\bM_0$ with $r$ vertices. The path is decomposed into
$k$ isolated vertices, each corresponding to a chain of type $\Gamma(1)$,
glued as indicated. The resulting graph is the tree
${\tilde G}_r$.}\label{fig:fold}
\end{figure}

\begin{figure}
\centering
\includegraphics[width=12.cm]{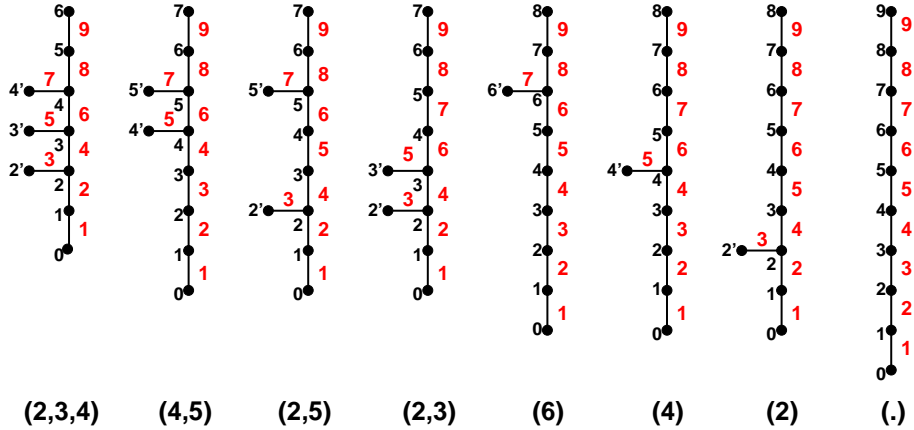}
\caption{\small The 8 trees $T_{2r+2}(i_1,i_2,...,i_s)$ for $r=4$. We have indicated the
values of the $i$'s under each tree, and the labeling of the edges (gray) and 
vertices (black).}\label{fig:treepat}
\end{figure}
\begin{example}\label{gpathex} 

Consider the case of an ascending Motzkin path containing steps of
type (i) and (ii) only. It decomposes into $r$ isolated vertices,
corresponding to $\Gamma_{\bm_i}=\Gamma(1).$ For the path $\bM_0$,
with steps only of type (ii), the result is $\Gamma_{\bM_0}=\wG_r$ (see
Fig.\ref{fig:fold}).

More generally, $\Gamma_\bM$ is obtained by gluing $r$ chains ${\Gamma}(1)$
vertically or horizontally. 

The path $\bM$ is composed of $p$ vertical chains of lengths
 $l_1,...,l_p\geq 1$, with steps of type (ii) only with $\sum_{i=1}^p
 l_i=r$. These are separated by $p-1$ steps of type (i).  Each
 vertical chain corresponds to a graph $\wG_{l_i}$, which are glued
 according to rule (i).  This results in a tree which has a
 vertical chain of length $2+\sum_i (l_i+1)$, and consecutive
 sequences of $l_i-1$ horizontal edges, separated by pairs of
 vertices.  The top two and bottom two vertices have no horizontal
 edge attached.  (see Fig.\ref{fig:treepat} for $r=4$).

The vertices $i_j$ to which the horizontal edges are attached 
$1< i_1<i_2<\cdots <i_s <
 2r-s$ where $s=r-p$, have the property that $i_1-1$ is odd, and
 $i_{j+1}-i_j$ is odd for all 
 $j<s$. In fact,
$$
\{i_1,i_2,\ldots,i_s\}=\{0,1,2,\ldots,2r-s+1\}\setminus \cup_{0\leq
j\leq p} \{j+\sum_{i=1}^j l_i,j+1+\sum_{i=1}^j l_i\}. 
$$
\end{example}
\begin{defn}\label{defn:trees}
The trees obtained in the previous example are denoted by
$T_{2r+2}(i_1,i_2,...,i_s)$.
\end{defn}
  The $2r+1$ edges of such trees are
ordered from bottom to top and labeled $1,2,..,2r+1$, including the
horizontal edges, which have labels $i_1+1,i_2+2,...,i_s+s$.

Noting that the sequence $j_a=(i_a+a-1)/2$ satisfies $1\leq
j_1<j_2<\cdots <j_s \leq r-1$ without further constraint, we see that
there are exactly $r-1 \choose s$ such trees for $r$ and $s$ fixed,
hence a total of $2^{r-1}$ when we sum over $s=0,1,2,...,r-1$. Note
that when $s=0$, $T_{2r+2}()={\tilde H}_r$, ``dual" of the chain of
$2r+1$ vertices $H_r$ introduced in Section
\ref{gpaths}.

This example is important, as it illustrates all $2^{r-1}$ possible
skeleton trees any $\Gamma_\bM$ in the fundamental domain can have:

\begin{defn}\label{skull}
The skeleton tree associated to $\bM$ is the tree $\Gamma_{\bM'}$, where
$\bM'$ is obtained from $\bM$ by replacing all the steps $(-1,1)$ by
vertical steps $(0,1)$. In other words, $\Gamma_{\bM'}$ is obtained
from $\Gamma_\bM$ by removing all its extra down-pointing edges.
\end{defn}


\subsubsection{Weights on $\Gamma_\bM$} We will compute the partition
function of paths on $\Gamma_\bM$, for which purpose, we assign a
weight to each oriented edge of $\Gamma_\bM$. An unoriented edge, in
this context, is considered to be a pair of edges oriented in opposite
directions.  The edge $(v,v')$ is considered to be an {\em ascending edge}
if the distance of $v'$ from the vertex 0 is greater than that of
$v$. Otherwise, it is a descending edge. We assign weight $1$ to all
ascending edges, and weight $y_e(\bM)$ to each descending edge $e$.

Edges are labeled as follows. Given $\Gamma_{\bM}$, consider the
skeleton tree $\Gamma_{\bM'}=T_{2r+2}(i_1,i_2,...,i_s)$ of Definition
\ref{defn:trees}. Its edges have labels $1,2,...,2r+1$, which we
retain for the graph $\Gamma_\bM$, and we assign to the descending
skeleton edge $i$ the weights $y_i$ (these are independent variables,
not necessarily related to the weights encountered earlier).

The extra descending edges of $\Gamma_\bM$ are labeled by the pairs
$(i,j)$ of vertices they connect, with $1<j+1<i\leq 2r-s$ (see Figure
\ref{fig:strictfive}).  The weights $y_{i,j}$ corresponding to
down-pointing edges $e=(i,j)$ with $i-j>1$ can be expressed in terms
of the skeleton weights.  In view of the gluing procedure, we restrict
our attention to strictly descending Motzkin paths $\bm$ with $k$
vertices (see Fig.\ref{fig:strictfive}).  There are $k(k-1)/2$
descending edges of type $(i\to j)$, $2\leq j+1 < i \leq k+1$.  Then
\begin{equation}\label{propcondi}
(y_{2j},y_{j+2,j},...,y_{k+1,j}) \propto
(y_{2j+1},y_{2j+2},y_{j+3,j+1},...,y_{k+1,j+1})\ \ \ {\rm for}\
j=1,2,...,k-1
\end{equation}
The proportionality is via overall non-vanishing scalar factors, so
these relations allow to express all the weights $y_{i,j}$ with
$i>j+1$ in terms of those of the skeleton tree,
$y_1,y_2,...,y_{2k+1}$. 
In fact, denote
\begin{equation}\label{notaweight}
y_{i+1,i}:= y_{2i}, \qquad  {\rm and}\ \  y_{i,i}:=y_{i',i}= y_{2i-1}.
\end{equation}
Then Equations \eqref{propcondi} are equivalent to
\begin{equation}\label{propconditwo}
y_{i,j}y_{m,\ell}= y_{i,\ell}y_{m,j} \ \ \ {\rm for}\ i > m\ {\rm
and}\ j>\ell
\end{equation}
That is,
\begin{equation}\label{solyij}
y_{i,j}={\prod_{\ell=j}^{i-1} y_{\ell+1,\ell} \over \prod_{\ell=j+1}^{i-1} y_{\ell,\ell}}=
{\prod_{\ell=j}^{i-1} y_{2\ell}\over \prod_{\ell=j+1}^{i-1} y_{2\ell-1} }
\end{equation}

These relations between edge weights will play a crucial role in
Section \ref{secpaths} below, when we discuss the path interpretation
of $R_{\al,n}$. They will also become clear when we express the effect
of cluster mutations on the weighted graphs.

\subsection{From Motzkin paths to heap graphs}
For completeness, we give the bijection between the path graphs
$\Gamma_\bM$ and the heap graphs dual to them. This subsection is not
necessary for further computations in this paper.  We construct a heap
graph $G_\bM$ for each Motzkin path $\bM$ via a generalized path-heap
correspondence.  

\begin{figure}
\centering
\includegraphics[width=6.cm]{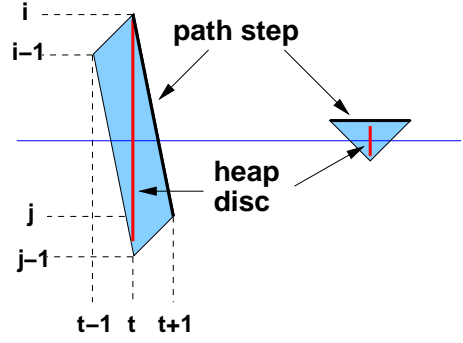}
\caption{\small The path-heap correspondence. To each descending
step $(t,i)\to (t+1,j)$ of any path on $\Gamma_\bM$, 
we associate the parallelogram indicated. The corresponding heap disc is a segment
on the vertical diagonal of the parallelogram. To each double horizontal step we associate
the triangle indicated. The corresponding heap disc is on the 
vertical median of the triangle.}\label{fig:genheapath}
\end{figure}

Given a graph $\Gamma_\bM$, we represent paths on the graph from $0$
to $0$ on the two-dimensional lattice as before. We advance by one
step in the ``time'' ($x$) direction for each step in the path, and
record the height of the vertex visited in the vertical coordinate.

Associate to each descending step $(t,i)\to (t+1,j)$ in this picture the
parallellogram with vertices $\{(t,i),(t+1,j)(t,j-1)(t-1,i-1)\}$, and
represent a vertical segment $\{(t, x), \ x\in
[j-1+\epsilon,i-\epsilon]\}$ for $\epsilon>0$ sufficiently small (see
Fig.\ref{fig:genheapath}). For double horizontal steps (steps of the
form $(t-1,i)\to (t,i)\to (t+1,i)$), we draw the half-diamond with
vertices $\{(t-1,i),(t+1,i),(t,i-1)\}$, and the segment $\{(t, x), \
x\in [i-1+\epsilon,i-\epsilon]\}$ for $\epsilon>0$ sufficiently small.
These segments represent the discs of the heap.

The heap graph $G_\bM$ encodes the overlap between these segments.
the vertices of $G_\bM$ are in bijection with the descending steps on
$\Gamma_\bM$, and the edges connect any pair of steps such that the
associated discs cannot freely slide horizontally without touching
each-other. Weights on the vertices are assigned according to edge
weights. 

To construct $G_\bM$ directly from $\bM$, we proceed as for
$\Gamma_\bM$.  We associate a graph $G_{\bm_i}$  to each strictly decending
piece $\bm_i$ of $\bM$.  We then glue these pieces according to the
type of separating step between them.

\begin{figure}
\centering
\includegraphics[width=12.cm]{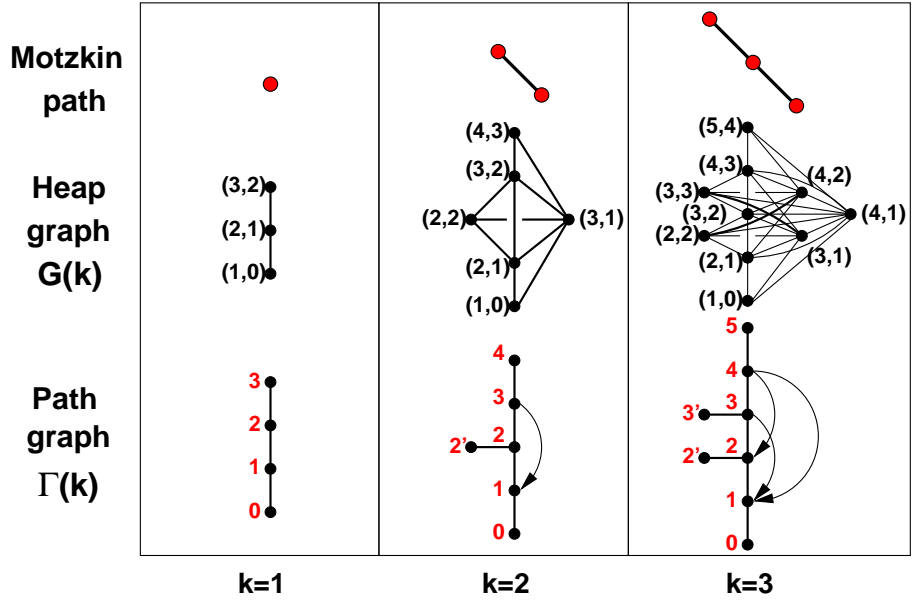}
\caption{\small The heap graphs $G(k)$ corresponding to the 
strictly descending Motzkin paths with $k=1,2,3$ vertices, together with
the corresponding path target graphs $\Gamma(k)$.
The vertices of $G(k)$ are indexed by the descending steps on $\Gamma(k)$.}\label{fig:overlapstrict}
\end{figure}

Let $\bm$ be a strictly descending Motzkin path with $k$ vertices. Let
$G_\bm=G(k)$.  Its vertices are indexed by the descending steps on
$\Gamma(k)$.    The descending
steps on $\Gamma(k)$ are indexed by $(i,j)$ for $i=j=2,3,...,k$,
$i=j+1=1,2,...,k+2$ and $k+1\geq i >j +1\geq 2$, hence $G(k)$ has a
total of $k-1+k+2+k(k-1)/2= (k+1)(k+2)/2$ vertices.  The descending
steps $k+2\to k+1$ and $1\to 0$ are singled out, and form the top and
bottom vertices of $G(k)$, denoted by $t$ and $b$.

The descending step $(i,j)$ ($i\geq j$) overlaps
with any descending step $(m,\ell)$ ($m\geq \ell$) such that $m$ or
$\ell\in [j,i]$ (or both). The set of all these overlapping descending
steps $(i,j)-(m,\ell)$ forms the edges of $G(k)$. We have represented
the graphs $G(k)$ for $k=1,2,3$ in Fig.\ref{fig:overlapstrict},
together with the path graphs $\Gamma(k)$ of Fig.\ref{fig:strictfive}
for $k=1,2,3$.

\begin{figure}
\centering
\includegraphics[width=14.cm]{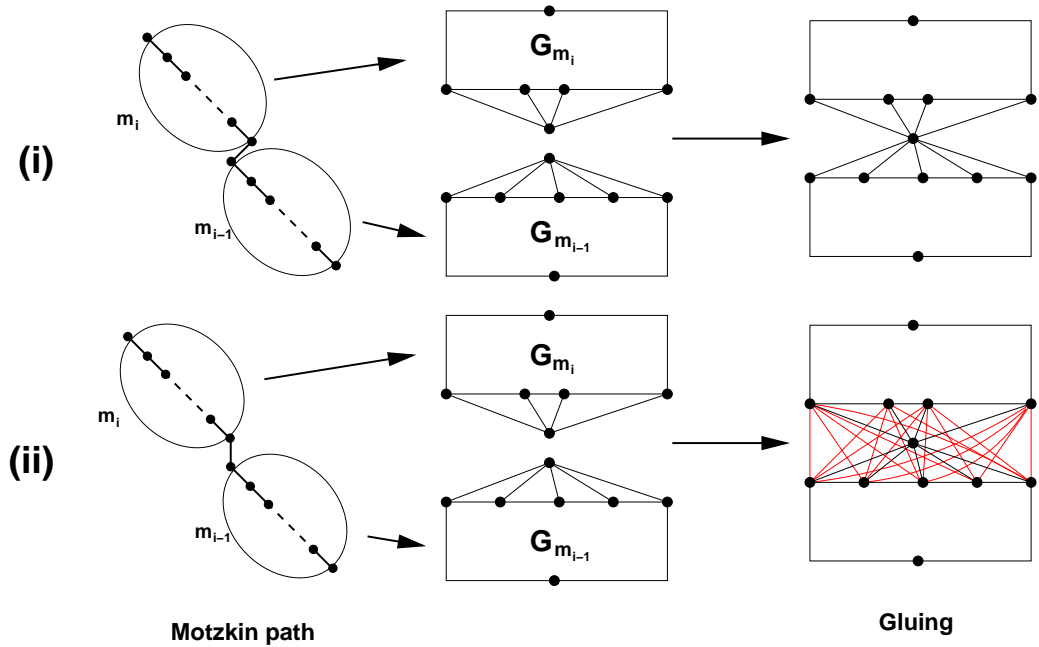}
\caption{\small The heap graphs $G_{\bm_{i-1}}$ and $G_{\bm_i}$ corresponding to the
strictly descending Motzkin paths  $\bm_{i-1}$ and $\bm_i$
are glued in two possible ways, according to whether the 
Motzkin paths are separated by (i) a step $(1,1)$ or (ii) a step $(0,1)$.
Both involve identifying the top vertex of $G_{\bm_{i-1}}$ with the bottom vertex of $G_{\bm_i}$,
but in the case (ii) all the vertices connected to the bottom vertex of $G_{\bm_i}$
must be connected to all the vertices connected to the top vertex of $G_{\bm_{i-1}}$.}\label{fig:glueheap}
\end{figure}

The graphs $G_{\bm_{i-1}}$ and $G_{\bm_i}$ are glued in two possible
ways, according to whether $\bm_{i-1}$ and $\bm_i$ are separated by a
step $(1,1)$ (type (i)) or $(0,1)$ (type (ii)) (see
Fig.\ref{fig:glueheap}):

\begin{itemize}
\item {\bf Type (i):} The top vertex $t_{i-1}$ of $G_{\bm_{i-1}}$ 
is identified with the bottom vertex $b_i$
of $G_{\bm_i}$.

\item {\bf Type (ii):} The top vertex $t_{i-1}$ of $G_{\bm_{i-1}}$ is
identified with the bottom vertex $b_i$ 
of $G_{\bm_i}$, and all the vertices connected to $t_{i-1}$ are
connected to all the vertices connected to $b_i$, via additional
edges.
\end{itemize}

\begin{example}
Consider $\mu_1(G_3)$ of
Figure \ref{fig:slfour} of Appendix \ref{appendixa} (top of second column
from left).  The corresponding Motzkin path is $\{
(1,1),(0,2),(0,3)\}$, and has two strictly descending pieces
$\bm_1=\{(1,1),(0,2)\}$ and $\bm_2=\{(0,3)\}$.  The graphs $G_{\bm_1}$
and $G_{\bm_2}$ correspond to the vertices denoted $\{1,2,3,4,5,8\}$
and $\{5,6,7\}$ respectively in Fig.\ref{fig:slfour}. As they are
glued according to the rule (ii), all the vertices connected to $5$ in
$G_{\bm_1}$, namely $4$ and $8$, are connected to all the vertices
connected to $5$ in $G_{\bm_2}$, namely $6$.
\end{example}

\begin{remark}
Let us define the overlap between descending edges of $\Gamma_{\bM}$
as follows: the descending edge $(i,j)$ ($i\geq j$) overlaps the
descending edge $(m,\ell)$ ($m\geq \ell$) if and only if $m$ or
$\ell\in [j,i]$ (or both).  We may now bypass the above gluing
procedure by directly associating to $\Gamma_\bM$ the graph $G_\bM$ as
follows: (i) the vertices of $G_{\bM}$ are in bijection with the
descending edges on $\Gamma_\bM$ and (ii) the edges of $G_\bM$ connect
all pairs of vertices such that the corresponding descending edges
overlap.
\end{remark}

The bijection between paths on $\Gamma_\bM$ and heaps on $G_\bM$
follows as before from decomposing any heap into shells of successive
foregrounds, and the fact that there is a unique way of arranging the
discs (and their surrounding parallelograms) on the square lattice
from bottom to top and left to right in each shell, and filling the
spaces in-between with ascending steps of the path.

\section{Path interpretation of $R_{1,n}$ and cluster mutations}
\label{pathinter} 
Path partition functions give a combinatorial interpretation for
$\{R_{1,n}\}$, expressed in terms of any cluster seed in the
fundamental domain.  For each Motzkin path $\bM$,
$F_1^{(r)}(t)=\sum_{n\geq 0} t^n R_{1,n}$ is a path partition function
on a graph $\Gamma_\bM$ with positive weights depending on
$\bx_\bM$. This results in a positivity theorem for all
$\{R_{1,n}\}_{n\in \Z}$, as well as explicit expressions for the
generating function.

\subsection{Path partition functions in terms of transfer matrices}

\subsubsection{Path transfer matrix}

Let $\cP_\bM^{(n)}(a,b)$ be the set of paths on the graph $\Gamma_\bM$
starting at vertex $a$ and ending at vertex $b$ with $n$ descending
steps.  The partition function is
\begin{equation}
Z_\bM^{(n)}(a,b)=\sum_{p\in \cP_\bM^{(n)}(a,b)} \prod_{e\in p} y_e(\bM).
\end{equation}
(Recall that only descending edges have non-trivial weights). Define
the generating function
\begin{equation}
Z_\bM(a,b)=\sum_{n\geq 0} t^n \, Z_\bM^{(n)}(a,b).
\end{equation}
This can be computed by use of the transfer matrix $T_\bM$, the weighted
incidence matrix of $\Gamma_\bM$, with an additional factor $t$ per
descending step. Its rows and columns are indexed by the vertices of
$\Gamma_\bM$ (the ordering on vertices is such that $i<i'<i+1$), with non-vanishing
entries$$ (T_\bM)_{i,j}=
\left\{ \begin{matrix} 
1 & \hbox{if $j\to i$ is an ascending edge},  \\
t y_{j,i} & \hbox{if $j\to i$ is a descending edge}.
\end{matrix} \right.
$$
where $y_{j,i}$ is the weight of the oriented edge  $j\to i$. 
We have
\begin{equation}\label{genefromT}
Z_\bM(a,b)= \left((I-T_\bM)^{-1} \right)_{b,a}
\end{equation} 
where $I$ is the identity matrix.

Let $F_\bM(t):= Z_\bM(0,0)$.  The entry $((I-T_\bM)^{-1})_{0,0}$
is computed by row-reduction, using upper
unitriangular row operations to obtain a lower triangular matrix, then
taking the $(0,0)$ entry of the latter.  To be precise, we iterate
the following procedure:
\begin{enumerate}
\item Let $t$ be the last row in the matrix $A$. Define the matrix
$A'$ with entries
$$
A'_{i,j}= A_{i,j}- {A_{i,\rt}A_{\rt,j}\over A_{\rt,\rt}}, \quad i<t.
$$
Then $A'_{i,\rt}=0$ for all $i< \rt$.
\item Truncate $A'$ by deleting its last row and column.
\end{enumerate}
Repeat this until the result is a $1\times 1$ matrix, which is  $1/(A^{-1})_{0,0}$.

For our particular set of graphs, the reduction procedure has a
graphical interpretation. It gives the partition function in terms of
a pruned graph with the top part removed, and replaced by a single
``loop'' with a different weight, corresponding to the partition
function on the pruned branch.

\begin{example}
The transfer matrix of the graph ${\tilde K}_2$ of
Figure \ref{fig:strictfive} is
\begin{equation}
T_{{\tilde K}_2}={\small \left( 
\begin{matrix} 
0 & t y_1 & 0       & 0       & 0            & 0 \\
1 & 0       & t y_2 & 0       & t y_{3,1} & 0 \\
0 & 1       & 0       & t y_3 & t y_4      & 0 \\
0 & 0       & 1       & 0       & 0            & 0 \\
0 & 0       & 1       & 0       & 0            & t y_5\\
0 & 0       & 0       & 0       & 1            & 0\\
\end{matrix}
\right)}
\end{equation}
with vertex order $(0,1,2,2',3,4)$.
Note that $y_{3,1}=y_2y_4/y_3$ due to  \eqref{propcondi}.
The reduction of $A=I-T_{{\tilde K}_2}$ is
\begin{eqnarray*}
I-T_{{\tilde K}_2}&\to & {\small \left(  \begin{matrix} 
1 &- t y_1 & 0       & 0       & 0             \\
-1 & 1       & -t y_2 & 0       &- t y_{3,1}  \\
0 & -1       & 1       & -t y_3 & -t y_4       \\
0 & 0       & -1       & 1       & 0             \\
0 & 0       & -1       & 0       & 1-t y_5    \\
\end{matrix} \right)} \to 
{\small \left(  \begin{matrix} 
1 &- t y_1& 0                                            & 0          \\
-1 & 1     & -t y_2 -{t y_{3,1}\over 1-t y_5}& 0    \\
0 & -1     & 1-{t y_4\over 1-t y_5}             & -t y_3       \\
0 & 0     & -1                                            & 1              \\
\end{matrix} \right)} \\
&\to &
{\small \left(  \begin{matrix} 
1 &- t y_1& 0                                             \\
-1 & 1     & -t y_2 -{t y_{3,1}\over 1-t y_5}  \\
0 & -1     & 1-ty_3 -{t y_4\over 1-t y_5}      \\
\end{matrix} \right) }\to 
{\small \left(  \begin{matrix} 
1 &- t y_1 \\
-1 & 1 -t{y_2 +{y_{3,1}\over 1-t y_5}\over 1-ty_3 -{t y_4\over 1-t y_5}} \\
\end{matrix} \right)} \\
&\to &
1-{ty_1\over 1 -t{y_2 +{y_{3,1}\over 1-t y_5}\over 1-ty_3 -{t y_4\over 1-t y_5}}} =
{1\over \left( (I-T_{{\tilde K}_2})^{-1})\right)_{0,0} }
\end{eqnarray*}
This is identical to the inverse of the factor of $t R_{1,1}$ in the second term of eq.\eqref{expfourAtwo}, 
upon identifying $z_i=y_i$, $i=1,2,3,4,5$ and $z_6=y_{3,1}$.
\end{example}

\subsubsection{Block structure}\label{blostru}

To see how this works in general, we now describe the structure of
$T_\bM$.  It consists of blocks which are put together according to
the gluing procedure.

\begin{figure}
\centering
\includegraphics[width=10.cm]{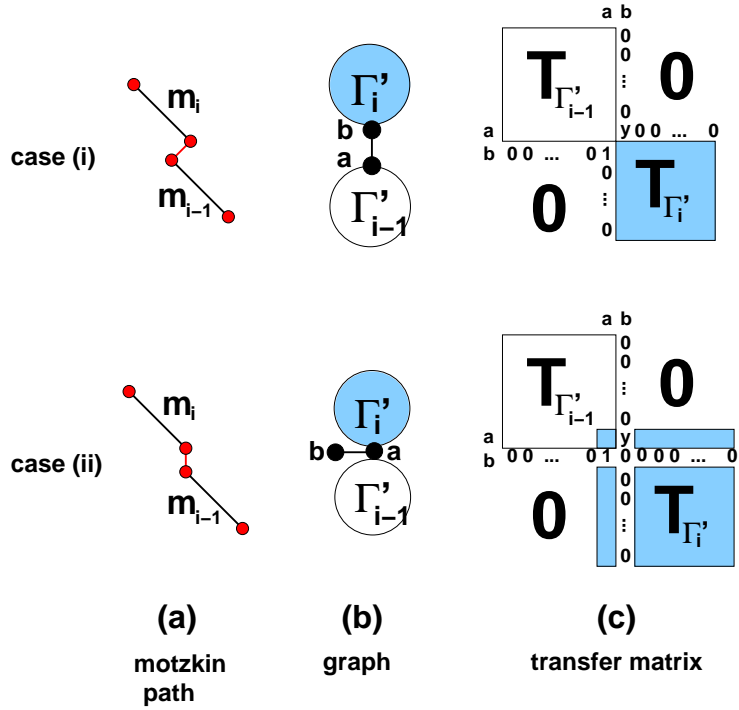}
\caption{\small Construction of the transfer matrix $T_\bM$
in the case of (i) vertical and (ii) horizontal
gluing of the subgraphs
$\Gamma_{m_{i-1}}'$ and $\Gamma_{m_i}'$. (a) Is the motzkin
path, (b) is the graph gluing and
(c) the gluing of the two diagonal blocks $T_{\Gamma_{i-1}'}$ and
$T_{\Gamma_i'}$.  We have shaded the
graph $\Gamma_i'$ in (b) and the corresponding matrix elements of the
block $T_{\Gamma_i'}$ in (c). The added element $y$ stands for the
weight $ty_{b,a}$ in both cases (i) and (ii). }\label{fig:pieces}
\end{figure}

Recall the decomposition of $\bM$ into strictly descending pieces
$\bm_i$. Let $\Gamma'_{\bm_i}$ be the graph with its bottom and top
vertices and edges removed. Define $T'_{\bm_i}$ to be the transfer
matrix of $\Gamma_{\bm_i}'$.  Let also $\Gamma_{\bm_{0}}'$ and
$\Gamma_{\bm_{p+1}}'$ denote the bottom and top vertices, glued via vertical
edges to $\Gamma'_{\bm_1}$ and $\Gamma'_{\bm_p}$. We write $\Gamma_\bM
= \Gamma'_{\bm_0} || \Gamma'_{\bm_1} || \cdots ||
\Gamma'_{\bm_{p+1}}$ to denote the gluing procedure. (Gluing via
the operation $||$ consists of adding a vertical edge or a horizontal
edge and vertex.)

The matrix $T_\bM$ is obtained by gluing the diagonal blocks
$T_{\Gamma_i'}$ according to whether the corresponding graph gluing is
via a (i) vertical or (ii) horizontal edge (see Figure \ref{fig:pieces}):

\noindent{\bf $\bullet$ Type (i):} The blocks $T_{\Gamma_{i-1}'}$ and
$T_{\Gamma_i'}$ occupy successive diagonal blocks in the matrix $T_\bM$ (vertex $t_{i-1}'$ of $\Gamma_{i-1}'$ is followed by vertex
$b'_{i}$ of $T_{\Gamma_i'}$). They are ``glued'' by the addition
of two matrix elements, $(T_\bM)_{b'_i,t'_{i-1}}=1$ and
$(T_\bM)_{t'_{i-1},b'_i}=y=ty_{b'_i,t'_{i-1}}$.

\noindent{\bf $\bullet$ Type (ii):} 
We place the two blocks $T_{\Gamma_{i-1}'}$ and $T_{\Gamma_i'}$ in the
matrix $T_\bM$ in such a way that the last row and column of
$T_{\Gamma_{i-1}'}$ (with index $t'_{i-1}$) coincide with the first
row and column of $T_{\Gamma_i'}$ (with index $b_i'$). We then insert
a new row and column labeled $b_i$ following $t'_{i-1}$, with
nonvanishing entries $(T_\bM)_{b_i,t'_{i-1}}=1$ and
$(T_\bM)_{t'_{i-1},b_i}=y=ty_{b_i,t'_{i-1}}$.

\subsection{Reduction and continued fractions}

\subsubsection{General reduction process}
Now consider the reduction procedure performed on the matrix $A=I-T_\bM$.
\begin{enumerate}
\item Let $(t',\rt)$ denote the last two row indices of $A$. 
The bottom right $2\times 2$ submatrix of $I-T_\bM$ has the form
$\left(\begin{matrix}
1 & -y \\
-1 & 1 
\end{matrix}\right)$
with $y=t y_{\rt,t'}=t y_{2r+1}$.
Reduction erases the row and column $\rt$ of $A$,
and replaces $A_{t',t'}=1\mapsto 1-y$.

\item {\bf Inductive step:}
After reduction of all rows with index $j>t'_i$, the resulting
matrix $T_{\Gamma_{\bm_0}'||\cdots ||\Gamma_i'}'$ differs from the
original only in the element $(t'_i,t'_i)$, denoted by $\gamma_i$, where
$$
(1-\gamma_i)^{-1}
=\left( (I-T_{\Gamma_{\{{t'}_i\}}\vert\vert\Gamma_{i+1}'\vert\vert 
\cdots \vert\vert \Gamma_{p+1}'})^{-1} \right)_{{t'}_i,{t'}_i} \ .
$$
We then reduce all indices $j>t'_{i-1}+1$. This results in the matrix $A'$.
Its lower right $2\times 2$ matrix, indexed by the vertices $a,b$, has the form:

\noindent{\bf Case (i):}
$\left(\begin{matrix} 1 & -y \\ -1 & 1-c
\end{matrix}\right)$ where $y=ty_{b,a}$ and 
$(1-c)^{-1}=\left( (I-T_{\Gamma_i'\vert\vert \cdots \vert\vert
\Gamma_{p+1}'})^{-1} \right)_{b,b}$. One more reduction step
eliminates row and column $b$ and gives the new entry
$A_{a,a}\mapsto  1-{y\over 1-c}$.

\noindent{\bf Case (ii):}
$\left(\begin{matrix}
1-c & -y  \\
-1 & 1
\end{matrix}\right)$. One more step in the reduction 
gives $A_{a,a}\mapsto 1-y-c$. 

In both cases, the net result is a modification of $A$
in which the bottom
right element is $1-\gamma_{i-1}$, where
$$
(1-\gamma_{i-1})^{-1}=
\left( (I-T_{\Gamma_{\{{t'}_{i-1}\}}\vert\vert\Gamma_{i}'\vert\vert 
\cdots \vert\vert \Gamma_{p+1}'})^{-1} \right)_{{t'}_{i-1},{t'}_{i-1}} \ .
$$
The bottom right element of the reduced transfer matrix is
$\gamma_{i-1}$ equal to (i) $y/(1-c)$ or (ii) $y+c $.

\end{enumerate}

\subsubsection{The case of a strictly descending Motzkin path}
Consider the transfer matrix corresponding to
$\Gamma(k)'$ (Figure \ref{fig:strictfive})).  For later use, we
use a more general matrix $T'(k)$, which has
$T(k)'_{k+1,k+1}= c$ instead of 0.  With 
vertex order $(1,2,2',3,3',...,k,k',k+1)$,
this matrix is
\begin{equation}
T'(k)={\small \left(
\begin{array}{lllllllllll}
0 & t y_2 & 0 & t y_{3,1} & 0 & t y_{4,1} & 0 & \cdots & t y_{k,1} & 0 &  t y_{k+1,1} \\
1 & 0 & t y_3 & t y_4 & 0 &  t y_{4,2} & 0 & \cdots & t y_{k,2} & 0 &  t y_{k+1,2} \\
0 & 1 & 0 & 0 & 0 & 0 & 0 & \cdots & 0 & 0 &  0 \\
0 & 1 & 0 & 0 & t y_5 & t y_6 & 0 & \cdots & t y_{k,3} & 0 &  t y_{k+1,3} \\
0 & 0 & 0 & 1 & 0 & 0 & 0 & \cdots & 0 & 0 &  0 \\
0 & 0 & 0 & 1 & 0 & 0 & t y_7 & \cdots & t y_{k,4} & 0 &  t y_{k+1,4} \\
\vdots & \vdots & \vdots & \vdots & \vdots & \vdots & \vdots & \ddots & \vdots & \vdots & \vdots \\
0 & 0 & 0 & 0 & 0 & 0 & 0 & \cdots & 0 & t y_{2k-1} & t y_{2k}\\
0 & 0 & 0 & 0 & 0 & 0 & 0 & \cdots & 1 & 0 & 0 \\
0 & 0 & 0 & 0 & 0 & 0 & 0 & \cdots & 1 & 0 & c \\
\end{array}
\right)}
\end{equation}

Reduction of the last index in $A=(I-T'(k))$
results in 
$$\left\{\begin{array}{ll} y_{k,i}\mapsto y_{k,i}+{y_{k+1,i}\over 1-c} &
\hbox{in row $i<k-1$};\\
y_{2k-2}\mapsto
y_{2k-2}+{y_{k+1,k-1}\over 1-c} & \hbox{in row $k-1$};\\
A_{kk}\mapsto1-t {y_{2k}\over 1-c} & \hbox{in row $k$}.\end{array}\right.
$$
whereas reduction of the next index changes $A_{kk}'=1-t y_{2k}/(1-c)\mapsto
1-t y_{2k-1}-t {y_{2k}\over 1-c}$. 

Let 
\begin{equation}\label{defphi}
\varphi_k(y_2,...,y_{2k};y_{3,1},y_{4,1},...,y_{k+1,k-1};c)
=\Big((I-T'(k))^{-1}\Big)_{1,1} 
\end{equation}
As a result of the two reduction steps, $\varphi_k$ is  replaced
by $\varphi_{k-1}$, with suitable substitutions of variables. 
\begin{lemma}\label{varphilemma}
The function $\varphi_k=\left((I-T'_{\Gamma(k)'})^{-1}\right)_{1,1}$
is determined by the recursion relation
\begin{eqnarray*}
&&\varphi_k(\{y_j\}_{j=2}^{2k};\{y_{j,i}\}_{2\leq i+1<j\le k+1};c)\\
&&\ \ \ \ \ \ \ \ =\varphi_{k-1}(\{y_j+\delta_{j,2k-2}{y_{k+1,k-1}\over 1-c}\}_{j=2}^{2k-2};
\{y_{j,i}+\delta_{j,k} {y_{k+1,i}\over 1-c}\}_{2\leq i+1<j\le k};t
y_{2k-1}+t {y_{2k}\over 1-c}), 
\end{eqnarray*}
with
\begin{equation}
\varphi_1(y_2;\ ;c)=\left(\left( \begin{matrix} 
1 & -t y_2 \\
-1 & 1-c
\end{matrix} \right)^{-1} \right)_{1,1}
={1\over 1-t{y_2\over 1-c}} \label{fracinit}.
\end{equation}
\end{lemma}

\begin{example}
For $k=2$, we have
\begin{eqnarray*}
\varphi_2(y_2,y_3,y_4;y_{3,1};c)=\varphi_1\left(y_2+{y_{3,1}\over 1-c};\ ; ty_3+{ty_4\over 1-c}\right)
={1\over 1-{ty_2+{ty_{3,1}\over 1-c}\over 1-ty_3-{ty_4\over 1-c}}} 
\end{eqnarray*}
For $k=3$, we have:
\begin{eqnarray*}
\varphi_3(y_2,y_3,y_4,y_5,y_6;y_{3,1},y_{4,1},y_{4,2};c)&=&
\varphi_2\left(y_2,y_3,y_4+{y_{4,2}\over 1-c}; y_{3,1}+{y_{4,1}\over 1-c};t y_5+{ty_6\over 1-c}\right)\\
&=& {1\over 1-{ty_2+{ty_{3,1}+{ty_{4,1}\over 1-c}\over 1-t y_5-{ty_6\over 1-c} }
\over 1-ty_3-{ty_4+{ty_{4,2}\over 1-c}\over 1-t y_5-t{y_6\over 1-c}}}} \\
\end{eqnarray*}
\end{example}

Thus, $\varphi_k$ is a finite, multiply branched continued fraction.
Its power series expansion in $t$ has coefficients which are
polynomials in the weights with non-negative integer
coefficients. This expression is still valid if we relax the
proportionality conditions
\eqref{propcondi}, but if we use these
for $j=1$, we see that
\begin{equation}\label{formphi}
\varphi_k={1\over 1-{y_2\over y_3}{ty_3+ ty_4 w_k\over 1-ty_3-ty_4 w_k}}
\end{equation}
where $w_k$ does not depend on $y_2,y_3,y_4$.

\subsection{Mutations as fraction rearrangements}\label{mainpro}

\subsubsection{Mutations and Motzkin paths}

\begin{figure}
\centering
\includegraphics[width=11.cm]{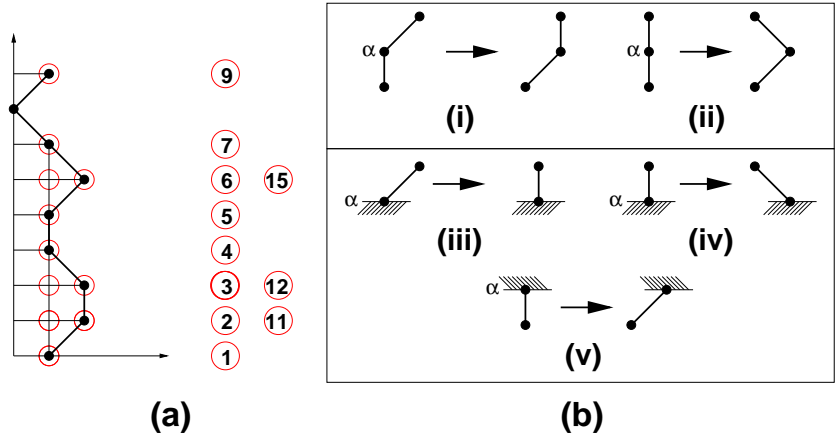}
\caption{\small A path $\bM$ with $r=9$ (a), and the
product of mutations (circled) yielding $\bM$ when applied on $\bM_0$.
In (b), the two
generic moves (i) and (ii) corresponding to the actions of the
mutation are indicated,
together with their ``boundary" version, (iii) and (iv) when $\al=1$,
and (v) when $\al=r$.}\label{fig:mudec}
\end{figure}

One can associate a unique sequence of mutations to each Motzkin path
$\bM$ using the following procedure. The
path $\bM=\{(m_\alpha,\al)\}_{\al\in I_r}$ is considered on the
lattice, and each lattice point $(x,\al)$ with $1\leq x \leq m_\al$
corresponds to a mutation $\mu_\al\ (\mu_{\al+r})$ if $x$ is odd
(even), respectively. The compound mutation $\mu$ such that
$\mu(\bx_0)=\bx_M$ is the product of these mutations read from bottom
to top, and from right to left.  For example, in Figure \ref{fig:mudec}
(a),
$\mu=\mu_{11}\mu_{12}\mu_{15}\mu_1\mu_2\mu_3\mu_4\mu_5\mu_6\mu_7\mu_9$.

This restricted set of mutations acting on $\bM_0$ yields any path
$\bM$ in the fundamental domain. We need to use only the two
elementary moves shown in Figure \ref{fig:mudec} (b) (i)--(ii), and
their ``boundary'' versions (iii)--(v). Without loss of generality we
may therefore restrict ourselves to this subset of mutations.



To see how this set of mutations acts on the partition function, we
compare the graphs and weights $(\Gamma_\bM,\by(\bM))$ and
$(\Gamma_{\bM'},\by(\bM'))$, where $\bM'=\mu_\al (\bM)$ or
$\mu_{\al+r} (\bM)$.

\begin{figure}
\centering
\includegraphics[width=16.cm]{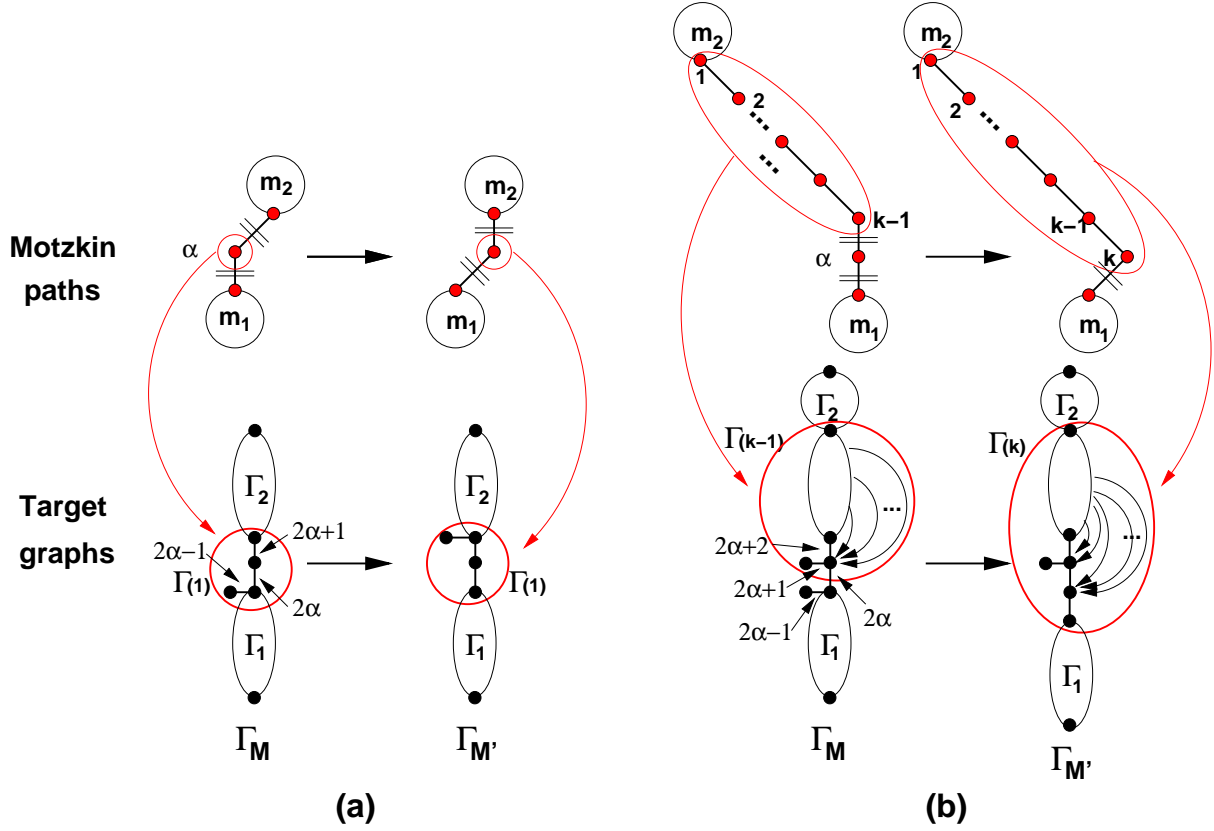}
\caption{\small The two bulk moves of Fig.\ref{fig:mudec} (b) 
acting on the paths and graphs. The edges involved in the
transformation are indicated. The mutation (a) involves the three
edges of the center piece $\Gamma(1)$, and is independent of the
structure of $\bm_1$ and $\bm_2$. Mutation (b) involves four edges,
and depends on the length of the strictly descending subpath above
vertex $\al+1$, increasing its length by 1.  On the graph it
transforms the block $\Gamma(k-1)$ into $\Gamma(k)$, creating $k$ new
descending edges.  }\label{fig:action}
\end{figure}

\subsubsection{Bulk mutations of type (a) and (b)}
Consider first the two bulk moves (i)-(ii) of Figure
\ref{fig:mudec} (b). Their effect on the
corresponding target graphs is shown in Figure \ref{fig:action} (a)
and (b).  The path $\bM$ is
decomposed into a vertex labeled $\al$, and bottom and top
pieces $\bm_1$ and $\bm_2$, corresponding to graphs $\Gamma(1)$,
$\Gamma_1$ and $\Gamma_2$ respectively.

In Figure \ref{fig:action} (a), the edge labels of the
piece $\Gamma(1)$ corresponding to the central isolated vertex $\al$
of $\bM$ are identified, they are $2\al-1,2\al$ and $2\al+1$
We find that all the independent edge weights in $\Gamma_\bM$ and
$\Gamma_{\bM'}$ are identical, except for the three weights
$y_{2\al-1},y_{2\al},y_{2\al+1}$ of $\Gamma_\bM$. These transform to
$y_{2\al-1}',y_{2\al}',y_{2\al+1}'$ in $\Gamma_{\bM'}$. As we shall
see below, the new weights correspond to the mutation of cluster variables. 

The matrix $A_\bM=1-T_{\Gamma_\bM}$, after the reduction of the block
corresponding to $\Gamma_2$, becomes:
\begin{equation}
{\small \left( 
\begin{array}{c|cc|c}
\hbox{\Large $I-T_{\Gamma_1}$} & \begin{matrix}  0 \\ \vdots \\ 0 \\
-ty_{2\al-1} \end{matrix} &  
\begin{matrix} 0 \\ \vdots \\ 0 \\ -ty_{2\al} \end{matrix} &
\hbox{\Large $0$} \\ \hline 
0\  \cdots \ 0 \ -1 & 1 & 0 & 0\ \hskip.2in  \cdots \hskip.2in \ 0 \\ 0\  \cdots
\ 0 \ -1 &0  & 1 & -ty_{2\al+1}\ 0 \  \cdots \ 0 \\ \hline
\hbox{\Large $0$}
& \begin{matrix} 0 \\ 0\\ \vdots \\ 0 \end{matrix} 
&  \begin{matrix}1 \\  0 \\ \vdots \\ 0 \end{matrix} 
&\hbox{\Large $I-T_{\Gamma_2}$} \\
\end{array}
\right)}\mapsto
{\small\left( 
\begin{array}{c|cc|c}
\hbox{\Large $I-T_{\Gamma_1}$} & \begin{matrix}  0 \\ \vdots \\ 0 \\ -ty_{2\al-1} \end{matrix} & 
\begin{matrix} 0 \\ \vdots \\ 0 \\ -ty_{2\al} \end{matrix} 
& \begin{matrix} 0 \\ \vdots \\ 0 \\ 0 \end{matrix} \\ \hline
0\  \cdots \ 0 \ -1 & 1 & 0 & 0 \\
0\  \cdots \ 0 \ -1 & 0 & 1 & -ty_{2\al+1} \\ \hline
0\  \cdots \ 0 \ \ \ \ \ 0 
& 0 & -1 & 1-c' \\
\end{array}
\right)}
\end{equation}
with $(1-c')^{-1}=\Big((I-T_{\Gamma_2})^{-1}\Big)_{b_2,b_2}$
($b_2$ is the bottom vertex of $\Gamma_2$).
Three more iterations of reduction replace the bottom right element of 
$I-T_{\Gamma_1}$ with
\begin{equation}
f=1- {ty_{2\al-1} -{t y_{2\al}\over 1-{t y_{2\al+1}\over 1-c'}}} \ .
\end{equation}
Similarly, reduction of $I-T_{\Gamma_{\bM'}}$ yields
(after reducing the $\Gamma_2$ part)
\begin{equation}
I-T_{\Gamma_{\bM'}} \mapsto{\small \left( 
\begin{array}{c|ccc}
\hbox{\huge $I-T_{\Gamma_1}$} & \begin{matrix}  0 \\ \vdots \\ 0 \\ -ty_{2\al-1}' \end{matrix} & 
\begin{matrix} 0 \\ \vdots \\ 0 \\ 0 \end{matrix} 
& \begin{matrix} 0 \\ \vdots \\ 0 \\ 0 \end{matrix} \\ \hline
0\ 0 \ \cdots \ 0 \ -1 & 1 & -ty_{2\al}' & 0 \\
0\ 0 \ \cdots \ 0 \ \ \ \ \ 0  & -1 & 1-c' & -ty_{2\al+1}' \\
0\ 0 \ \cdots \ 0 \ \ \ \ \ 0 
& 0 & -1 & 1\\
\end{array}
\right)}
\end{equation}
with $c'$ as above. Three more reduction iterations
result in the updated bottom right element of $I-T_{\Gamma_1}$,
\begin{equation}
f'=1- {ty_{2\al-1}' \over 1 -{t y_{2\al}' \over 1-t y_{2\al+1}' -c'}} \ .
\end{equation}

The partition functions of paths on $\Gamma_\bM$ and $\Gamma_{\bM'}$
must coincide, they are the same partition function expressed in terms
of different variables. They coincide if and only if
the weights $y'_{2\al-1},y'_{2\al},y'_{2\al+1}$ are such
that $f=f'$ (all the other weights are equal). Application
of the rearrangement Lemma \ref{secmov}, with
$a=y_{2\al-1}$, $b=y_{2\al}$, $c=y_{2\al+1}$ and $d=c'$, gives
$f=f'$ if and only if:
\begin{equation}\label{weightmovone}
y_{2\al-1}'=y_{2\al-1}+y_{2\al}, \quad
y_{2\al}'={y_{2\al}y_{2\al+1}\over y_{2\al-1}+y_{2\al}},
\quad y_{2\al+1}'={y_{2\al-1}y_{2\al+1}\over y_{2\al-1}+y_{2\al}} \ .
\end{equation}

We interpret this transformation as the effect of the mutation
$\mu_\al$ or $\mu_{\al+r}$ on the skeleton weights of $\Gamma_\bM$,
resulting in a rearrangement of the continued fraction $F_\bM(t)$ into
$F_{\bM'}(t)$.

The mutation in Figure \ref{fig:action} (b) is slightly more subtle,
because it depends on the length of the strictly descending subpath in
$\bm$ above $\al$, whose length is increased by 1.  All {\em
independent} edge weights are unaffected by the mutation, except those
of edges $2\al-1,...,2\al+2$ (recall that the other edge weights 
are ratios of skeleton weights $y_1,y_2,...,y_{2r+1}$).

The transfer matrix is
\begin{equation}
T_\bM={\small \left( 
\begin{array}{c|c|l}
\hbox{\Large $T_{\Gamma_1}$} & \begin{matrix}  0 \\ \vdots \\ 0 \\ ty_{2\al-1} \end{matrix} & 
\begin{matrix}\begin{matrix} 0 \\ \vdots \\ 0 \\ ty_{2\al}
\end{matrix} & \ \ \ \   \hbox{\huge $0$} \end{matrix}\\ \hline
0\ 0 \ \cdots \ 0 \ 1 & 0 & 0 \hskip.5in\cdots\hskip.5in 0 \\ \hline
\begin{matrix}  \\ \hbox{\huge $0$} \end{matrix} 
& \begin{matrix} \\ 0 \\   \vdots \\ 0 \end{matrix}  
&\hbox{\Large $T_{\Gamma_{\hbox{\small $y_{2\al+1},y_{2\al+2}$}}(k-1)\vert\Gamma_2}$} \\
\end{array}
\right)},
\end{equation}
where $\Gamma_{y_{2\al+1},y_{2\al+2}}(k-1)$ is the circled graph in
$\Gamma_\bM$, on the bottom of Fig.\ref{fig:action} (b), for which we
indicate the values of the weights of its first two lowest edges.
After reduction, we are left with $(I-T_{\Gamma_1})$, with the bottom diagonal element replaced by
\begin{equation} \label{lutfin}
f=1- ty_{2\al-1}-{ty_{2\al}\over 1-c(y_{2\al+1},y_{2\al+2})} \ ,
\end{equation}
where $(1-c(y_{2\al+1},y_{2\al+2}))^{-1}=
\Big((I-T_{\Gamma_{y_{2\al+1},y_{2\al+2}}(k-1)\vert\Gamma_2})^{-1}\Big)
_{b',b'}$, 
where $b'$ denotes the next-to-bottom vertex of $\Gamma(k-1)$.

On the other hand,
\begin{equation}
T_{\bM'}={\small \left( 
\begin{array}{c|c|l}
\hbox{\huge $T_{\Gamma_1}$} & \begin{matrix}  0 \\ \vdots \\ 0 \\ ty_{2\al-1}' \end{matrix} & 
\begin{matrix}\begin{matrix} 0 \\ \vdots \\ 0 \\ 0 \end{matrix} & \ \ \ \   \hbox{\huge $0$} \end{matrix}\\ \hline
0\ 0 \ \cdots \ 0 \ 1 & 0 & ty_{2\al}' \ 0 \   * \ 0 \ * \ 0 \  \cdots \ 0 \ * \\ \hline
\hbox{\huge $0$}
& \begin{matrix} 1 \\ 0  \\ \vdots \\ 0 \end{matrix}  
&\hbox{\huge $T_{\Gamma_{\hbox{\small $y_{2\al+1}',y_{2\al+2}'$}}(k-1)\vert\Gamma_2}$} \\
\end{array}
\right)}
\end{equation}
where the entries $*$ stand for matrix elements of the form
$ty_{j,i}'$, where the vector $(ty_{2\al}',(
ty_{j,i}' ))$ is proportional to the vector
$(ty_{2\al+1}',ty_{2\al+2}',( ty_{j,i}))$, which are the weights
appearing in the
row below it. This fact can be used to eliminate the $*$ entries using
the row below. Then the reduction gives the matrix
\begin{equation}
{\small \left(
\begin{array}{c|cc}
\hbox{\huge $I-T_{\Gamma_1}$} & \begin{matrix}  0 \\ \vdots \\ 0 \\ ty_{2\al-1}' \end{matrix} & 
\begin{matrix} 0 \\ \vdots \\ 0 \\ 0 \end{matrix} \\ \hline
0\ 0 \ \cdots \ 0 \ -1 & 1+{y_{2\al}'\over y_{2\al+1}'} & -{y_{2\al}'\over y_{2\al+1}'} \\
0\ 0 \ \cdots \ 0 \ \ \ \ 0 & -1 & 1-c(y_{2\al+1}',y_{2\al+2}') 
\end{array}
\right)}
\end{equation}
where $\Big(1-c(y_{2\al+1}',y_{2\al+2}')\Big)^{-1}=
\Big((I-T_{\Gamma_{y_{2\al+1}',y_{2\al+2}'}(k-1)\vert\Gamma_2})^{-1}\Big)_{b',b'}$,
with $c$ as in \eqref{lutfin}.  Two further
reduction steps replace the bottom right element of $(I-T_{\Gamma_1})$
by
\begin{equation}
f'=1- {ty_{2\al-1}'\over 1-{y_{2\al}'\over y_{2\al+1}'} 
{c(y_{2\al+1}',y_{2\al+2}')\over 1-c(y_{2\al+1}',y_{2\al+2}')} } \ .
\end{equation}

The partition functions for graphs on $\Gamma_\bM$ and
$\Gamma_{\bM'}$ must be equal, so we look for the transformation 
$(y_{2\al-1},y_{2\al},y_{2\al+1},y_{2\al+2})\to
(y_{2\al-1}',y_{2\al}',y_{2\al+1}',y_{2\al+2}')$ such that $f=f'$.

The function $c'=c(y_{2\al+1}',y_{2\al+2}')$ is identified by
row-reducing $(I-T_{\Gamma(k)\vert\Gamma_2})$. Reducing the $\Gamma_2$
part results in replacing the bottom right element of
$(I-T_{\Gamma(k)})$ with $1-d'$, where
$(1-d')^{-1}=\Big((I-T_{\Gamma_2})^{-1}\Big)_{b_2, b_2}$. This yields
the relation $1-{y_{2\al}'\over y_{2\al+1}'}{c'\over 1-c'}
=\varphi_{k}^{-1}(y_{2\al}',y_{2\al+1}',y_{2\al+2}',y_{2\al+3}...,y_{2\al+2k-2};
\{y_{j,i}\}; d')$, with $\varphi_{k}$ as in Lemma
\ref{varphilemma}, and we get $c'=y_{2\al+1}'+w_k
y_{2\al+2}'$, by comparison with the general expression
\eqref{formphi}. Therefore
\begin{equation}\label{finlut}
f=1- ty_{2\al-1}-{ty_{2\al}\over 1- t y_{2\al+1}- t y_{2\al+2} w_k} =
f'= 1- {ty_{2\al-1}'\over 1-{y_{2\al}'\over y_{2\al+1}'} { t y_{2\al+1}'+ t y_{2\al+2}' w_k\over
1-t y_{2\al+1}'+ t y_{2\al+2}' w_k}} \ .
\end{equation}
Applying the rearrangement Lemma \ref{secmov} with $a=y_{2\al-1}$,
$b=y_{2\al}$ and $c=y_{2\al+1}+y_{2\al+2} w_k$, we conclude that
\eqref{finlut} is satisfied if and only if
\begin{equation}\label{weightmovtwo}
y_{2\al-1}'=y_{2\al-1}+y_{2\al}, \quad y_{2\al}'={y_{2\al}y_{2\al+1}\over y_{2\al-1}+y_{2\al}},
\quad y_{2\al+1}'={y_{2\al-1}y_{2\al+1}\over y_{2\al-1}+y_{2\al}},
\quad y_{2\al+2}'={y_{2\al+2}y_{2\al-1}\over y_{2\al-1}+y_{2\al}} \ ,
\end{equation}
while all other weights $y'_i$ are equal to $y_i$. This is
interpreted as the effect of the mutation $\mu_\al$ or $\mu_{r+\al}$
on the graph weights.

\subsubsection{Boundary mutations}\label{boundamut}

We consider the mutations of Fig.\ref{fig:mudec} (b) (iii-v). 

\noindent{\bf $\bullet$ Case (v).}
We have $\al=r$, which is just case (a) of Fig.\ref{fig:action}, with
$\bm_2=\emptyset$ and $\Gamma_2$ reduced to a vertex.  The
transformation of weights is \eqref{weightmovone}.

\noindent{\bf $\bullet$ Cases (iii-iv).}
Here, $\al=1$. We take $\bm_1=\emptyset$ and $\Gamma_1$ a single
vertex in both cases (a) and (b) of Figure \ref{fig:action}.  However,
the action of $\mu_1$ or $\mu_{r+1}$ introduces a new feature, which
we call re-rooting. This is because the effect of the mutation is an
application of the rearrangement Lemma
\ref{secmov} on the corresponding weighted path partition function
$F_\bM(t)$.  This is possible only if the graph $\Gamma_\bM$ 
is rooted at vertex $b'=1$ instead of vertex $b=0$.  We write
$F_{\bM}(t)=1+t y_1(\bM) F'_{\bM}(t)$, where
$F_\bM(t)=\Big((I-T_{\bM})^{-1}\Big)_{0,0}$, while
$F'_{\bM}(t)=\Big((I-T_{\bM})^{-1}\Big)_{1,1}$.

This re-rooting must take place whenever we act via the moves (iii-iv)
as a direct consequence of the cases (a) and (b) of
Fig.\ref{fig:action}.  Indeed, $\Gamma_1$ is reduced to a vertex.  The
lower edge $2\al-1=1$ is horizontal, and the vertex common with
$\Gamma_1$ is $b'=1$, rather than $b=0$, so that $\Gamma_1=\{b'\}$.
The weight $y_1$ of the horizontal edge is associated with the step
$0\to 1$ in the re-rooted formulation.  

In general, the weighted path partition function $F_\bM(t)$,
corresponding to the Motzkin path $\bM$, is related to the initial
generating function $F_{\bM_0}(t)=R_{1,0}^{-1}F_1^{(r)}(t)$, via the
following sequence of re-rootings.
We start from the canonical sequence of Motzkin paths leading from
$\bM_0$ to $\bM=(m_1,...)$ via our restricted set of mutations.
Within this sequence, the paths $\bM^{(i_\ell)}$, $\ell=1,2,...,m_1$
are those which differ from their predecessor only in via the action
of the ``boundary'' mutations $\mu_1$ or $\mu_{r+1}$.  Note that
$m_1(\bM_\ell)=\ell$.  Only the boundary mutations $\mu_1$ and
$\mu_{r+1}$ imply a re-rooting (otherwise
$F_{\bM^{(i-1)}}=F_{\bM^{(i)}}$). Thus,

\begin{lemma}\label{reroot}
The partition function $F_\bM(t)$ is obtained from
$F_{\bM_0}(t)=R_{1,0}^{-1}F_1^{(r)}(t)$ via the sequence of
re-rootings:
$$
F_{\bM_{\ell-1}}(t)=1+t y_1(\bM_{\ell-1}) F_{\bM_\ell}(t), \ \
\ell=1,2,...,m_1
$$
with $F_\bM=F_{\bM_{m_1}}$, where $m_1=m_1(\bM)$ is the first entry of
$\bM=(m_\al,\al)_{\al=1}^r$.
\end{lemma}

This allows to interpret $R_{1,n+m_1}(\bx_\bM)$, expressed as a
function of the cluster variable $\bx_\bM$, in terms of the partition
function $F_\bM(t)$. Let us denote by $F(t)\Big\vert_{t^n}$ the
coefficient of $t^n$ in the series $F(t)$.

\begin{lemma}\label{rofF}
$$
R_{1,n+m_1}=R_{1,0} F_{\bM_0}(t)\Big\vert_{t^{n+m_1}}
=R_{1,0}\, y_1(\bM_0)y_1(\bM_1)...y_1(\bM_{m_1})\, F_\bM(t)\Big\vert_{t^n}
$$
\end{lemma}
\begin{proof}
We use Lemma \ref{reroot}. The prefactor is obtained by collecting the
successive multiplicative weights of each re-rooting.
\end{proof}

\subsubsection{Main theorem}

Let us summarize the results of this section in the following:

\begin{thm}\label{sumup}
The mutation of cluster variables $\mu_\al$ or $\mu_{\al+r}:\bx_\bM \to \bx_{\bM'}$
is equivalent to a rearrangement relating the continued fractions $F_\bM\to F_{\bM'}$
that generate weighted paths on the rooted target graphs $\Gamma_\bM$ and
$\Gamma_{\bM'}$. 
The edge weights of the corresponding skeleton trees, $\by(\bM)=\{y_1(\bM),...,y_{2r+1}(\bM)\}$ 
and $\by'=\by(\bM')$
are related through either of the two transformations \eqref{weightmovone}
or \eqref{weightmovtwo}, while all other weights remain the same.
\end{thm}

\subsection{Weights and the mutation matrix $B$}
There is a simple expression for the edge weights of $\Gamma_\bM$ in
terms of the cluster variables $\bx_\bM$ and the mutation matrix
$B(\bM)$ at the seed $\bM$.  To specify all the edge weights
$\Gamma_\bM$, one need only specify
$y_1(\bM),y_2(\bM),...,y_{2r+1}(\bM)$, due to the relations
\eqref{propcondi} for the other weights.

\begin{thm}\label{weightexpression}
The weights $y_1(\bM),y_2(\bM),...,y_{2r+1}(\bM)$ of the skeleton tree
$\Gamma_\bM$ are the following Laurent monomials of the cluster
variable $\bx_\bM$:
\begin{eqnarray}
y_{2\al-1}(\bM)&=& 
{\lambda_{\al,m_\al}\over \lambda_{\al-1,m_{\al-1}}}\ , \ (\al=1,2,...,r+1  ),
\label{oddy}\\
y_{2\al}(\bM)&=& {\mu_{\al+1,m_\al+1}\over \mu_{\al,m_{\al}}}
\left(1-\delta_{m_\al,m_{\al+1}+1}(1-
{\lambda_{\al+1,m_{\al+1}}\over \lambda_{\al+1,m_{\al}}})\right)
\label{eveny} \\
&&\hskip1in \times 
\left(1-\delta_{m_{\al-1},m_{\al}+1}(1-
{\lambda_{\al-1,m_{\al}}\over \lambda_{\al-1,m_{\al-1}}})\right)\ ,
\ (\al=1,2,...,r), \nonumber
\end{eqnarray}
where
\begin{equation}\label{lamu}
\lambda_{\al,n}={R_{\al,n+1}\over R_{\al,n}} , \qquad
\mu_{\al,n}={R_{\al,n}\over R_{\al-1,n}}. 
\end{equation}
\end{thm}
\begin{proof}
Weights are determined by their initial values
\eqref{qtyzero} at seed $\bM_0$ and by the recursion relations
\eqref{weightmovone} (in the case  $m_{\al-1}=m_\al < m_{\al+1}$) and
\eqref{weightmovtwo} (in the case $m_{\al-1}=m_\al = m_{\al+1}$) for
each mutation in our restricted set, all other weights being invariant.


Using the $Q$-system \eqref{renom}, one checks that  the relations 
\eqref{weightmovone} and
\eqref{weightmovtwo} are satisfied by
the weights
\eqref{oddy} and \eqref{eveny}. 
\end{proof}
Note that eqs.\eqref{oddy}-\eqref{eveny} involve only the
data $(R_{\al,m_\al},R_{\al,m_{\al+1}})_{\al=1}^r$, the
cluster variables at the seed $\bM$.

\begin{example}\label{tIpathweight}
Consider ascending Motzkin paths as in Example \ref{gpathex}.
The weights from Theorem
\ref{weightexpression} are 
\begin{eqnarray}
y_{2\al-1}(\bM)&=& {\lambda_{\al,m_\al}\over
\lambda_{\al-1,m_{\al-1}}}\ ,  \quad
y_{2\al}(\bM)= {\mu_{\al+1,m_\al+1}\over \mu_{\al,m_{\al}}}\ ,
 \label{exampleweights}
\end{eqnarray}
since only mutations of type (a) of the previous subsection are used.
In the particular case of $\bM=\bM_0$,
Equations \eqref{exampleweights} reduce to \eqref{qtyzero}. 

In the case of the strictly ascending path $\bM_{max}=(0,1,...,r-1)$,
with cluster variables
$(R_{\al,\al-1},R_{\al,\al})_{\al\in I_r}$, 
\begin{equation}
x_{2\al-1}:=y_{2\al-1}(\bM_{max})= {R_{\al-1,\al-2}R_{\al,\al}\over
R_{\al-1,\al-1}R_{\al,\al-1}},\quad
x_{2\al}:=y_{2\al}(\bM_{max})=  {R_{\al-1,\al-1}R_{\al+1,\al}\over
R_{\al,\al-1}R_{\al,\al} }\ .
\label{termsx}
\end{equation}
The graph $\Gamma_{\bM_{max}}={\tilde H}_r$,
the chain of $2r+2$ vertices. The partition function is
\begin{equation}\label{stielF}
F_{\bM_{max}}(t)=F_1^{(r)}(t)=\sum_{n\geq 0} t^nR_{1,n}
={R_{1,0}\over 1-t {x_1\over 1-t {x_2 \over 1- t {x_3\over  \ \ \ \ \  {\ddots \over 1-t {x_{2r}\over 1-t x_{2r+1} }}}}}}
\end{equation}
\end{example}

\begin{remark}\label{stielrem}
Recalling the discrete Wronskian expression $R_{\al,n}=\det_{1\leq
i,j\leq \al}( R_{1,n+i+j-\al-1})$, we may view the result
\eqref{termsx}-\eqref{stielF} as a particular case of a theorem of
Stieltjes \cite{STIEL1} (see also \cite{STIEL2}), which expresses the
formal power series expansion $F(\lambda)=\sum_{k\geq 0} (-1)^k
a_k/\lambda^{k+1}$ around $\infty$ for any sequence $a_k$,
$k\in\Z_{\geq 0}$, as the continued fraction
$$
F(\lambda)={1\over \beta_1 \lambda +{1\over \beta_2+{1\over \beta_3 \lambda +\cdots }}} 
$$
where $\beta_{2k}=(\Delta_k^{0})^2/(\Delta_k^1 \Delta_{k-1}^1)$, 
$\beta_{2k+1}=(\Delta_k^1)^2/(\Delta_k^0\Delta_{k+1}^0)$, and 
\begin{equation}\label{hank}
\Delta_k^m=\det_{1\leq i,j\leq k}(a_{m+i+j-2})
\end{equation}
are Hankel determinants involving the sequence $a_k$.
Indeed, writing $\lambda =-1/t$ and taking $a_k=R_{1,k}$, we easily identify the two 
continued fraction expressions $F(\lambda)$ and $F_1^{(r)}(t)$ of 
eq.\eqref{stielF} upon taking $x_k={1\over \beta_{k}\beta_{k+1}}$, while $1/\beta_1=a_0=R_{1,0}$ yields the overall
prefactor. Note that the continued fraction is actually finite, as $1/\beta_{2r+3}\propto \Delta_{r+2}^0=0$.
\end{remark}

The weights computed in Theorem \ref{weightexpression} are
simply related to the exchange matrix $B(\bM)$ of the 
cluster algebra at the seed $\bM$. 

\begin{thm}\label{bmatexpression}
The exchange matrix $B(\bM)$ at the point $\bM=\{(m_\al,\al)\}_{\al=1}^r$ reads, for $1\leq i,j\leq r$:
\begin{eqnarray}
&B(\bM)_{i,j}&=(-1)^{\lfloor {m_j+1\over 2} \rfloor} 
\left((-1)^{\lfloor {m_i\over 2} \rfloor}-(-1)^{\lfloor {m_j\over 2} \rfloor}\right)\delta_{|i-j|,1}\nonumber \\
&B(\bM)_{i,j+r}&=
(-1)^{\lfloor {m_i+1\over 2} \rfloor+\lfloor {m_j\over 2} \rfloor+1} C_{i,j}\nonumber\\
&B(\bM)_{i+r,j}&=
(-1)^{\lfloor {m_i\over 2} \rfloor+\lfloor {m_j+1\over 2} \rfloor} C_{i,j}\nonumber\\
&B(\bM)_{i+r,j+r}&=(-1)^{\lfloor {m_i\over 2} \rfloor} 
\left((-1)^{\lfloor {m_j+1\over 2} \rfloor}-(-1)^{\lfloor {m_i+1\over 2} \rfloor}\right)\delta_{|i-j|,1}
\label{Bmatrix}
\end{eqnarray}
where $C$ is the Cartan matrix of $A_r$.
\end{thm}
\begin{proof}
By induction. The theorem is true when $\bM=\bM_0$, where
$B_{\bM_0}=\left(\begin{matrix}0 & -C^t \\ C & 0\end{matrix}\right)$.

Assuming it is true for some $\bM$ with $m_{\al-1}=m_\al \leq
m_{\al+1}$, let us prove it for $\bM'$, with
$m'_\beta=m_\beta+\delta_{\beta,\al}$. We distinguish according
to the parity of $m_\al$: if $m_\al$ is even, we use the mutation
$\mu_\al$, otherwise we use $\mu_{\al+r}$.  

Assume $m_\al$ even. Under the mutation $\mu_\al$, the matrix
elements $B(\bM)_{i,\al}$ and $B(\bM)_{\al,j}$ are negated. This is
compatible with Equation \eqref{Bmatrix} by noting that $\lfloor
{m_\al'\over 2}\rfloor =\lfloor {m_\al\over 2}\rfloor$ and that
$\lfloor {m_\al'+1\over 2}\rfloor =\lfloor {m_\al+1\over 2}\rfloor+1$,
which gives the expected extra minus sign. For the
other entries of $B(M)$, we use \eqref{Bmut}:
$$
B(\bM')_{i,j}=B(\bM)_{i,j} +{\rm sgn}( B(\bM)_{i,\al})
[B(\bM)_{i,\al}B(\bM)_{\al,j}]_+
$$
with $i,j\neq \al$.  Assume $j\leq r$. Then $B(\bM)_{\al,j}\neq 0$
only if $\lfloor {m_j\over 2} \rfloor={m_\al\over 2} \pm 1$, while
$|i-\al|=|j-\al|=1$.  Due to the Motzkin path condition, this is only
possible if $m_j=m_\al-1$, but is impossible, as we must have
$m_{\al+1},m_{\al-1}\geq m_\al$. Therefore 
$B_{i,j}(\bM)=B_{i,j}(\bM')$.  Similarly, if
$i\leq r$, we have the same conclusion. 

If $i,j>r$, we write $i=i'+r$ and $j=j'+r$, with
$i',j'\leq r$. Then
$$
B(\bM)_{i'+r,\al}B(\bM)_{\al,j'+r}=-C_{i',\al}C_{\al,j'}(-1)^{\lfloor
{m_{i'}\over 2} \rfloor+\lfloor {m_{j'}\over 2} \rfloor}
$$ 
is positive only if $i'=\al$ and $j'=\al\pm 1$, in which case (i)
$m_{j'}=m_\al$ or (ii) $j'=\al+1$ and $m_{j'}=m_{\al+1}=m_\al+1$, (or
with $i'$ and $j'$ interchanged).  Then
$B(\bM)_{i'+r,\al}B(\bM)_{\al,j'+r}=2$.  In the case (i), we have
$B_{\al+r,j'+r}(\bM)=0$, hence $B_{\al+r,j'+r}(\bM')=0+2=2$, compatible
with Equation \eqref{Bmatrix}, as $\lfloor {m_\al'+1\over 2}\rfloor =\lfloor
{m_\al+1\over 2}\rfloor +1$.  In the case (ii), we have
$B_{\al+r,\al+1+r}(\bM)=-2$, hence $B_{\al+r,\al+1+r}(\bM')=-2+2=0$,
also in agreement with Equation \eqref{Bmatrix}, as ${m_{\al+1}+1\over 2}
={m_\al\over 2}+1=\lfloor {m_\al'+1\over 2}\rfloor$.

  The case of
$m_\al$ odd is treated analogously.
\end{proof}

\begin{remark}
It is interesting to note that the $B$-matrices of Theorem
\ref{bmatexpression} only have entries in
$\{0,\pm1,\pm2\}$. This is true only for the cluster subgraph
${\mathcal G}_r$, as the entries grow arbitrarily in the
complete cluster graph ${\mathbb T}_{2r}$.  We also note the
remarkable property, that the sum of the four blocks $r\times r$ of
the $B$-matrices always vanishes, namely that
$B_{i,j}+B_{i+r,j}+B_{i,j+r}+B_{i+r,j+r}=0$ for $i,j=1,2,...,r$.
\end{remark}

The sequence $\Big(
(R_{\al,m_\al})_{\al=1}^r ,(R_{\al,m_\al+1})_{\al=1}^r\Big)$ and the
sequence of cluster variables $x_\bM=(R_{\rm even},R_{\rm odd})
\equiv (x_i(\bM))_{i=1}^{2r}$  (where the $R_{\al,m}$ are listed
first for the even entries in $m$ and then for the odd ones) are
related via a permutation $\sigma_\bM$:
\begin{equation}\label{permuseed}
\sigma_\bM(\al)=\left\{ \begin{matrix} 
\al & {\rm if} \ m_\al \ {\rm is}\ {\rm even} \\
\al+r & {\rm if} \ m_\al \ {\rm is}\ {\rm odd}, \ {\rm and} \ \al\leq r\\
\al-r & {\rm if} \ m_\al \ {\rm is}\ {\rm odd}, \ {\rm and} \ \al > r
\end{matrix} \right.
\end{equation}
We have the following expression:
\begin{thm}\label{bmatweight}
The weights computed in Theorem \ref{weightexpression} are related to the exchange matrix of the
cluster algebra $B(\bM)$:
\begin{eqnarray} 
{y_{2\al-1}(\bM)\over y_{2\al}(\bM)}=\prod_{i=1}^{2r} x_i(\bM)^{B(\bM)_{i,\sigma_\bM(\al)}} ,\qquad
{y_{2\al}(\bM)\over y_{2\al+1}(\bM)}=\prod_{i=1}^{2r}
x_i(\bM)^{B(\bM)_{i,\sigma_\bM(\al+r)}},~(\al\in I_r.) \label{relaBy}
\end{eqnarray}

\end{thm}

\subsection{Positivity of the cluster variables $R_{1,n}$}\label{onefposit}
The variables $R_{1,n}$ can now be shown to be positive Laurent
polynomials of any possible set of cluster variables $\bx_\bM$.  Note that the
graphs $\Gamma_\bM$ are invariant under translations of $\bM$, but the
same is not true the weights $\by(\bM)$.

\begin{thm}\label{mainRone}
For any Motzkin path $\bM=\{(m_\al,\al)\}_{\al=1}^r$ with $r$ vertices
and any $n\geq 0$, the solution $R_{1,n+m_1}(\bx_\bM)$ of the $A_r$
$Q$-system, expressed as a function of $\bx_\bM$, is equal to
$R_{1,m_1}$ times the generating function for weighted paths on
$\Gamma_\bM$, starting and ending at the root, with $n$ down steps,
and with weights $y_e(\bM)$ given by Theorem \ref{weightexpression},
namely:
$$
R_{1,n+m_1}(\bx_\bM)=R_{1,m_1}\, F_\bM(t)\Big\vert_{t^n}
$$
Moreover,  $R_{1,n+m_1}$ is expressed as an explicit Laurent polynomial 
of the cluster variable $\bx_\bM$,
with non-negative integer coefficients, for all $n\in \Z$.
\end{thm}
\begin{proof}
By Theorem \ref{inverposit}, we can restrict ourselves to the Motzkin paths of the fundamental
domain.
The first statement of the Theorem follows from Lemma \ref{rofF}, 
with the prefactor 
$R_{1,0} y_1(\bM_0) ... y_1(\bM_p)=R_{1,m_1}$, by use of \eqref{oddy} for $\al=1$.

The second statement is clear for $n\geq 0$.
It is a direct consequence of the reinterpretation of $R_{1,n+m_1}$ as the
generating function for weighted paths with n down steps, for all $n\geq 0$, by also noting that the
weights $y_e(\bM)$ are all positive Laurent monomials of the initial data, as a consequence of 
Theorem \ref{weightexpression}. 

For $n<0$, we simply use an enhanced version
of the reflection property of Lemma \ref{firstrem}. Indeed, noting that the $Q$-system 
equations are invariant under the reflection $n\to -n$, we deduce that the quantity
$R_{1,n+m_1}$, expressed as a function
of the initial data $\bx_{\bM}$, is the {\it same} as $R_{1,-n-m_1}$ expressed as a function 
of the reflected initial data $\bx_{\bM^t}$, where the reflected Motzkin path 
$\bM^t=\{(m_\al^t,\al)\}$ satisfies $m_\al^t=-(m_\al+1)$ for $\al=1,2,...,r$. In other words,
if $R_{1,n+m_1}=f(\bx_\bM)$, then $R_{1,-n-m_1}=f(\bx_{\bM^t})$.

Let us now use
the translation invariance of the $Q$-system in the form of eq.\eqref{fundaMiam}. 
We first translate the reflected Motzkin path 
$\bM^t$ back to the fundamental domain, namely
so that its lowest entry $m_\al^t$ be zero. This is done by considering
$m^t={\rm Min}_{\al=1,2,...,r}(m_\al^t)$, and the translate $\bM'=\bM^t-m^t$, with entries
$m_\al'=m_\al^t-m^t$ for all $\al$. Then for $n<0$, the quantity $R_{1,-n-m_1+m^t}$ is still
given by the same expression as before in terms of the shifted Motzkin path $\bM'$,
namely $R_{1,-n-m_1+m^t}=f(\bx_{\bM'})$. But now the first
part of the Theorem applies, as $m_1'=m_1(\bM')=-m_1-1-m^t$. Indeed, as $-n+1\geq 0$, 
we deduce that $R_{1,-n+1+m_1'}=f(\bx_{\bM'})$ is a positive Laurent polynomial of the 
reflected-translated data $\bx_{\bM'}$. As $R_{1,n+m_1}=f(\bx_{\bM})$ via the same function $f$,
we conclude that $R_{1,n+m_1}$ is also a positive Laurent polynomial
of the initial data $\bx_\bM$, for all $n<0$. This completes the proof of the Theorem.
\end{proof}





In view of the correspondence between path partition functions on
$\Gamma_\bM$ and heap partition functions on $G_\bM$, we have also
\begin{cor}\label{heapstrikeagain}
The weighted heap graph $G_\bM$ associated to the Motzkin path $\bM$ is constructed as above.
The heaps on $G_\bM$ with $n$ discs are in bijection with the weighted paths on $\Gamma_\bM$
with $n$ descending steps, starting and ending at the root. 
For any Motzkin path $\bM$ and any $n\geq 0$,
the quantity $R_{1,n+m_1}$ is expressed in terms of the cluster variable $\bx_\bM$ as $R_{1,m_1}$
times the partition function for weighted heaps with $n$ discs on $G_\bM$.
\end{cor}

\section{Strongly non-intersecting path interpretation of $R_{\al,n}$}\label{secpaths}
We now provide a combinatorial interpretation of the determinant
expressions for $R_{\al,n}$ ($\al>1$) as partition functions of
families of strongly non-intersecting paths on $\Gamma_\bM$. For this,
we need to generalize the usual notion of non-intersecting lattice
paths. As a result we obtain a proof of the positivity of the
determinant expression for $R_{\al,n}$ with $\al>1$.

\subsection{Ascending Motzkin paths: Paths on trees $T_{2r+2}(I)$}
\label{alphagpath}
First consider the case of the ascending Motzkin paths.  The graphs
$\Gamma_\bM$ are skeleton trees, as in Section \ref{gpathex}, denoted by
$T_{2r+2}(I)$ where $I=\{i_1,i_2,\ldots,i_s\}$ with $1<i_1<i_2<\cdots
<i_s<2r-s$ and $i_1-1$, $i_2-i_1$,\ldots $i_s-i_{s-1}$ odd.  The
variables $R_{1,n}$ are partition functions of paths on these trees,
starting and ending at vertex 0.

Paths on $T_{2r+2}(I)$ are equivalent to paths on the two-dimensional
lattice, with the $y$-coordinate restricted to $0\leq y\leq
2r+1-s$, with the following possible steps: (i) $(x,y)\to(x+1,y+1)$; (ii)
$(x,y)\to(x+1,y-1)$ and (iii) $(x,y)\to (x+2,y)$ ($y\in I$).

We say that two such paths intersect if they share a vertex.

\begin{thm}\label{genealpha}
Let $n\geq \al-1$. Then $R_{\al,n}$ is $(R_{1,0})^\al$ times the
partition function of families of $\al$ non-intersecting paths on
$T_{2r+2}(I)$, starting at the points
$(0,0),(2,0),...,(2\al-2,0)$ and ending at the points
$(2n,0),(2n+2,0),...,(2n+2\al-2,0)$, with weights as in Example
\ref{tIpathweight}.
\end{thm}
\begin{proof}
We apply the Lindstr\"om-Gessel-Viennot (LGV) theorem
\cite{LGV1,LGV2}. 
Consider the partition function
$Z_{A_1,...,A_\al}^{E_1,...,E_\al}$ of non-intersecting families of
$\al$ weighted paths from the initial points $A_1,A_2,...,A_\al$ to
endpoints $E_1,E_2,...,E_\al$. We assume that if $i<j$
and $k<l$, then a path $A_i\to E_l$ must intersect
any path $A_j\to E_k$. Let $Z_{A_i\to E_j}$ be the
partition function for weighted paths from $A_i$ to $E_j$, then
\begin{equation}\label{lgv}
Z_{A_1,...,A_\al}^{E_1,...,E_\al}=\det_{1\leq i,j\leq \al} \left(
Z_{A_i\to E_j}\right)
\end{equation}

Now, $R_{1,n}$
is $R_{1,0}$ times the partition function of
paths on $T_{2r+2}(I)$ from $(0,0)$ to $(2n,0)$. Let
$A_i=(2i-2,0)$ and $E_i=(2n+2\al-2i,0)$ ($i=1,2,...,\al$). Since
${Z}_{A_i\to E_j}$ is the partition
function for paths of $2n+2\al-2i-2j+2$ time-steps, it may be
identified with $R_{1,n+\al-i-j+1}/R_{1,0}$. We conclude that
\begin{equation}
Z_{A_1,...,A_\al}^{E_1,...,E_\al}
={1\over (R_{1,0})^\al}\det_{1\leq i,j\leq \al}
\left(R_{1,n+\al-i-j+1} \right)={R_{\al,n}\over (R_{1,0})^\al} 
\end{equation}
where we have used Lemma \eqref{Rwronsk}.
\end{proof}




\begin{figure}
\centering
\includegraphics[width=11.cm]{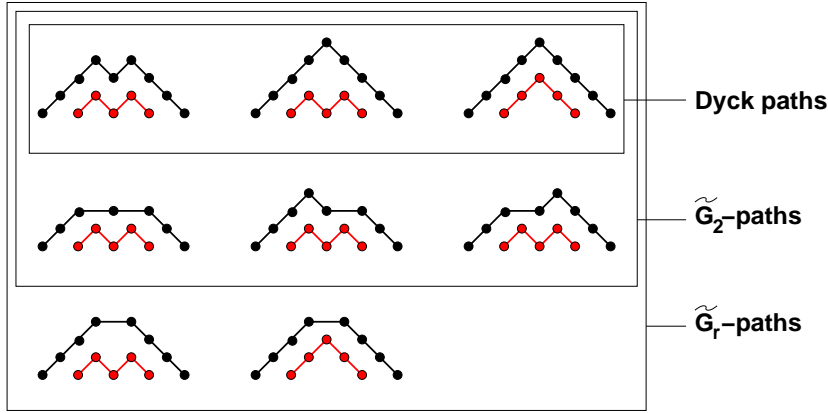}
\caption{\small The eight pairs of non-intersecting ${\tilde
G}_r$-paths of $8$ and $4$ time-steps for $r\geq 3$. The top box
contains the non-intersecting ${\tilde
H}_r$- (Dyck-) paths. There are six pairs of paths on
$\wG_2$.}\label{fig:exatwo} 
\end{figure}


\begin{example}
For $\al=2$, $n=2$ and $r\geq 3$, 
\begin{eqnarray*}
R_{2,3}=R_{1,4}R_{1,2}-(R_{1,3})^2=
&=&x_1^2x_2x_3(x_1 x_3+x_1 x_4+x_2 x_4)\\
&=&y_1^2 y_2 (y_1y_3^2+2 y_1y_3y_4+y_1y_4^2+(y_1+y_2)y_4(y_5+y_6))
\end{eqnarray*}
with $x_i$ as in \eqref{termsx} and $y_i$ as in \eqref{qtyzero}.  The
first expression corresponds to the three pairs of paths in the top of Figure
\ref{fig:exatwo}. 
The second expression
corresponds to the eight pairs of non-intersecting
${\tilde G}_r$-paths of $8$ and $4$ time-steps depicted in
Figure \ref{fig:exatwo}. We also indicated in this figure
the six pairs of non-intersecting paths on ${\tilde G}_2$, for which
no horizontal step at height $3$ is allowed, eliminating the two
configurations of the bottom row.
\end{example}

\subsection{Strongly non-intersecting paths on $\Gamma_\bM$}\label{patrep}
Theorem \ref{genealpha} can be generalized to paths on $\Gamma_\bM$ even
when $\Gamma_\bM$ is not a tree. To do this, we must generalize the notion of
non-intersecting paths, using
a representation of $\Gamma_\bM$-paths on the two-dimensional
lattice which preserves the property that the number of descending
steps is half the horizontal length of the path. (Note that this
representation is different from the one used in Figure \ref{fig:genheapath}.)

\subsubsection{Two-dimensional representation of $\Gamma_\bM$-paths}
\begin{figure}
\centering
\includegraphics[width=8.cm]{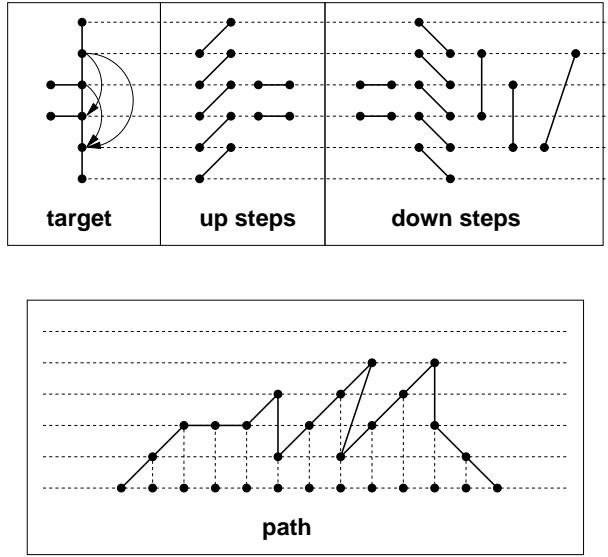}
\caption{\small The two-dimensional path representation of a 
path on
$\Gamma_\bM$ with $\bM=(2,1,0)$ (see
Figure \ref{fig:strictfive} with $r=3$).  Descents of height $2$ are
vertical (horizontal
displacement $2-h=0$), while descents of height $3$ go back one step
horizontally ($2-h=-1$). With this choice, the 
horizontal distance between the starting and ending point is twice the
number of descents ($16=2\times 8$ here).}\label{fig:pathslope}
\end{figure}

Some of the descending
steps on $\Gamma_\bM$ may have height $h>1$, in case there is an edge $j\to
i$ with $j-i>1$. When $h=0$ or $1$, each descending
step is in bijection with an ascending one, but this is not true if
$h>1$.  We require that the total horizontal displacement 
in a path $i\to i+1\to \cdots \to j-1\to j\to i$ should be
independent of the height $h=j-i$ of descending step used. Therefore, on the
two-dimensional lattice,
we draw a descent along an edge of length $h$ as a line from $(x,y)$ to
$(x-h+2,y-h)$.  Then a path which goes up $h$ single steps,
then down one step of height $h$ has horizontal length $2$,
independently of $h$.  

With this convention, the horizontal distance between the start and
end of a path from the 0 to 0 on $\Gamma_\bM$ is twice the number of
its descending steps (see Figure \ref{fig:pathslope} for an
illustration).

\begin{defn} \label{goodpath}
The two-dimensional representation of a path with $n$ descents on
$\Gamma_\bM$ is a path in $(\Z_{+})^2$, starting at $(0,0)$ and ending
at $(2n,0)$, with the possible steps:
\begin{itemize}
\item $(x,y)\to (x+1,y+1)$ whenever there is an edge $y\to y+1$ in
$\Gamma_\bM$. 
\item $(x,y)\to (x+2,y)$ whenever there is an edge $y\to y'\to
y$ in $\Gamma_\bM$.
\item $(x,y)\to (x+2-h,y-h)$, $h\geq 1$ whenever there is an edge
$y\to y-h$ in $\Gamma_\bM$.
\end{itemize}
\end{defn}
Note that these steps all preserve the parity of $x+y$, which is even,
so the path is on the even sublattice in $\Z^2$, with even $x+y$.

\subsubsection{Intersections of paths on the lattice}
The determinant $R_{\al,n}=\det_{1\leq i,j\leq
\al} (R_{1,n+\al+1-i-j})$ may be interpreted using a generalization of
the LGV formula. Intersections of the paths in Definition
\ref{goodpath} are of a more general type than the usual case.
Paths may intersect along edges, and not just at a vertex. It is clear
that edge intersections can occur only along descending edges. The
possible crossing types can be catalogued as follows.

\begin{figure}
\centering
\includegraphics[width=5.cm]{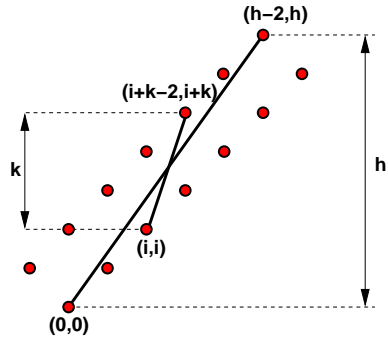}
\caption{\small The possible intersections of descending steps of two paths.
The first descent is $(h-2,h)\to(0,0)$, and the second 
is of type $(i+2-k,i+k)\to(i,i)$ with
$0\leq k <h$. There are $h-k-1$ possible values of $i$.
}\label{fig:crossings}
\end{figure}

\begin{lemma}\label{interlem}
Consider a step of path $\cP_1$ of type $(h-2,h)\to (0,0)$.
Then a descending step of a second path $\cP_2$
may cross this step only via the following $h(h-1)/2$ possible steps:
$$ 
(i+k-2,i+k)\to (i,i), \qquad {\rm for}\ \ 0\leq i\leq h-k-1, \ \ 0\leq k\leq h-1.
$$
\end{lemma}
A generic case is represented in Fig.\ref{fig:crossings}.




\subsubsection{A weight preserving involution on families of paths}

\begin{figure}
\centering
\includegraphics[width=12.cm]{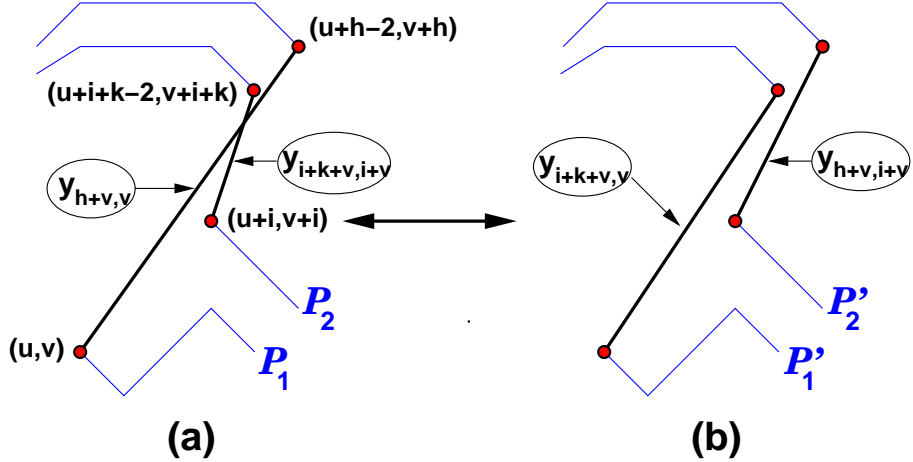}
\caption{\small A typical edge intersection of $\Gamma_\bM$-paths
(a) and the result of the flipping operation on it (b). We have the
weights of the steps. We have the identity of path weights, since
$y_{h+v,v}y_{i+k+v,i+v}= y_{h+v,i+v}y_{i+k+v,v}$.  The paths in (b)
are said to be ``too close'' to each other.}\label{fig:flipping}
\end{figure}

Given the list of edge intersections in Lemma \ref{interlem}, we
define a weight-preserving involution on families of paths which we call {\em
flipping}.  Suppose two paths, $\cP_1$ and $\cP_2$, intersect along an
edge $e_1=(u+h-2,v+h)\to (u,v)$ of $\cP_1$ and $e_2=(u+i+k-2,v+i+k)\to
(u+i,v+i)$ of $\cP_2$.  Suppose $\cP_i=L_i\cup e_i\cup R_i$ ($i=1,2$),
where $L_i$ is the subpath of $\cP_i$ before the vertex $(u+h-2,v+h)$,
and so forth.

\begin{defn}\label{flipdef} 
The flipping of $\cP_1$ and $\cP_2$ w.r.t. the intersection along the edges
$e_1$ and $e_2$ the configuration of two new paths $\cP_1',\cP_2'$
such that $\cP_i'=L_i\cup e_i'\cup R_i$, where
$e_2'=(u+i+k-2,v+i+k)\to (u,v)$ and
$e_1'=(u+h-2,v+h)\to (u+i,v+i)$.
\end{defn}

Flipping is illustrated in Figure \ref{fig:flipping}. 
The weight of two paths $(\cP_1,\cP_2)$ is equal to the product of
the weights of $\cP_1$ and $\cP_2$.

\begin{lemma}\label{flipweight}
The weight of the two paths $(\cP_1,\cP_2)$ is equal to that
of the flipped paths $(\cP_1',\cP_2')$ of Definition \ref{flipdef}.
\end{lemma}
\begin{proof} 
$$w(e_1)w(e_2)=y_{h+v,v}y_{i+k+v,i+v}=
y_{h+v,i+v}y_{i+k+v,v}=w(e_i')w(e_2'),$$
%
By virtue of Equation \eqref{propconditwo}.
The rest of the weights of the path configurations are unchanged by
the flipping operation, so the Lemma follows.
\end{proof}

For the graph $\Gamma_\bM$, there is a finite list of all possible
results of the flipping procedure.

\begin{lemma}\label{listooclose}
Given $\Gamma_\bM$, the results of flipping an
intersection between a pair of paths is a pair of
paths which include pairs of steps of the form
$((u+h-2,v+h)\to (u+i,v+i), 
(u+j-2,v+j)\to (u,v))$ for any $1\leq i \leq j <h$, where there is a
descending edge $v+h\to v$ in $\Gamma_\bM$.
\end{lemma}
\begin{proof}
By construction of $\Gamma_\bM$ (see Section \ref{tardef}), if there
exists a descending edge $v+h\to v$ on $\Gamma_\bM$, then for all
$1\leq i \leq j <h$, the edges $v+j\to v+i$ exist on $\Gamma_\bM$, as
well as $v+h\to v+i$ and $v+j\to v$.  The list of the definition
therefore exhausts all possible cases of flippings of intersections
along edges.
\end{proof}

\begin{defn}\label{tooclose}
Two paths obtained as the result of the flipping of an intersection
are called ``too close'' to each other.
\end{defn}
For example, the paths in Figure \ref{fig:flipping} (b) are too close.

Conversely, we may define the flipping of a pair of paths which are
too close to each other (with respect to the pair of edges that are
too close) as the inverse of the transformation of Definition
\ref{flipdef}.  The flipping thus defined is an involution.

\begin{defn}\label{stronginter}
Two paths are said to be strongly non-intersecting if (i) they do
not intersect and (ii) they are not too close.
\end{defn}

\subsubsection{Generalization of LGV}
We have the formula
\begin{equation}\label{deteralphaR}
{R_{\al,n+m_1}\over (R_{1,m_1})^\al}=\det_{1\leq i,j\leq \al} 
\left({R_{1,n+m_1+\al+1-i-j}\over R_{1,m_1}}\right)
\end{equation}
where $n+m_1\geq \al-1$. Using Theorem \ref{mainRone} and the path
presentation above,
$R_{1,n+m_1+\al+1-i-j}/R_{1,m_1}$ is 
the sum over
weighted $\Gamma_\bM$-paths from
$A_i=(2i-2,0)$ to $E_j=(2n+2\al-2j,0)$.

The proof of the LGV formula \eqref{lgv} uses an involution on
configurations of paths, which leaves their weight invariant up to a
sign.  This cancels various pairs of intersecting paths in the
determinant expansion.  This involution is defined as follows: it
leaves non-intersecting configurations invariant, otherwise,
it interchanges the beginnings (until the first intersection) of the
two leftmost paths that intersect first.  The only terms not cancelled
in the determinant correspond to the non-intersecting path
configurations.

The determinant
\eqref{deteralphaR} is also as a sum over families of paths from
$\{A_i\}_{i=1}^\al$ to the $\{E_j\}_{j=1}^\al$, with their weight
multiplied by the signature of the permutation $\sigma$ such that
$A_i$ is connected to $E_{\sigma (i)}$.

We define an involution $\Phi$ as follows. If a configuration has at
least one intersection or edges which are too close to each other,
then $\Phi$ ``flips" the beginnings of the two leftmost paths which
intersect or are too-close, whichever comes first. That is, we either
interchange the beginnings of paths if there is an intersection along a
common vertex, or we apply the flipping procedure described above.
Configurations which have no intersections nor edges which are too close
to each other are invariant under $\Phi$.

The map $\Phi$ is clearly an involution. Applying Lemma
\ref{flipweight}, the total weight of a flipped
configuration is invariant.  As for the ordinary case, $\Phi$
pairs up terms in the expansion of the determinant which cancel
each other, and we are left only with the $\Phi$-invariant configurations. 
Therefore the determinant \eqref{deteralphaR} is equal to the partition
function of strongly non-intersecting paths of Definition
\ref{stronginter}.

\begin{thm}\label{alphageneral}
For any Motzkin path $\bM$, and for $n+m_1\geq \al-1$, 
$R_{\al,n+m_1}/(R_{1,m_1})^\al$ is the partition function of 
configurations of $\al$ strongly non-intersecting paths
on $\Gamma_\bM$ starting at the points $A_i=(2i-2,0)$, $i=1,2,...,\al$
and ending at the points $E_j=(2n+2\al-2j,0)$, $j=1,2,...,\al$.  The
weights are expressed in terms of the Motzkin path as in
Theorem \ref{weightexpression}, and \eqref{solyij}.
\end{thm}

\subsection{Positivity of $R_{\al,n}$ for all $\al,n$}\label{positfinal}

Theorem \ref{alphageneral} implies that $R_{\al,n+m_1}$ is a positive
Laurent polynomial of any initial data, provided $n+m_1\geq \al-1$,
since all the weights in Theorem \ref{weightexpression} and
\eqref{solyij} are positive Laurent monomials of the initial
data $\bx_\bM$.

We now consider the case when $n+m_1<\al-1$.  The determinant in
eq.\eqref{deteralphaR} involves some quantities $R_{1,m}$ with indices
$m<0$, for which we have no interpretation as partition functions for
paths.  In order to generalize the result for arbitrary $n\in \Z$,
we relate the expressions $R_{\al,n+m_1}$ in the range $n+m_1< \al-1$ to
some expressions $R_{\al',n'+m_1'}$ with $n'+m_1'\geq \al'-1$.

\begin{figure}
\centering
\includegraphics[width=6.cm]{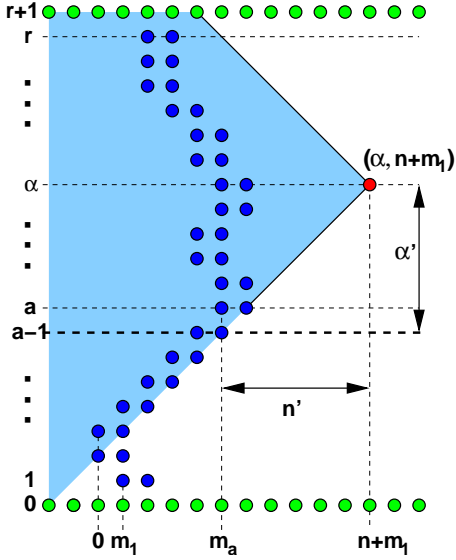}
\caption{\small The generic situation $m_\al+1<n<\al-1$ in the $(m,\al)$ plane. 
We have represented
the light-cone of values determining $R_{\al,n+m_1}$ (shaded area
inside the wedge on the left of the point $(\al,n+m_1)$), 
as well as the initial data, in the form of
two parallel Motzkin paths 
$\bM=\{(m_\beta,\beta)\}_{\beta=1}^r$ and $\{(m_{\beta}+1,\beta)\}_{\beta=1}^r$. 
The light-cone cuts out a portion $\bM'=\{(m_\beta,\beta)\}_{\beta=a}^r$ of $\bM$
so that $R_{\al,n+m_1}$ only depends on the corresponding initial data. We may
therefore truncate the picture by taking as new origin the (dashed) line
$\beta=a-1$, and interpret $R_{\al,n+m_1}$ as the solution $R_{\al',n'+m_a}$
of the $A_{r-a+1}$ $Q$-system, with initial data indexed by $\bM'$, and with
$\al'=\al-a+1$ and $n'=n+m_1-m_a$.}\label{fig:lightcone}
\end{figure}

Consider the case when $m_\al+1< n+m_1<\al-1$. Due to the structure
of the $Q$-system, each $R_{\beta,n}$ is inductively obtained from the values $R_{\beta,n-1}$,
$R_{\beta,n-2}$, $R_{\beta+1,n-1}$ and $R_{\beta-1,n-1}$. 
Consequently, $R_{\al,n+m_1}$ is a function only of the initial
data that are contained in a ``light-cone" of values of $(\beta,m)$, such that
$\al-n-m_1+k\leq \beta \leq \al+n+m_1-k$, for $k<n+m_1$ (see Fig.\ref{fig:lightcone}). 

In fact,  
$R_{\al,n+m_1}$ depends only on the initial
data $\bx_{\bM'}$ corresponding to a subset  $\bM'=\{(m_\al,\al)\}_{\al=a}^r$ of the
Motzkin path $\bM$, namely such that $(a,m_a+1)$ lies on the boundary of the light-cone, i.e.
$a=\al-n-m_1+k$ and $m_a+1=k$. For the sake of this calculation we may freely modify 
the values of $R_{a-1,m}$
and set them to $1$, as they are not involved in the expression of $R_{\al,n+m_1}$.
This has the effect of transforming the problem into one for $A_{r'}$, with $r'=r-a+1$.

More precisely, $R_{\al,n+m_1}$, as a function of a subset of the initial data
$\bx_\bM$,
may be reinterpreted as the solution
$R_{\al',n'+m_1'}'$ of the $A_{r'}$ $Q$-system,
expressed in terms of the initial data $\bx_{\bM'}$, with $m_\beta'=m_{a+\beta-1}$,
$\beta=1,2,..,r-a+1$. In this new expression, we have 
$n'+m_1'=n+m_1=\al-a+m_a+1=\al'+m_a\geq \al'$.

Note that the minimum $m'$
of $m_\beta$ on the interval $[a,r]$ may be strictly positive. 
In that case, we must use the translational invariance property of 
Lemma \ref{clustinv} (see also eq.\eqref{invtrans}) to first 
translate both the Motzkin path $\bM'$ and the index $n'$ by $-m'$.
We get $n'-m'+m_1'=\al'+m_1'-m'\geq \al'$, and the Theorem  
\ref{alphageneral} can be applied to conclude that $R_{\al',n'-m'+m_1'}'$
is a positive Laurent polynomial of the translated initial data $\bx_{\bM''}$
corresponding to $m_\beta"=m_\beta'-m'$ for all $\beta$. By Lemma
\ref{clustinv}, we find that $R_{\al',n'+m_1'}'$ is a positive Laurent polynomial
of the un-translated data $\bx_{\bM'}$.

We deduce that $R_{\al,n}$ is a positive Laurent polynomial of the
initial data $\bx_\bM$ for all values of $n> m_\al+1$. 

Finally, the positivity result is extended to $n< m_\al$ 
(including negative values of $n$) by use of
the corresponding invariance $n\to -n$ of the $Q$-system (see Lemma \ref{firstrem}), 
and the same trick as in the proof of Theorem \ref{mainRone}, involving the 
reflected-translated Motzkin path $\bM'$.
We deduce the final:

\begin{thm}\label{cestlaluttefinale}
The solution $R_{\al,n}$ of the $Q$-system for $A_r$ is a Laurent polynomial with 
non-negative integer coefficients of any initial data $\bx_\bM$ indexed
by any Motzkin path $\bM$, for all $n\in \Z$.
\end{thm}






\section{Asymptotics}
\label{asympto}
In this section, we consider two limiting cases of the results:
Solutions in the the limit $r\to \infty$, corresponding to the
$Q$-sytem of $A_{\infty/2}$, and solutions $R_{1,n}$ when
$n\to\infty$. In the latter case, we are interested in the asymptotic
behavior of the number of paths contributing to the partition function
$R_{1,n}$ as $n\to\infty$.

\subsection{The limit $A_{\infty/2}$}

In Equation \eqref{renom}, retaining only the boundary condition at
$\al=0$, $R_{0,n}=1$ ($n\in \Z$) but dropping the condition at
$\al=r+1$, gives us the solutions for the algebra $A_{\infty/2}$. All
of the results of the previous sections generalize in a
straightforward way.

The initial seed is an infinite sequence,
$(R_{\al,m_\al},R_{\al,m_\al+1})_{\al\in \Z_+}$ indexed by
semi-infinite Motzkin paths $\bM$.
For example,
for $\bM_0=(0,0,...)$, and the corresponding heap graph is
$G_{\infty/2}$, generalizing Figure \ref{fig:depth}.  Then $R_{1,n}$ as a
function of $\bx_{\bM_0}$ is $R_{1,0}$ times the generating function
for weighted heaps on $G_{\infty/2}$, $F_1(t)$ obtained as the limit
of \eqref{contia}:
\begin{equation}\label{ainfcont}
F_1(t)=1+t {y_1\over 1-t y_1-t{y_2\over 1-t y_3-t{y_4\over 1-t
y_5-t{y_6\over {}\ \ {\ddots \over 1-t y_{2\al-1}-t{y_{2\al}\over {} \
\ \ \ddots } } } } } }
\end{equation}
Using the rearrangement Lemmas \ref{firstmov} and
\ref{secmov} we can write the expressions for the generating function
in terms of other initial seeds.
For instance, the generating function corresponding to the ``maximal" Motzkin
path $\bM_{\max}$ with $m_\al=\al-1$ is the continued fraction:
\begin{equation}\label{htcont}
{1\over 1-t {y_1\over 1-t {y_2\over 1-t {y_3\over 1-t{y_4\over 1-t {y_5\over 1-t{y_6\over {}\ \ 
{\ddots  \over 1-t {y_{\al}\over {} \ \ \ \ddots } } } } } } } } }
\end{equation}

The limit of the generating function corresponding to
the strictly descending Motzkin path with $r$ vertices, of the form
$1/(1-t y_1 \varphi_r)$, with $\varphi_r$ as in Lemma
\ref{varphilemma} is a ``continued fraction" with infinitely many
branchings, a sort of self-similar object.

\subsection{Numbers of configurations}

We consider the number of configurations contributing to $R_{1,n}$ in
general.  For this purpose, set $R_{\al,m_\al}=R_{\al,m_{\al+1}}=1$
($\al=1,2,...,r$). This implies that all the edge weights are
$y_{e}=1$.  Therefore, $R_{\al,n}$ is a non-negative integer equal
counts the numbers of configurations of the related statistical model.

For example, when $\bM=\bM_0$,
\begin{lemma}
For the initial data $R_{\al,0}=R_{\al,1}=1$, $\al=1,2,...,r$,
the generating function $F_1^{(r)}(t)=\sum_{n\geq 0} R_{1,n}t^n$ reads:
\begin{equation}
F_1^{(r)}(t)=1+t {P_{r}(t)\over P_{r+1}(t)}, \ {\rm with}\ \ 
P_m(t)=t^{m\over 2} U_m\left( {1\over \sqrt{t}}-\sqrt{t}\right)
\end{equation}
where $U_m$ are the Chebyshev polynomials of the second kind, with 
$U_m(2\cos \theta)={\sin (m+1)\theta \over \sin \theta}$. The corresponding limit $r\to \infty$
for $A_{\infty/2}$ reads:
\begin{equation}\label{schroeder}
F_1^{(\infty)}(t)=1+tz(t)={3-t -\sqrt{1-6t+t^2}\over 2}=1+t+ 2  t^2 + 6 t^3 + 22 t^4 + 90 t^5 + 394 t^6 + \cdots
\end{equation}
where $z(t)={1-t -\sqrt{1-6t+t^2}\over 2t}$ is the generating function
of the large Schroeder numbers.
\end{lemma}
\begin{proof}
We use the expression \eqref{contia} with all weights
equal to 1. Then there is a recursion
relation, $P_{r+1}(t)=(1-t)P_r(t)-t P_{r-1}(t)$, with $P_{0}(t)=1$ and
$P_1(t)=1-t$. The limit $z(t):=\lim_{r\to\infty} P_r(t)/P_{r+1}(t)$
therefore satisfies $z=1-t+{t\over z}$. It also coincides
with the continued fraction \eqref{ainfcont} with $y_i=1$ for all $i$.
\end{proof}

The rate of growth of $R_{1,n}$, considered as the number of paths
of length $2n$ on $\Gamma_{\bM_0}$, can also be analyzed.  The radius of
convergence of the series $F_1^{(r)}(t)$ is given by the smallest root
of the denominator of the fraction, which is $U_{r+1}\big({1\over
\sqrt{t}}-\sqrt{t}\big)$. As the zeros of the Chebyshev polynomial
$U_m$ are $2\cos\,\pi {k\over m+1}$, $k=1,2,...,m$, when
$n\to\infty$
\begin{equation}
R_{1,n} \sim C_r \times \left(\cos\big({\pi\over r+2}\big)+\sqrt{\cos^2\big({\pi\over r+2}\big)+1}
\right)^{2n} 
\end{equation}
for some constant $C_r$.

The number of paths on $\Gamma_{\bM_{max}}$ is also simple to compute,
as it is the number of Dyck paths of length $2n$, which are limited
by height $2r+1$:
\begin{equation}
F_1^{(r)}(t)={1\over \sqrt{t}}
{U_{2r+1}\left({1\over \sqrt{t}}\right) \over U_{2r+2}\left({1\over \sqrt{t}}\right)}
\end{equation}
In the limit $r\to \infty$,
\begin{equation}
F_1^{(\infty)}(t)={1 -\sqrt{1-4t}\over 2t}=1+t+ 2  t^2 + 5 t^3 + 14
t^4 + 42 t^5 + 429 t^6 + \cdots
\end{equation}
is the generating function of the Catalan numbers $c_n={1\over
n+1}{2n\choose n}$.

As a result of the correspondence to domino tilings in the next
section, this function counts the number of domino tilings of the
(possibly truncated) halved Aztec diamond with $d$ tiles missing.  The
function $R_{\al,n}$ counts the number of its indented versions. 

The large $n$ behavior is $R_{1,n}\sim C_r'
\big(2\cos({\pi\over 2r+3})\big)^{2n}$, for some constant $C_r'$.

Now consider the expression as a function of the initial seed
corresponding to $\bM_2$, with $m_\al=r-\al$ (the maximal descending
Motzkin path), where $R_{1,n+r-1}/R_{1,r-1}$ is as in Lemma
\ref{varphilemma}. 
\begin{eqnarray*}
F_1^{(r)}(t)=\sum_{j=0}^{r-2}R_{1,j}t^j +t^{r-1} R_{1,r-1}\Phi_r(t;\by),\qquad
\Phi_r(t;\by)={1\over 1-t y_1 \varphi_r(\{y_j\};\{y_{i,j}\};c)},
\end{eqnarray*}
where $\varphi_r$ is defined in Lemma \ref{varphilemma}, and the
arguments are the weights $y_i=y_i(\bM_2)$, $i=2,3,...,2r$,
$y_{i,j}=y_{i,j}(\bM_2)$, $2\leq j+1<i\leq r+1$, and
$c=ty_{2r+1}(\bM_2)$. 
\begin{lemma}
Setting $R_{\al,r-\al}=R_{\al,r-\al+1}=1$ ($\al\in I_r$),
\begin{equation}
\Phi_r(t)=1+t {V_r(t)\over V_{r+1}(t)}, \ \ V_m(t)=(-1)^mU_{2m}(\sqrt{t}).
\end{equation}
\end{lemma}
\begin{proof}
Setting the weights $y$'s to $1$ in $1/(1-t
\varphi_r)$,  using the recursive definition of $\varphi_r$ of Lemma
\ref{varphilemma} with
$\varphi_m(\{y_j\};\{y_{i,j}\};ty_{2m+1})\vert_{\by=1}={1\over
1-\theta_m(t)}$ ($m\geq 0$), $\varphi_0=1$ and $\theta_0=0$, we have
$\theta_m={t+\theta_{m-1}\over 1-t-\theta_{m-1}}$.  Let
$1-t-\theta_m={V_{m+1}\over V_m}$, then we have a three-term recursion
relation for $V_m$: $V_{m+1}=(2-t)V_m-V_{m-1}$, where $V_0=1,V_1=1-t$
(or equivalently $V_{-1}=V_0=1$). The solution is
$V_m(t)=(-1)^mU_{2m}(\sqrt{t})$, by iterating the Chebyshev recursion
relation $U_{m+1}(x)=xU_m(x)-U_{m-1}(x)$, to get
$U_{n+1}=(x^2-2)U_{n-1}-U_{n-3}$, and identifying the initial
conditions.
\end{proof}

The asymptotic behavior of $R_{1,n}$ with these boundary conditions
is found from the smallest root of the denominator
$U_{2r+2}(\sqrt{t})$.  Then
\begin{equation}\label{ratogro}
R_{1,n} \sim C_r'' \times {1\over \left(2 \sin\big({\pi\over
2(2r+3)}\big)\right)^{2n} } 
\end{equation}
as $n\to \infty$, for some constant $C_r''$.

In this case, the limit $r\to\infty$ of $\Phi_r$ is
ill-defined. The Motzkin path $\bM_2$ has $m_1=r-1\to \infty$,
hence does not have a good limit in our picture.  More
importantly, the graph $\Gamma_{\bM_2}$ has an infinite number of
incoming edges at each vertex $1,2,3...$. Hence, as we count paths
according to their numbers $n$ of down steps, we get an infinite
number of paths from $0$ to $0$ as soon as $n\geq 2$. This can also be
seen in the fact that the rate of growth in eq.\eqref{ratogro}
diverges as $2r/\pi$.

\begin{example}
In the cases $r=1,2,3$, we have 
\begin{eqnarray*}
\Phi_1(t)&=&{1\over 1-{t\over 1-{t\over 1-t}}}
=1+t{1-t\over 1-3t+t^2}
=1-t{U_2(\sqrt{t})\over U_4(\sqrt{t})}\\
\Phi_2(t)&=&{1\over 1-{t \over 1-{t+{t\over 1-t}\over 1-t-{t\over 1-t}}}}
=1+t{1-3t+t^2\over 1-6t+5t^2-t^3}
=1-t{U_4(\sqrt{t})\over U_6(\sqrt{t})}\\
\Phi_3(t)&=&{1\over 1-{t \over 1-{t+{t+{t\over 1-t}\over 1-t-{t\over 1-t}}\over 
1-t-{t+{t\over 1-t}\over 1-t-{t\over 1-t}}}}}
=1+t{1-6t+5t^2-t^3\over  1-10t+15t^2-7t^3+t^4}
=1-t{U_6(\sqrt{t})\over U_8(\sqrt{t})}\\
\end{eqnarray*}

\end{example}

More generally, for arbitrary choices of $\bM$ and the associated initial conditions 
$R_{\al,m_\al}=R_{\al,m_\al+1}=1$, we have $F_1^{(r)}(t)=L_r(t)/K_{r+1}(t)$, where $L_m(t),K_m(t)$
are polynomials of degree $m$ with integer coefficients depending on $\bM$, and $K_m(0)=1$
and $K_m(t)\sim (-t)^m$ for large $t$.
Indeed, from the results of Section \ref{cons},
we know that $K_{r+1}(t)$ is the generating 
function for hard particles on $G_\bM$, with weight $-t$
per particle. So the empty configuration contributes $K_{r+1}(0)=1$.
Due to connectivity of $G_\bM$, the maximally occupied hard-particle
configuration corresponds
to all odd vertices $2i+1$, $i=0,1,...,r$ being occupied by a particle (these are the
duals of all the skeleton-tree edges of $\Gamma_\bM$, with odd labels). The corresponding weight
is therefore $(-t)^{r+1}$. 

For $r$ large,  it is always possible to reexpress $F_1^{(r)}(t)$ for any Motzkin path $\bM$
obtained from $\bM_0$ via finitely many mutations in the form of a rational fraction of the
variables $t$ and $P_{r-1}(t)/P_r(t)$. Indeed, the branchings of the continued fraction
expression for $F_1^{(r)}(t)$ all include some ``tails" of the form $1/(1-t y_a -t (y_b/(1-t y_c...)))$
which, when we set all $y_i=1$, reduce to some ratio $P_m(t)/P_{m+1}(t)$. By use
of the 3-term recursion relation for $P_m$, this can always be rewritten as a rational
fraction of $t$ and $P_{r-1}/P_r$. Then, in the case of $A_{\infty/2}$, upon taking
the limit $r\to\infty$, we see that, as $P_{r-1}/P_r\to z(t)$, $F_1^{(r)}(t)$ takes
the general form
$A(t)+B(t) \sqrt{1-6t+t^2}$, where $A$ and $B$ are rational fractions of $t$ with integer coefficients.

For instance let us take $\bM=\mu_1\bM_0$. Then we have
\begin{eqnarray*}
F_1^{(r)}(t)&=& R_{1,0}+ {t R_{1,1} \over 1-t {y_1\over 1-t 
{y_2+{y_{3,1}\over 1-t y_5 -t{y_6\over 1-t y_7-\cdots}}
\over 1-t y_3-t {y_4\over 1-t y_5 -t{y_6\over 1-t y_7-\cdots}}}}}\Bigg\vert_{\by=\bf 1}\\
&=&2+{t\over 1-{t\over 1-t\left(1+{P_{r-2}(t)\over P_{r-1}(t)}\right){P_{r-1}(t)\over P_r(t)}}}=
2+t +{t^2\over 2-t-{P_{r-1}(t)\over P_r(t)}}\\
\end{eqnarray*}
where we have used the $Q$-system to rewrite $R_{1,0}={R_{1,1}^2+R_{2,1}\over R_{1,2}}=2$,
and the recursion relation for $P$ to eliminate $P_{r-2}$.
In the limit $r\to\infty$, this yields for the $A_{\infty/2}$ $Q$-system with initial conditions
$R_{1,1}=R_{1,2}=1$ and $R_{\al,0}=R_{\al,1}=1$ for all $\al\geq 2$:
$$
F_1^{(\infty)}(t)=2+t+ {t^2\over 2-t-z(t)}= 2+t+t^2{1-5t+2t^2+\sqrt{1-6t+t^2}\over 1-7t+5t^2-t^3}
$$

\section{The relation to domino tilings}\label{aztec}
The solutions of the $A_\infty$ $T$-system,
also known as the octahedron equation \cite{RR,FZLaurent,KT},
\begin{equation}\label{Tsystem}
T_{i,j,k+1} T_{i,j,k-1} = T_{i,j+1,k}T_{i,j-1,k} - T_{i+1,j,k}T_{i-1,j,k}
\end{equation}
were given in \cite{SPY} in terms of the partition function for domino
tilings of the Aztec diamond and its generalizations. The $A_r$
$T$-system (the same equation, with additional boundary conditions) is
the fusion relation satisfied by the transfer matrix of the
generalized, inhomogeneous Heisenberg spin chain \cite{KR}.  The
$Q$-system \eqref{qsys} is a limit of the $T$-system, with the index
$j$ dropped, whereas the renormalized $Q$-system \eqref{renom} is
obtained by formally dropping the index $i$.  Therefore one should be
able to recover the solutions of the $Q$-system from those of the
$T$-system.

Here, we give the explicit connection between the path formulation of
$R_{\al,n}$ and domino tilings.  In
particular, we express $R_{\al,n}(\bx_\bM)$ for any Motzkin path
$\bM$, as partition functions for tilings of certain domains of the
square lattice by means of $2\times 1$ and $1\times 2$ dominos, and of
rigid ``defect" pairs of square tiles $1\times 1$.

\subsection{Paths and matchings of the Aztec diamond}\label{matchdiam}

In this subsection, we consider the solutions $R_{\al,n}$ as functions
of $\bx_{0}.$ For a restricted subset of the indices $(\al,n)$
there is an alternative combinatorial interpretation of
$R_{\al,n}$, related to the results of \cite{SPY} on the solutions of
the octahedron equation. This relation holds only for a restricted
subset, because the boundary conditions of \cite{SPY} are incompatible
with our truncation\footnote{This truncation is also different from
the truncation considered in \cite{AH} for the so-called bounded
octahedron recurrence.} $R_{0,n}=R_{r+1,n}=1$. The connection is therefore
valid only sufficiently far away from these boundaries.

Consider a system of equations with indices $(\al,n)\in \Z\times \Z$
(the $A_\infty$ $Q$-system):
\begin{equation}\label{octa}
\rho_{\al,n+1}\rho_{\al,n-1}=\rho_{\al,n}^2+\rho_{\al+1,n}\rho_{\al-1,n},
\ \al,n\in \Z 
\end{equation}
 To write $\rho_{\al,n}$ as a function of
$\{\rho_{\al,0},\rho_{\al,1}\}_{\al\in\Z}$ we use the recursion
relation
$$
\rho_{\al,n+1}={\rho_{\al,n}^2+\rho_{\al+1,n}\rho_{\al-1,n}\over
\rho_{\al,n-1} } .
$$
By induction, $\rho_{\al,n}$ depends only on the subset of the initial data
$\{\rho_{\beta,0},\rho_{\beta,1}\}$ with $\al-n+1\leq \beta \leq
\al+n-1$. Hence, if we identify $\rho_{\al,i}=R_{\al,i}$ ($i=0,1$), then
$\rho_{\al,n}=R_{\al,n}$ if
\begin{equation}\label{restrinal}
n\leq {\rm Min}(\al,r+1-\al).
\end{equation}
We will show that $R_{\al,n}$, with indices in the set, is
the generating function for positively
weighted matchings of the Aztec diamond.

Recall the interpretation of $\rho_{\al,n}$ in terms of partition
functions of matchings of Aztec diamonds \cite{SPY}.  The Aztec
diamond centered at $(\al,0)$ of radius $n$ is the set $ {\mathcal
A}_{\al,n} =\left\{ (\beta,\gamma)\in \Z^2, \, \vert \,
|\al-\beta|+|\gamma|< n\right\} $ with boundary $ {\mathcal B}_{\al,n}
=\left\{ (\beta,\gamma)\in \Z^2, \, \vert \, |\al-\beta|+|\gamma|=
n\right\} $ We denote by $A_{\al,n}$ and $B_{\al,n}$ the subsets of
the square lattice with vertices in $\Z^2+({1\over 2},{1\over 2})$
made of squares centered on the points of ${\mathcal A}_{\al,n}$ and
${\mathcal B}_{\al,n}$ respectively.

Consider the matchings, or compact dimer coverings, of
$A_{\al,n}$. These are configurations of occupation of edges
(including their vertices) of $A_{\al,n}$ by dimers, where each vertex
is covered by exactly one dimer.  Each square $(\beta,\gamma)\in
A_{\al,n}$ has either 2,1, or 0 edges covered, and we define
$\epsilon(\beta,\gamma)=-1,0,1$, respectively. Each square
$(\beta,\gamma)\in B_{\al,n}$ has either $0$ or $1$ edge covered,
and we  define $\epsilon(\beta,\gamma)=0,1$ respectively. Let
$$
\theta_{\al,n}(\beta,\gamma)=\left\{\begin{matrix} 0 & {\rm if}\,
\al+\beta+\gamma+n=0\, {\rm mod}\, 2\\ 
1 & {\rm otherwise} \end{matrix}
\right.
$$
Then the generating function for matchings of $A_{\al,n}$ is
\begin{equation}\label{match}
M_{\al,n}=\sum_{{\rm matchings}\, {\rm of}\, A_{\al,n}} 
\prod_{(\beta,\gamma)\in {\mathcal A}_{\al,n}\cup {\mathcal B}_{\al,n}} 
\Big(\rho_{\beta,\theta_{\al,n}(\beta,\gamma)}\Big)^{\epsilon(\beta,\gamma)}.
\end{equation}
The following is proved in \cite{SPY}:
\begin{thm}
The solution $\rho_{\al,n}$ of the system \eqref{octa}, expressed as a
function  of
 $\{\rho_{\al,0},\rho_{\al,1}\}_{\al\in\Z}$, is equal to 
the generating function for matchings of $A_{\al,n}$ \eqref{match}.
\end{thm}

\begin{figure}
\centering
\includegraphics[width=8.cm]{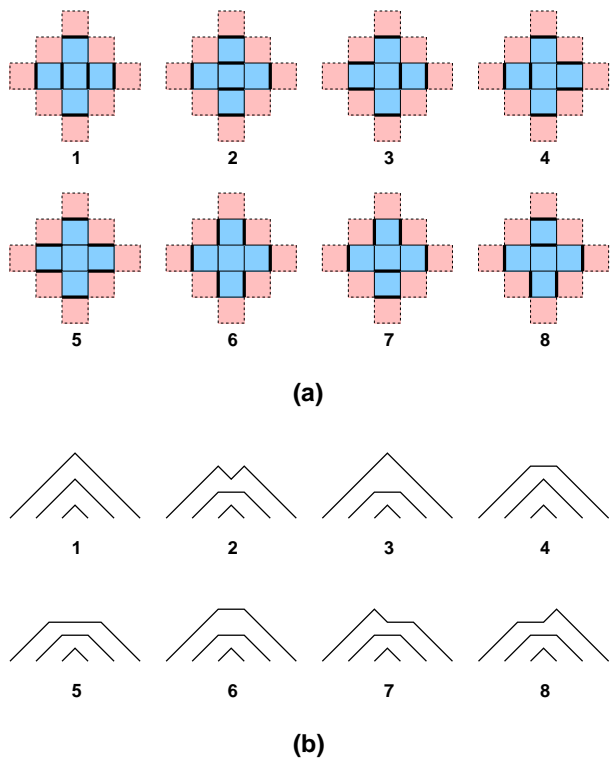}
\caption{\small (a) The $8$ matchings of $A_{\al,3}$
and (b) the non-intersecting $G_r$-paths from $(0,0),(2,0),(4,0)$ to 
$(6,0),(8,0),(10,0)$, where $r\geq 4$.
The labeling corresponds to the weights $a_i$, $i=1,2,...,8$.}\label{fig:exaztec}
\end{figure}

\begin{example}\label{aztex}
Figure \ref{fig:exaztec} shows, for $n=3$, (a) the eight matchings of
the Aztec diamond $A_{\al,3}$, centered at $(\al,0)$ after rotation by
$\pi/4$.  The corresponding weights of eq.\eqref{match} are (we denote
by $a_i$ the weight of the configuration labeled $i$ in
Fig.\ref{fig:exaztec}):
\begin{eqnarray*}
&&a_1={R_{\al-2,1}R_{\al,1}R_{\al+2,1}\over R_{\al-1,0}R_{\al+1,0}}, \ \
a_2={R_{\al-1,1}^2R_{\al+1,1}^2\over R_{\al,0}^2R_{\al,1}}, \ \
a_3={R_{\al-1,1}^2R_{\al+2,1}\over R_{\al-1,0}R_{\al+1,0}}, \\
&&a_4={ R_{\al-2,1}R_{\al+1,1}^2\over R_{\al-1,0}R_{\al+1,0}},\ \ 
a_5={R_{\al,1}^3\over R_{\al,0}^2}, \ \ 
a_6={R_{\al-1,1}^2R_{\al+1,1}^2\over R_{\al-1,0}R_{\al,0}R_{\al+1,0}}, \ \
a_7=a_8={R_{\al-1,1}R_{\al,1}R_{\al+1,1}\over R_{\al,0}^2}
\end{eqnarray*}
and we have $R_{\al,3}=\sum_{i=1}^8 a_i$. Note that this expression is valid only for $\al=2,3,...,r-1$,
$r\geq 3$, and
provided we set $R_{0,n}=R_{r+1,n}=1$ for all $n$.
This can be compared to our previous result for $R_{\al,3}$ as solution of the 
$A_r$ $Q$-system, obtained by cutting out
the initial data at $\beta=\al-3$, hence working instead with $A_{r'}$, $r'=r-\al+3$.
This allows to reinterpret $R_{\al,3}$ for $A_r$ as $R_{3,3}'$ for $A_{r-\al+3}$, with
initial data $R_{\al',i}'=R_{\al'+\al-3,i}$, $i=0,1$ and $\al'=1,2,...,r-\al+3$. As such it is
interpreted as $(R_{\al-2,0})^3$ ($=(R_{1,0}')^3$) times
the generating function for triples of non-intersecting $G_r$-paths
(for any $r\geq 4$) from the points
$(0,0),(2,0),(4,0)$ to the points $(6,0),(8,0),(10,0)$, but with the weights
$y_1'={R_{\al-2,1}\over R_{\al-2,0}}$ and 
$y_{2\beta -1}'= {R_{\beta+\al-3,1}R_{\beta+\al-4,0}\over R_{\beta+\al-3,0}R_{\beta+\al-4,1}}$ for $\beta\geq 2$
and $y_{2\beta}'={R_{\beta+\al-4,0}R_{\beta+\al-2,1}\over R_{\beta+\al-3,0}R_{\beta+\al-3,1}}$ for $\beta\geq 1$.
There are exactly $8$ such triples, and their weights (multiplied by $R_{\al-2,0}^3$)
match one by one the weights $a_i$ above.
They are represented in Fig.\ref{fig:exaztec} (b).
\end{example}

\subsection{Lattice paths and domino
tilings}\label{pathtil} 

\begin{figure}
\centering
\includegraphics[width=6.cm]{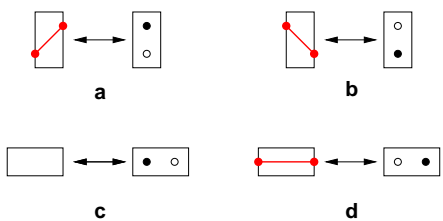}
\caption{\small A bijection between lattice paths and 
domino tilings. The dual square lattice is bicolored,
so there are 4 possible domino tiles,
labeled $a,b,c,d$.  Each corresponds to a step of path as indicated:
$a: (1,1)$, $b:(1,-1)$, $c:$ no step, and $d:(2,0)=(1,0)+(1,0)$,
double horizontal step.}\label{fig:dominobij}
\end{figure}

There is a standard bijection between domino tilings and $\wG_r$
lattice paths.  (see e.g. \cite{KJ}).  Consider the dominos as tiling
a chessboard. Then there are four types of
$2\times 1$ or $1\times 2$ dominos (see Figure
\ref{fig:dominobij}). These are in bijection with the
four path steps: $(1,1)$, $(1,-1)$, $(2,0)=(1,0)+(1,0)$ or no step at
all. Thus, 
a domino tiling of a domain $\mathcal D$ can be rephrased in terms of
a configuration of non-intersecting paths connecting points on the
boundary $\partial \mathcal D$.
\begin{figure}
\centering
\includegraphics[width=8.cm]{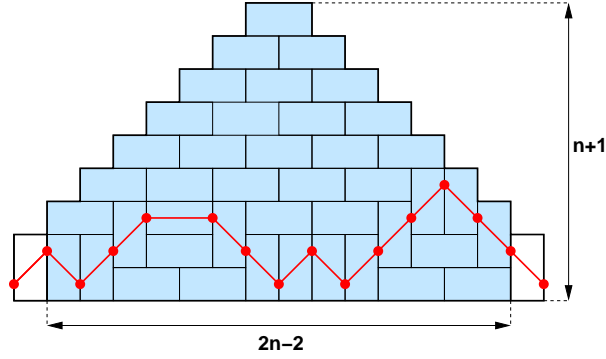}
\caption{\small The bijection between paths of length
$2n$ steps from $(0,0)$ to $(2n,0)$ and domino tilings of a ``halved"
Aztec diamond $HA_n$, of width $2n$ and height $n+1$.}
\label{fig:ronedomino}
\end{figure}
\begin{figure}
\centering
\includegraphics[width=10.cm]{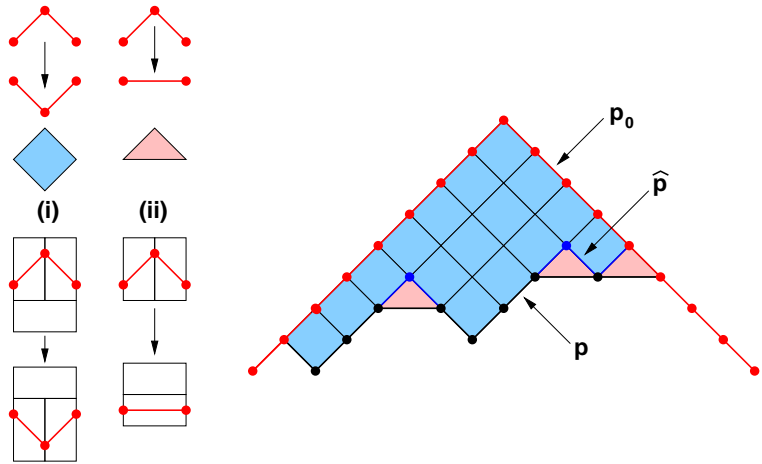}
\caption{\small The local moves (i-ii) which give
any Aztec path $p$ from the maximal path $p_0$, and an example of such a
transformation. First one goes from $p_0$ to the Dyck path $\hat{p}$
closest to $p$ by a sequence of ``box removals" (i), before getting to
$p$ via ``half-box removals" (ii).}
\label{fig:locaz}
\end{figure}
Conversely, consider $\wG_\infty$-lattice paths in the non-negative integer
quadrant $\Z_+^2$ from $(0,0)$ to $(2n,0)$ of length
$2n$. (We call these Aztec paths in this section.)
\begin{lemma}\label{onealema}
Aztec paths are in bijection with domino tilings of the ``halved"
Aztec diamond $HA_n$ represented in Figure \ref{fig:ronedomino}, of width
$2n$ and height $n+1$, with a floor at height $h=-1/2$, a
ceiling at height $h=n+1/2$, and a white face on the bottom left.
\end{lemma}
\begin{proof}
Any Aztec path may be obtained from the ``maximal'' Aztec path $p_0$
($n$ up steps followed by $n$ down steps) using the
procedure shown
in Figure \ref{fig:locaz}.


Given the path $p_0$, use the tiles $a,b,d$ of Figure
\ref{fig:dominobij} to tile the region traversed by the path. One must
then tile $HA_{n-1}'$, the half Aztec diamond below this region, with
width $2n-2$ and with the bottom row removed.  The bottom left and
right corner tiles are frozen, and are of type $c$, as are all the
tiles touching the left and right boundaries of $HA_{n-1}'$. By
induction, the entire domain $HA_{n-1}'$ must be tiled with $c$ tiles,
a unique configuration. Therefore there is a unique tiling of $HA_n$
associated to the maximal path $p_0$.

For any other path $p$, we use the sequence of moves in Figure
\ref{fig:locaz},
from $p_0$ to $p$, and apply the corresponding transformations on the
associated tiling by use of the bijection of Fig.\ref{fig:dominobij}.
This produces a unique tiling of $HA_n$ for each Aztec path $p$.

Conversely, given any tiling of $HA_n$, there is a unique path
associated to it via the map in Figure \ref{fig:dominobij}. A
path can only enter the domain from its lower left end ($1\times 2$
domino of type $a$ or $2\times 1$ domino of type $d$), as all faces
touching the left boundary are black, hence the tile $d$ cannot be
used. Similarly, the path can only exit via the lower right end
($1\times 2$ domino of type $b$ or $2\times 1$ domino of type $d$).
We deduce that there is a unique path associated to the tiling, that
goes from $(0,0)$ to $(0,2n)$.
\end{proof}

We assign weights to the steps of Aztec paths 
according to the height $h$ at which the step starts:
\begin{itemize}\item
Weight $1$ for steps of type $(1,1)$, irrespectively of $h$.
\item
Weight $z_{2h-2}$ for a step $(1,-1)$ from height $h$ to $h-1$, $h\geq 2$;
weight $z_1$ for a step $(1,-1)$ from $h=1$ to $h=0$.
\item weight $z_{2h-1}$ for a step $(2,0)$ at height $h$, $h\geq 2$;
weights $w_0,w_1$ for steps $(2,0)$ at heights $h=0$ and $h=1$, respectively.
\end{itemize}
Domino tilings receive weights according to the bijection of Figure
\ref{fig:dominobij}, with domino $c$ having weight $1$. We will refer
to those as weighted domino tilings in the following.

\subsection{$Q$-system solutions as functions of $\bx_0$ and domain tilings}
Recall the interpretation of $R_{1,n}$ as the partition function of
weighted paths on ${\tilde G}_r$.  These are weighted Aztec paths with
the following identification of weights: $z_{2r+1}=w_0=w_1=z_j=0$
$(j\geq 2r+3)$, $z_i=y_i$ $(i=1,2,...,2r)$ and $z_{2r+2}=y_{2r+1}$,
with $y_j$ as in \eqref{qtyzero}.  
The condition $z_j=0$ for $j\geq 2r+4$ corresponds
to truncating the tiled domain to the inside of a strip of height
$r+3$, with a ceiling at height $h=r+5/2$, with floor at height
$h=-1/2$. The conditions $w_0=w_1=z_{2r+1}=z_{2r+3}=0$ forbid the use
of the tile $d$ in the two bottom ($-1/2\leq h\leq 3/2$) and top ($r+1/2\leq
h\leq r+5/2$) rows.

\begin{figure}
\centering
\includegraphics[width=10.cm]{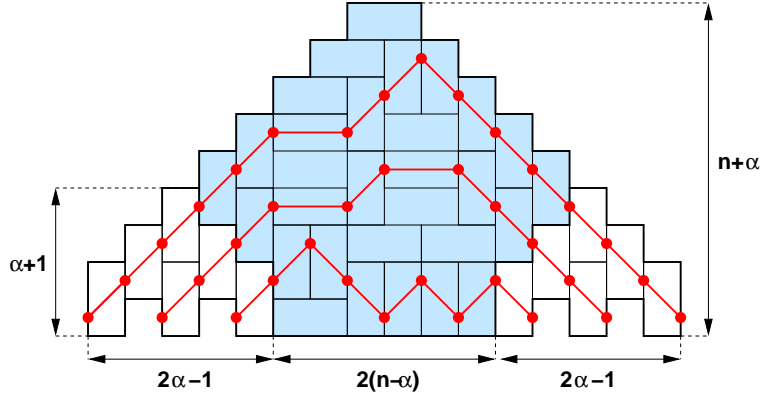}
\caption{\small The bijection between families of $\al$ Aztec paths on
$\Z_+^2$ from $(0,0),(2,0),...,(2\al-2,0)$ to
$(2n,0),(2n+2,0),...,2n+2\al-2,0)$ and the domino tilings of the
indented halved Aztec diamond $IHA_{n,\al}$ of width $2n+2\al-2$ and
height $n+\al$. Here, $\al=3$.}
\label{fig:ralphadomino}
\end{figure}

\begin{lemma}\label{twoalema}
The families of $\al$ non-intersecting Aztec paths on $\Z_+^2$
from $(0,0),(2,0),...,(2\al-2,0)$ to $(2n,0),(2n+2,0),...,2n+2\al-2,0)$, with steps $(1,1)$, $(1,-1)$ 
and double steps $(2,0)=(1,0)+(1,0)$, are in bijection with domino tilings of the ``indented halved" 
Aztec diamond $IHA_{n,\al}$ represented in Fig.\ref{fig:ralphadomino}. 
\end{lemma}
\begin{proof}
Along the same lines as for the $\al=1$ case of Lemma \ref{onealema}. 
We first associate a tiling to any configuration of $\al$ non-intersecting Aztec paths.
We simply use the fact that
any configuration of $\al$ non-intersecting Aztec paths may be obtained from the configuration
where all paths are maximal via successive
applications of the local moves of Fig.\ref{fig:locaz}. We apply the construction
of the proof of Lemma \ref{onealema} to the bottom-most path, then to the next bottom-most,
etc.

Conversely, to construct the path configurations from the tilings, we note that
the two successions
of $\al-1$ indentations of the bottom boundary (left and right) impose the presence of $\al-1$ tiles $a$ 
and one $a$ or $d$ on the left
and $\al-1$ tiles $b$ and one $b$ or $d$ on the right, and that paths can enter or exit nowhere
else on the boundary, due to coloring constraints. This
corresponds to exactly $\al$ paths touching the boundary of the domain, entirely determined by
the tiling.
\end{proof}
Given that $R_{\al,n}$ is the partition function of $\al$
non-intersecting Aztec paths, with the identification of weights as
above, this gives the relation to domino tilings.

For sufficiently large $r$, the restriction of the weights effectively
reduces the domain to be tiled to the shaded region in
Figure \ref{fig:ralphadomino}, as the tiling of the rest of the domain is
entirely fixed (only dominos $a$ on the left, and $b$ on the right).
The case $\al=1$ is also depicted with the same convention in
Fig.\ref{fig:ronedomino}.

In the particular case when $\al=n$, and $r\geq 2n-1$, the shaded
domain to be tiled is exactly the Aztec diamond itself (with width and
height $2n-2$), with no constraints.
It is interesting to compare the
partition function $D_n(\by)$ 
for the weighted tilings
of this domain to the partition function for weighted Aztec diamond matchings \eqref{match}. 
On one hand, we know that it is equal to
\begin{equation}
D_n(y_1,y_2,...,y_{2r+1})={R_{n,n}\over R_{1,0}^n y_1^n y_2^{n-1}y_4^{n-2}\cdots y_{2n-2}}
\end{equation}
where we have used the interpretation of $R_{n,n}/R_{1,0}^n$ as generating function
for ${\tilde G}_r$-paths, applied the path-tiling bijection of Section \ref{pathtil},
and removed the contributions of the tiles outside of the shaded domain.

As discussed in Sections \ref{positfinal} and \ref{matchdiam}, the solution $\rho_{\al,n}$ 
of the $A_\infty$ $Q$-system for pairs $\al,n$ 
obeying \eqref{restrinal} is identified with $R_{n,n}'$, solution of the $A_{r}$ $Q$-system for any $r\geq 2n-1$,
and with initial data $R_{\beta,0}'=R_{\al-n+\beta,0}$ and $R_{\beta,1}'=R_{\al-n+\beta,1}$ for $\beta=1,2,...,r$, 
namely with the weights (see also Example \ref{aztex}):
\begin{eqnarray*}
y_{2\beta -1}'&=& {R_{\beta+\al-n,1}R_{\beta+\al-n-1,0}\over R_{\beta+\al-n,0}R_{\beta+\al-n-1,1}} \ {\rm for}\  
\beta=1,2,...,r+1\\
y_{2\beta}'&=&{R_{\beta+\al-n-1,0}R_{\beta+\al-n+1,1}\over R_{\beta+\al-n,0}R_{\beta+\al-n,1}} \ {\rm for} \  
\beta=1,2,...,r\\
\end{eqnarray*}
So we get two different ways of computing the same partition function $\rho_{\al,n}$ for Aztec diamond matchings,
one via the expression \eqref{match}, the other as 
$R_{n,n}'$. 

As is apparent from the Example
\ref{aztex}, the weights leading to the expression \ref{match} and
those of the families of non-intersecting paths 
are in bijection, but are incompatible with our weighted
bijection. Indeed, if we rotate all the matchings of
Fig.\ref{fig:exaztec} (a) clockwise by a quarter-turn, we find that
the corresponding dual domino tilings match the configurations of
non-intersecting paths of Fig.\ref{fig:exaztec} (b) via our bijection
only in the cases $2,3,4,7,8$, while some permutation of $1,5,6$ must
be applied. This suggest perhaps that other natural bijections should
exist.

\subsection{Solutions of the $Q$-system as functions of $\bx_\bM$ 
and tilings of domains with defects} We now ask for a tiling
interpretation for the weighted $\Gamma_\bM$-paths, describing
$R_{\al,n}$ as a function of the 
seed data $\bx_\bM$.

Consider first $\bM_{max}$, with $m_\al=\al-1$, where $R_{1,n}$ is a
partition function of weighted Dyck paths on the strip $0\leq y \leq
2r+1$ from $(0,0)$ to $(2n,0)$. Comparing this with the situation
described in Lemma \ref{onealema}, we can make a direct connection by
forbiding the steps $(2,0)$, or by using tiles $a,b,c$ only.  Thus,
$R_{1,n}/R_{1,0}$ is the generating function for tilings
of the domain $HA_n$ by means of tiles $a,b,c$ only, and with weights
$1$ per tiles $a$ or $c$, and $y_i$ per tile $b$ with center at height
$i-1/2$, for $i=1,2,...,2r+1$.

More generally,
let
$\Gamma_\bM=T_{2r+2}(i_1,...,i_s)$, a skeleton tree as in Definition \ref{defn:trees}.
As before we choose the weights $w_0=w_1=0$,
$z_j=0$ for $j\geq 2(2r+1-s)-1$, $z_{2(2r-s)}=y_{2r+1}(\bM)$, and
$z_{2j-1}=0$ for all indices $j\in [2,2r-s-1]\setminus
\{i_1,i_2,...,i_s\}$, corresponding to missing horizontal edges at
height $j$ in $\Gamma_\bM$. The remaining $z_i$'s are identified with
$y_j(\bM)$, in increasing order of indices.  Again, the vanishing
conditions impose some truncation to a strip $h\in [-1/2,2r-s+3/2]$,
and forbid the use of tiles $d$ except at heights $i_1,i_2,...,i_s$.

Next, consider the case of the Motzkin path
$\bM(k):=\mu_1\mu_2...\mu_k (\bM_0)$, with $m_1=m_2=\cdots =m_k=1$,
and $m_j=0$ for $j>k$. The graph $\Gamma_{\bM(k)}$ is ${\tilde G}_r$
with one additional down-pointing edge $k+2\to k$, assuming that
$r\geq k+1$.

\begin{figure}
\centering
\includegraphics[width=10.cm]{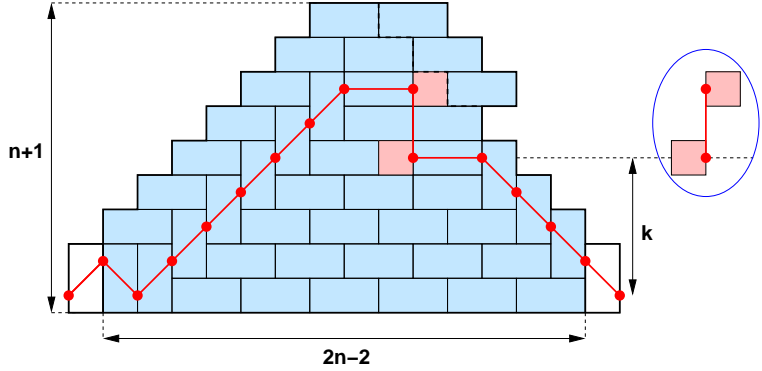}
\caption{\small The bijection between modified Aztec $A_{\bM(k)}$-paths
from $(0,0)$ to $(2n,0)$ and the
domino tilings of a domain made of the half Aztec diamond
(dashed lines), enhanced as indicated, tiled by
means of dominos and
rigid pairs of $1\times 1$ squares with centers at height $k+2$ and
$k$ (depicted in the medallion).}
\label{fig:mdomain}
\end{figure}

We define modified Aztec paths ($A_{\bM(k)}$-paths) as the
corresponding $\Gamma_{\bM(k)}$-paths in the $\Z_+^2$ quadrant
(without the usual restriction due to $r$ being finite).  It has an extra
step $(0,-2)$ only from height $h=k+2$ to height $h=k$. The weighting
is as in Section \ref{pathtil}, and the extra step receives a weight
$z_{k+2,k}$.  Paths from $(0,0)$ to $(2n,0)$ are in bijection with
tilings of the domain depicted in Fig.\ref{fig:mdomain}, by means of
the domino tiles of
Fig.\ref{fig:dominobij}, plus rigid pairs of $1\times 1$
defects tiles (see the medallion in Fig.\ref{fig:mdomain}), with centers
at positions $(t-1/2,k)$ and $(t+1/2,k+2)$ with $k+2\leq t\leq 2n-k$,
$t$ integer.  The ``maximal" path $p_0$ now involves a
succession of $n+1$ up steps $(1,1)$, followed by $n-k$ down steps
$(1,-1)$, then one down step $(0,-2)$ and finally $k$ down steps
$(1,-1)$. The local moves (i) and (ii) of Fig.\ref{fig:locaz} must now
be supplemented by new moves expressing the reduction of the special
down steps $(0,-2)$.

The partition function for weighted $A_{\bM(k)}$-paths is identified
with $R_{1,n}$ expressed as a function of
$\bx_{\bM(k)}$ (up to a multiplicative factor of $R_{1,0}$), choosing
the weights $w_0=w_1=z_{2k+1}=z_{2k+3}=0$,
$z_j=0$ for $j\geq 2r+4$, $z_{2k+2}=y_{2k+1}$, $z_j=y_j$ for
$j=1,2,...,2k$, and
$z_{k+2,k}=y_{k+2,k}=y_{2k}y_{2k+2}/y_{2k+1}$, where the $y_j\equiv
y_j({\bM(k)})$. Equivalently, $R_{1,n}/R_{1,0}$ is the partition
function for weighted tilings of the domain of
Figure \ref{fig:mdomain}, delimited by the tiles corresponding to the
maximal path $p_0$, and by the floor at height $h=-1/2$.

For an arbitrary Motzkin path $\bM$, the $A_\bM$ Aztec paths
are the $\Gamma_{\bM}$-paths without restriction due to $r$.  

\begin{figure}
\centering
\includegraphics[width=8.cm]{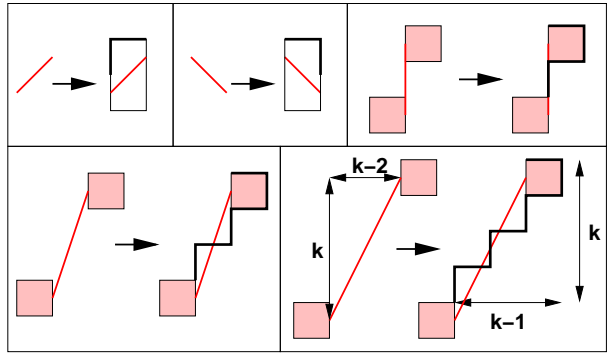}
\caption{\small The local rules for constructing the domain ${\mathcal D}_\bM^{(n)}$ 
out of the maximal $A_\bM$-path.}
\label{fig:rulbound}
\end{figure}

\begin{defn}
The domain ${\mathcal D}^{(n)}_\bM$ is defined as follows. Starting
from $\Gamma_\bM$, remove any down-pointing edges $i\to j$
if $\Gamma_\bM$ has an edge $i'\to j'$ with $[j,i]\subset[j',i']$
(strict inclusion). 
This leaves us with a set of
``maximal" down-pointing edges. Define the ``maximal" path
from $(0,0)$ to $(2n,0)$, which has no horizontal step, and goes as
far up and to the right as possible, namely starts with a maximal
number of steps $(1,1)$, descends via $n$ maximal descending steps.
The domain ${\mathcal D}^{(n)}_\bM$ is constructed by
associating boundary pieces to each step, as shown in
Figure \ref{fig:rulbound}. Finally, the domain is completed by a
horizontal line at height $h=-1/2$.
\end{defn}

\begin{lemma}
The $A_\bM$-paths from $(0,0)$ to $(2n,0)$ are in bijection with
tilings of the domain ${\mathcal D}^{(n)}_\bM$ by means
of the usual $2\times 1$ and $1\times 2$ tiles, plus rigid pairs of
square $1\times 1$ tiles centered at points of the form
$(t-{i-j-1\over 2},j)$ and $(t+{i-j-1\over 2},i)$ for all the pairs
$i,j$ of vertices connected by down-pointing edges on $\Gamma_\bM$.
\end{lemma}

\begin{figure}
\centering
\includegraphics[width=13.cm]{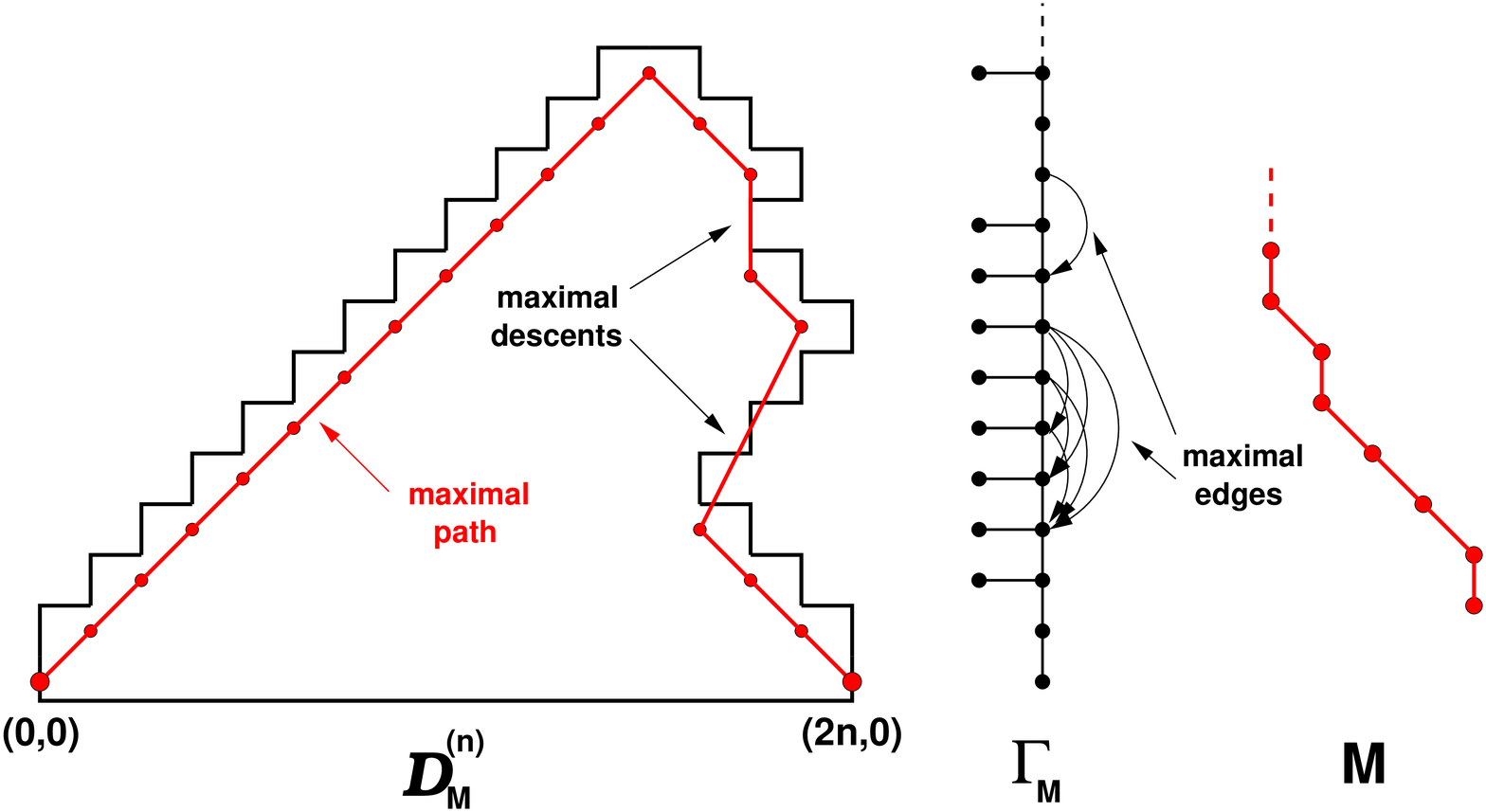}
\caption{\small A typical domain ${\mathcal D}_\bM^{(n)}$, with its
maximal $A_\bM$-path (in red), the target graph $\Gamma_\bM$
and the Motzkin path $\bM$.
We have indicated the two maximal edges, along which the maximal path takes its maximal descents.}
\label{fig:boundom}
\end{figure}

We refer to Figure \ref{fig:boundom} for an example.
We can now say that $R_{1,n+m_1}/R_{1,m_1}$ is the partition function
of tilings of ${\mathcal D}_\bM^{(n)}$ as described above.

\begin{figure}
\centering
\includegraphics[width=8.cm]{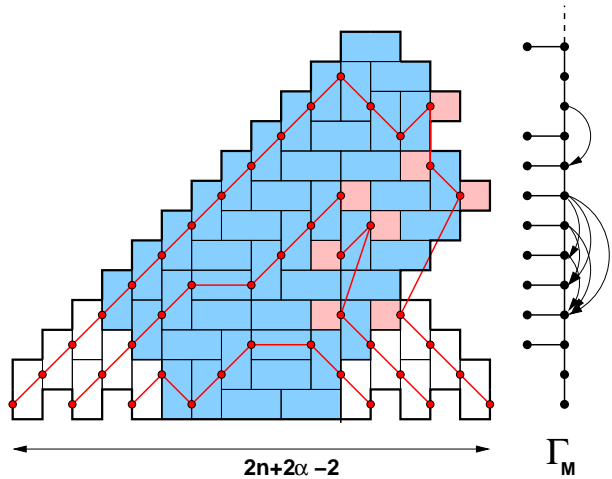}
\caption{\small The indented domain ${\mathcal D}_\bM^{(n)}$, serving for the representation
of $R_{\al,n}$ in terms of tilings, and the target graph $\Gamma_\bM$
(here $\al=3$).  Note that the tiling involves defect pairs
corresponding to descents of $\Gamma_\bM$, and that horizontal edges
(tiles $d$) are forbidden at heights where $\Gamma_\bM$ has no
horizontal edges. Outside of the shaded blue domain, all the tiles are
fixed by the indentations and the absence of $d$ tile in the bottom
row.}
\label{fig:indentdom}
\end{figure}

Finally, recall that $R_{\al,n+m_1}/(R_{1,m_1})^n$ is the generating
function for families of $\al$ strongly non-intersecting
$\Gamma_\bM$-paths, starting at $(0,0),...,(2\al-2,0)$ and ending at
$(2n,0),...,(2n+2\al-2,0)$. These paths can be represented as tilings
of the domain ${\mathcal D}_\bM^{(n)}$ with $\al-1$ additional
indentations on the bottom left and bottom right boundary (see
Fig.\ref{fig:indentdom}), by means of dominos and rigid defect pairs
corresponding to down-pointing edges of $\Gamma_\bM$. The
strong non-intersection of the paths imposes extra constraints on the
tilings, by forbiding some local configurations.





\section{Conclusion}\label{conc}


In this paper, we found a simple structural explanation for the
cluster mutations, which allow for sweeping the set of possible
initial data for the $Q$-system, in terms of simple local rearrangements
of the continued fractions that generate the $R_{1,n}$'s. This local
move allowing for generating mutations is reminiscent of the local
``Yang-Baxter"-like relation used in \cite{FZposit} in the context of
total positivity of the Grassmannian, as expressed through local
positive transfer matrices for networks. We do not yet fully undertand
this relation, although a partial explanation is found in \cite{DFK09}.

In view of Example \ref{tIpathweight}, we can think of the
constructions of this paper as generalizations of the Stieltjes
theorem \cite{STIEL1}, that now allow to re-express the series
$F(\lambda)=\sum_{k\geq 0}(-1)^k a_k/\lambda^{k+1}$ in different ways
as (mutated) possibly multiply branching continued fractions, whose
coefficients are particular combinations of Hankel determinants
involving the sequence $a_k$, each such rewriting corresponding to a
Motzkin path. To make the contact with our results, we simply have to
take $t=-1/\lambda$, $a_k=R_{1,k}$ and to identify
$R_{\al,n}=\Delta_\al^{n-\al+1}$ with the Hankel determinants of
eq.\eqref{hank}. For each Motzkin path $\bM$, we find a new continued
fraction expression for $F(\lambda)$ involving only the Hankel
determinants corresponding to the cluster variable $\bx_\bM$, via the
weights $y_i(\bM)$.







Equations called $Q$-systems exist for all simple Lie algebras. We
have checked that in these cases, $R_{1,n}$ also satisfies linear
recursion relations with constant coefficients.  We expect the
constructions of the present paper to generalize to all these cases.
In particular, we expect cluster positivity to be a consequence of the
LGV formula applied to counting possibly interacting families of
non-intersecting paths.

The $Q$-system is a specialization of the $T$-sytem, a discrete
integrable system with one additional parameter. It was shown in
\cite{DFK08} that $T$-systems can be considered as cluster algebras,
in general, of infinite rank. The statistical models in this paper are
particularly well-suited to the solution of the $T$-system \cite{DFK09a}.


In the case of the cluster variables corresponding to seeds of the
cluster algebra outside the subgraph $\mathcal G_r$, we do not have an
interpretation in terms of a statistical model. Outside this subgraph,
the exchange matrix $B$ has entries which grow arbitrarily. This
seems to suggest that the evolution in directions leading outside
$\mathcal G_r$ is not integrable.

\appendix
\section{The case of $A_3$ in the heap formulation.}\label{sectathree}

\begin{figure}
\centering
\includegraphics[width=14.cm]{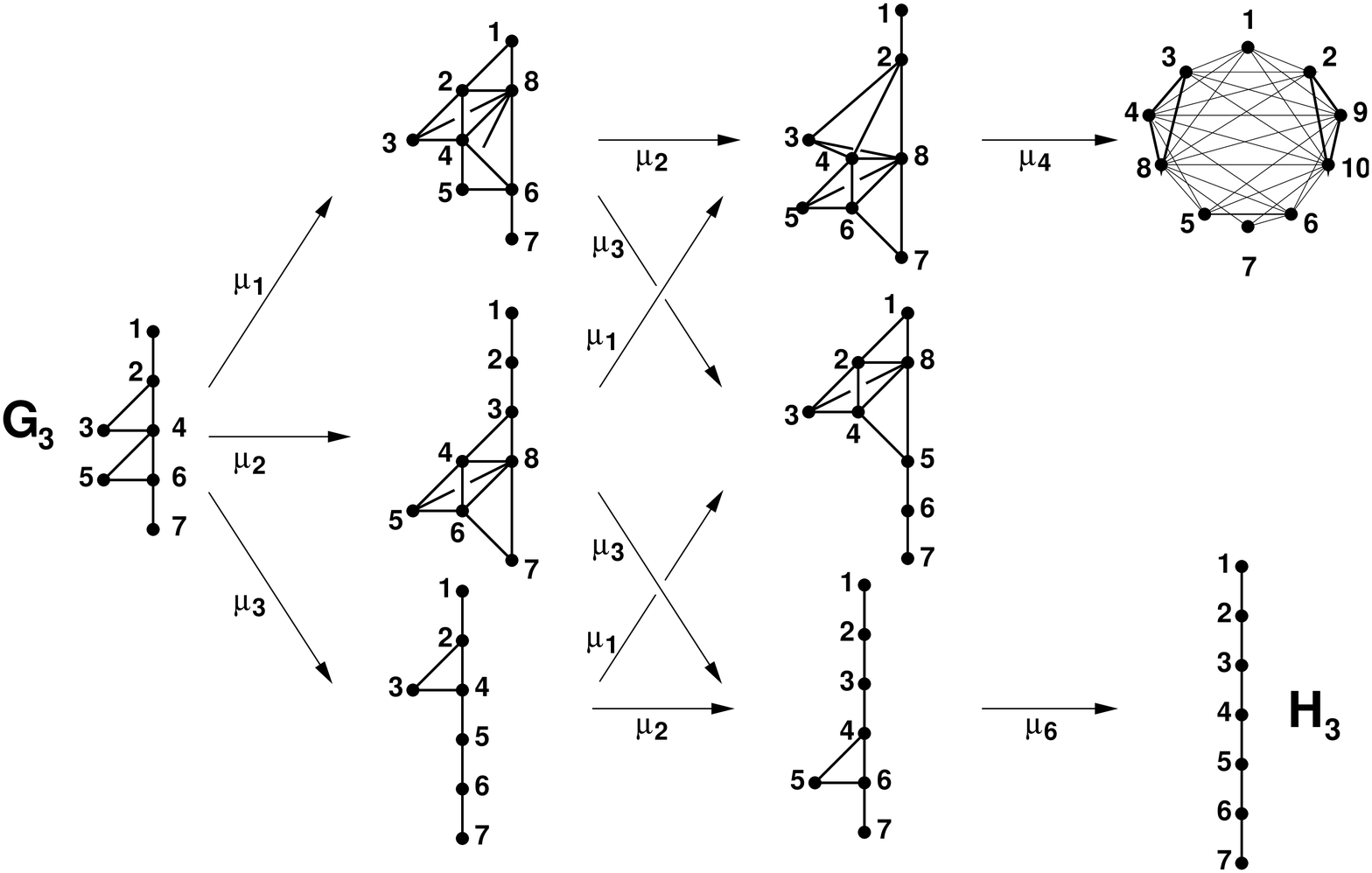}
\caption{\small The heap graphs corresponding to the $R_{1,n}$'s for the case $A_3$, 
and the corresponding mutations of seeds.}\label{fig:slfour}
\end{figure}

In this appendix, we detail the heap interpretation of the solution $R_{1,n}$ of the $A_3$ $Q$-system,
as expressed in terms of the various cluster variables corresponding to the 9 seeds of the fundamental domain.
We apply the same systematic method as in Example \ref{sectatwo}.
For simplicity, we have just depicted in Fig.\ref{fig:slfour} the heap graphs, together with the mutations
relating them, for each cluster variable of the 
fundamental domain. The corresponding rearrangements of the generating 
function $F_1^{(3)}(t)$ and new weights are given below,
with the labeling of vertices indicated in Fig.\ref{fig:slfour}.
These are straightforwardly generated by carefully
following the successive applications of Lemmas \ref{firstmov} and \ref{secmov}.

\noindent{\bf Initial cluster variable: $G_3$.} 
We start from the graph $G_3$ and the associated heap generating function
$F_1^{(3)}(t)=1+t y_1 \Phi^{(G_3)}(t)$, with
$$
\Phi^{(G_3)}(t)={1\over 1-t y_1-t{y_2\over 1-t y_3 -t{y_4\over 1-t y_5 -t{y_6\over 1-t y_7}}}}
$$
with the $y$'s as in \eqref{qtyzero}, namely
$$
y_1={R_{1,1}\over R_{1,0}},y_2={R_{2,1}\over R_{1,0}R_{1,1}},y_3={R_{1,0}R_{2,1}\over R_{2,0}R_{1,1}},
y_4={R_{1,0}R_{3,1}\over R_{2,0}R_{2,1}},
y_5={R_{2,0}R_{3,1}\over R_{3,0}R_{2,1}},y_6={R_{2,0}\over R_{3,0}R_{3,1}},y_7={R_{3,0}\over R_{3,1}} .
$$
Note that the weights are related via $y_1y_3y_5y_7=1$.

\noindent{\bf Mutation $\mu_1 (G_3)$.}
We apply Lemma \ref{secmov} to vertices $1$, $2$ and the structure attached to $3$. This gives
$F_1^{(3)}(t)=1+t y_1 \Phi^{(\mu_1 (G_3))}(t)$, with
$$
\Phi^{(\mu_1(G_3))}(t)={1\over 1-t {z_1\over 1-t{z_2 +{z_8\over 1-t z_5 
-t{z_6\over 1-t z_7}} \over 1-t z_3-t{z_4\over 1-t z_5 -t{z_6\over 1-t z_7}}}}}
$$
with 
$$
z_1=y_1+y_2,  z_2={y_2 y_3\over y_1+y_2}, 
z_8={y_2 y_4\over y_1+y_2},  z_3={y_1 y_3\over y_1+y_2},  z_4={y_1 y_4\over y_1+y_2},
z_5=y_5, z_6=y_6, z_7=y_7
$$
namely, upon using the $Q$-system relation $R_{1,0}R_{1,2}=R_{1,1}^2+R_{2,1}$:
\begin{eqnarray*}
&&z_1={R_{1,2}\over R_{1,1}},z_2={R_{2,1}^2\over R_{2,0}R_{1,1}R_{1,2}},
z_3={R_{1,1}R_{2,1}\over R_{2,0}R_{1,2}},
z_4={R_{1,1}^2R_{3,1}\over R_{2,0}R_{2,1}R_{1,2}},\\
&&z_5={R_{2,0}R_{3,1}\over R_{3,0}R_{2,1}},
z_6={R_{2,0}\over R_{3,0}R_{3,1}},z_7 ={R_{3,0}\over R_{3,1}},
z_8={R_{3,1}\over R_{2,0}R_{1,2}} .
\end{eqnarray*}
Note that there is one more $z$ than $y$'s, but that we now have two relations
$z_1z_3z_5z_7=1$ and $z_3z_8=z_2z_4$.

\noindent{\bf Mutation $\mu_2 (G_3)$.}
We apply Lemma \ref{secmov} to vertices $3$, $4$ and the descendent
structure attached to $4$. This gives
$F_1^{(3)}(t)=1+t y_1 \Phi^{(\mu_2 (G_3))}(t)$, with
$$
\Phi^{(\mu_2 (G_3))}(t)={1\over 1-t x_1-t {x_2\over 1-t {x_3\over 1-t{x_4+{x_8\over 1-t x_7}\over 
1-t x_5-t {x_6\over 1-t x_7}}}}}
$$
with
$$
x_1=y_1,x_2=y_2,x_3=y_3+y_4,x_4={y_4y_5\over y_3+y_4},x_8={y_4y_6\over y_3+y_4},
x_5={y_3y_5\over y_3+y_4},x_6={y_3y_6\over y_3+y_4},x_7=y_7 .
$$
namely, upon using the $Q$-system relation $R_{2,0}R_{2,2}=R_{2,1}^2+R_{1,1}R_{3,1}$:
\begin{eqnarray*}
&&x_1={R_{1,1}\over R_{1,0}},x_2={R_{2,1}\over R_{1,0}R_{1,1}},
x_3={R_{1,0}R_{2,2}\over R_{1,1}R_{2,1}},x_4={R_{1,1}R_{3,1}^2\over R_{3,0}R_{2,1}R_{2,2}},\\
&&x_5={R_{2,1}R_{3,1}\over R_{3,0}R_{2,2}} ,
x_6={R_{2,1}^2\over R_{3,0}R_{3,1}R_{2,2}},x_7={R_{3,0}\over R_{3,1}},x_8={R_{1,1}\over R_{3,0}R_{2,2}} .
\end{eqnarray*}
As for the $z$'s above, there is one more $x$ than $y$'s, but they obey two relations:
$x_1x_3x_5x_7=1$ and $x_5x_8=x_4x_6$.

\noindent{\bf Mutation $\mu_3 (G_3)$.}
This is part of the sequence of graphs of Example \ref{gpathex}. It is obtained by applying the 
Lemma \ref{secmov} to the vertices $5,6,7$ of $G_3$. This gives
$F_1^{(3)}(t)=1+t y_1 \Phi^{(\mu_3 (G_3))}(t)$, with
$$
\Phi^{(\mu_3 G_3)}(t)={1\over 1-t  \, t_1 -t{t_2\over 1-t\, t_3 -t{t_4 \over 1-t {t_5\over 1-t {t_6\over 1-t \, t_7}}}}}
$$
with
$$
t_1=y_1,t_2=y_2,t_3=y_3,t_4=y_4,t_5=y_5+y_6,t_6={y_6y_7\over y_5+y_6},t_7={y_5y_7\over y_5+y_6} .
$$
namely, upon using the $Q$-system relation $R_{3,0}R_{3,2}=R_{3,1}^2+R_{2,1}$:
$$
t_1={R_{1,1}\over R_{1,0}},t_2={R_{2,1}\over R_{1,0}R_{1,1}},
t_3={R_{1,0}R_{2,1}\over R_{2,0}R_{1,1}},t_4={R_{1,0}R_{3,1}\over R_{2,0}R_{2,1}},
t_5={R_{2,0}R_{3,2}\over R_{2,1}R_{3,1}},t_6={R_{2,1}\over R_{3,1}R_{3,2}},t_7={R_{3,1}\over R_{3,2}} .
$$
Note that the $t$ satisfy the relation $t_1t_3t_5t_7=1$.

\noindent{\bf Mutation $\mu_1 \mu_2 (G_3)$.}
We apply Lemma \ref{secmov} to the vertices $1$, $2$, and the substructure attached to $3$
in the graph $\mu_2 (G_3)$ (see Fig.\ref{fig:slfour}). This gives 
$\Phi^{(\mu_1\mu_2 (G_3))}(t)=\Phi^{(\mu_2 (G_3))}(t)$, with
\begin{equation}\label{muonetwo}
\Phi^{(\mu_1\mu_2 (G_3))}(t)={1\over 1-t {u_1\over 1-t {u_2\over 1-t u_3
-t{u_4+{u_8\over 1-t u_7}\over 1-t u_5-t {u_6\over 1-t u_7}}}}}
\end{equation}
with
$$
u_1=x_1+x_2,u_2={x_2x_3\over x_1+x_2},u_3={x_1 x_3\over x_1+x_2},
u_4=x_4,u_5=x_5,u_6=x_6,u_7=x_7,u_8=x_8 .
$$
namely, upon using the $Q$-system relation $R_{1,0}R_{1,2}=R_{1,1}^2+R_{2,1}$:
\begin{eqnarray*}
&&u_1={R_{1,2}\over R_{1,1}},u_2={R_{2,2}\over R_{1,1}R_{1,2}},
u_3={R_{1,1}R_{2,2}\over R_{1,2}R_{2,1}},
u_4={R_{1,1}R_{3,1}^2\over R_{3,0}R_{2,1}R_{2,2}}\\
&&u_5={R_{2,1}R_{3,1}\over R_{3,0}R_{2,2}},u_6={R_{2,1}^2\over R_{3,0}R_{3,1}R_{2,2}},
u_7={R_{3,0}\over R_{3,1}},u_8={R_{1,1}\over R_{3,0}R_{2,2}} .
\end{eqnarray*}
Note that the $u$'s satisfy the two relations $u_1u_3u_5u_7=1$ and $u_5u_8=u_4u_6$.
It is instructive to see how to arrive at the same result from the sequence of mutations
$\mu_2\mu_1(G_3)=\mu_1\mu_2(G_3)$. We now start from the graph $\mu_1 (G_3)$,
and apply the Lemma \ref{secmov} with (i) the vertex 2 and its substructure, (ii) the vertex
8 and its substructure common with that of vertex 2
(iii) its attached substructure not common with that of the vertex 2, namely with
\begin{eqnarray*}
a&=&{z_2\over 1-t z_3-t {z_4 \over 1-t z_5 -t {z_6\over 1-t z_7}}}\\
b&=&{z_8\over 1-t z_3-t {z_4 \over 1-t z_5 -t {z_6\over 1-t z_7}}}\\
c&=&z_5+{z_6\over 1-t z_7}\\
\end{eqnarray*} 
The result is
\begin{equation}\label{mutoone}
\Phi^{(\mu_2\mu_1 (G_3))}(t)={1\over 1-t {u_1'\over 1-t {u_2'\over 
\left(1-t z_3-t {z_4 \over 1-t z_5 -t {z_6\over 1-t z_7}} \right)\left(
1-t{u_4'+{u_8'\over 1-t u_7'}\over 1-t u_5'-t {u_6'\over 1-t u_7'}}\right)}}}
\end{equation}
with
\begin{eqnarray*}
&&u_1'=z_1=y_1+y_2=u_1,u_2'=z_2+z_8=y_2{y_3+y_4\over y_1+y_2}={x_2x_3\over x_1+x_2}=u_2,\\
&&u_4'={z_8z_5\over z_2+z_8}={y_4y_5\over y_3+y_4}=x_4=u_4,
u_8'={z_8z_6\over z_2+z_8}={y_4y_6\over y_3+y_4}=x_8=u_8,\\
&&u_5'={z_2z_5\over z_2+z_8}=x_5=u_5,
u_6'={z_2z_6\over z_2+z_8}={y_3y_6\over y_3+y_4}=x_6=u_6,u_7'=z_7=x_7=u_7 ,
\end{eqnarray*}
where we have used the expressions of the $z$'s and the $x$'s in terms
of the $y$'s. Applying again 
Lemma \ref{secmov} to the first factor under $u_2'$, with
$a=z_3$, $b=z_4$, $c=z_5+z_6/(1-t z_7)$, we may rewrite:
\begin{equation}\label{substru}
z_3+{z_4 \over 1-t z_5 -t {z_6\over 1-t z_7}}={u_3'\over 1-t 
{u_4''\over 1-t u_5''-t{u_6''\over 1-t u_7''}}}
\end{equation}
with
\begin{eqnarray*}
&&u_3'=z_3+z_4=y_1{y_3+y_4\over y_1+y_2}={x_1x_3\over x_1+x_2}=u_3,
u_4''={z_4z_5\over z_3+z_4}={z_8z_5\over z_2+z_8}=u_4'=u_4,\\
&&u_5''={z_3z_5\over z_3+z_4}={z_2z_5\over z_2+z_8}=u_5'=u_5,
u_7''=z_7=u_7 ,
\end{eqnarray*}
where we have used the relation $z_3/z_4=z_2/z_8$ between the $z$'s.
Substituting this into \eqref{mutoone}, we finally recover eq.\eqref{muonetwo}. 
This case is made particularly cumbersome
by the fact that the vertex 2 contains a substructure 
(which has not yet been the case so far), encoded in the fraction \eqref{substru}.
From this example, we learn that the mutation $\mu_2$ consists actually of two simultaneous 
applications of the Lemma \ref{secmov}, one ``usual" and one within the substructure attached
to the initial vertex 2.

\noindent{\bf Mutation $\mu_1 \mu_3 (G_3)$.}
We apply Lemma \ref{secmov} to the vertices $1$, $2$, and the substructure attached to $2$
in the graph $\mu_3 (G_3)$ (see Fig.\ref{fig:slfour}). This gives 
$\Phi^{(\mu_1\mu_3 (G_3))}(t)=\Phi^{(\mu_3 (G_3))}(t)$, with
\begin{equation}\label{muonethree}
\Phi^{(\mu_1\mu_3 (G_3))}(t)={1\over 1-t {w_1\over 1-t\left({w_2+{w_8\over 1-t {w_5\over 1-t {w_6
\over 1-t w_7}}}\over 1-t w_3-t {w_4\over  1-t {w_5\over 1-t {w_6
\over 1-t w_7}}}}\right)}}
\end{equation}
with
$$
w_1=t_1+t_2,w_2={t_2 t_3\over t_1+t_2},w_8={t_2 t_4 \over t_1+t_2},w_3={t_1t_3\over t_1+t_2},
w_4={t_1t_4\over t_1+t_2},w_5=t_5,w_6=t_6,w_7=t_7 .
$$
namely, upon using the $Q$-system relation $R_{1,0}R_{1,2}=R_{1,1}^2+R_{2,1}$:
\begin{eqnarray*}
&&w_1={R_{1,2}\over R_{1,1}},w_2={R_{2,1}^2\over R_{2,0}R_{1,1}R_{1,2}},
w_3={R_{3,1}\over R_{1,2}R_{2,0}},w_4={R_{1,1}^2R_{3,1}\over R_{2,0}R_{2,1}R_{1,2}},\\
&&w_5={R_{2,0}R_{3,2}\over R_{2,1}R_{3,1}},w_6={R_{2,1}\over R_{3,1}R_{3,2}},
w_7={R_{3,1}\over R_{3,2}},w_8={R_{3,1}\over R_{2,0}R_{1,2}} .
\end{eqnarray*}
Note the relations: $w_1w_3w_5w_7=1$ and $w_2w_4=w_3w_8$. For completeness, let us compute
the function $\Phi^{(\mu_3\mu_1(G_3))}=\Phi^{(\mu_1 (G_3))}$, and check that 
indeed $\mu_1\mu_3(G_3)=\mu_3\mu_1(G_3)$ yield the same result. 
We now apply the Lemma \ref{secmov}  with the vertices 5, 6 and 7 of the graph $\mu_1(G_3)$,
with respectively $a=z_5$, $b=z_6$ and $c=z_7$, while $d=0$.
This gives 
\begin{equation}\label{muthreeone}
\Phi^{(\mu_3\mu_1(G_3))}={1\over 1-t {w_1'\over 1-t{w_2' +{w_8'\over 
1-t {w_5'\over 1-t{w_6'\over 1-t w_7'}}} \over 1-t w_3'-t{w_4'\over 
1-t {w_5'\over 1-t{w_6'\over 1-t w_7'}}}}}}
\end{equation}
with
\begin{eqnarray*}
&&w_1'=z_1=y_1+y_2=t_1+t_2=w_1,\ \ w_2'=z_2={y_2y_3\over y_1+y_2}={t_2 t_3\over t_1+t_2}=w_2 \\
&&w_8'=z_8={y_2 y_4\over y_1+y_2}={t_2t_4\over t_1+t_2}=w_8,\ \ 
w_3'=z_3={y_1 y_3\over y_1+y_2}={t_1t_3\over t_1+t_2}=w_3, \\
&&w_4'=z_4={y_1 y_4\over y_1+y_2}={t_1t_4\over t_1+t_2}=w_4, \ \ 
w_5'=z_5+z_6=y_5+y_6=t_5=w_5, \\
&&w_6'={z_6 z_7\over z_5+z_6}={y_6y_7\over y_5+y_6}=t_6=w_6,\ \ 
w_7'={z_5z_7\over z_5+z_6}={y_5y_7\over y_5+y_6}=t_7=w_7,
\end{eqnarray*}
where the various identifications are made
by use of the expressions of the $z$'s and $t$'s in terms of the $y$'s. 
We deduce that the expressions \eqref{muonethree} and \eqref{muthreeone} are identical.

\noindent{\bf Mutation $\mu_2 \mu_3 (G_3)$.}
This is part of the sequence of graphs of Example \ref{gpathex}. It is obtained by applying the 
Lemma \ref{secmov} to the vertices $3,4$ and the structure attached to $5$
of the graph $\mu_3 (G_3)$ (see Fig.\ref{fig:slfour}). This gives
$\Phi^{(\mu_2\mu_3 (G_3))}(t)=\Phi^{(\mu_3 (G_3))}(t)$, with
\begin{equation}\label{mutwothree}
\Phi^{(\mu_2\mu_3 (G_3))}(t)={1\over 1-t s_1 -t {s_2\over 1-t {s_3 \over 1-t 
{s_4 \over 1-t s_5-t {s_6\over 1-t s_7}}}}}
\end{equation}
with
$$
s_1=t_1,s_2=t_2,s_3=t_3+t_4,s_4={t_4t_5\over t_3+t_4},
s_5={t_3 t_5\over t_3+t_4},s_6=t_6,s_7=t_7 .
$$
namely, upon using the $Q$-system relation $R_{2,0}R_{2,2}=R_{2,1}^2+R_{1,1}R_{3,1}$:
$$
s_1={R_{1,1}\over R_{1,0}},s_2={R_{2,1}\over R_{1,0}R_{1,1}},s_3={R_{1,0}R_{2,2}\over R_{1,1}R_{2,1}},
s_4={R_{1,1}R_{3,2}\over R_{2,1}R_{2,2}},s_5={R_{2,1}R_{3,2}\over R_{3,1}R_{2,2}},
s_6={R_{2,1}\over R_{3,1}R_{3,2}},s_7={R_{3,1}\over R_{3,2}} .
$$
Note the relation $s_1s_3s_5s_7=1$. For completeness, let us compute
the function $\Phi^{(\mu_3\mu_2 (G_3))}=\Phi^{(\mu_2 (G_3))}$, and check that 
indeed $\mu_2\mu_3 (G_3)=\mu_3\mu_2 (G_3)$ yield the same result. 
We now apply the Lemma \ref{secmov}  on the graph $\mu_2(G_3)$,
with (i) the vertex 4 and its substructure (ii) the vertex 
8 and its substructure
common to vertex 4, and the substructure of vertex $8$ not common to vertex 4. 
This amounts to taking $a=x_4/(1-tx_5-t x_6/(1-t x_7))$, $b=x_8/(1-tx_5-t x_6/(1-t x_7))$ 
and $c=x_7$, while $d=0$.
This yields 
\begin{equation}\label{muthreetwo}
\Phi^{(\mu_3\mu_2 (G_3))}={1\over 1-t s_1'-t{s_2'\over 1-t {s_3'\over
1-t {s_4'\over \left(1-t x_5-t {x_6\over 1-t x_7}\right)\left(1-t {s_6'\over 1-t s_7'} \right)}}}}
\end{equation}
with
\begin{eqnarray*}
&&s_1'=x_1=y_1=t_1=s_1, \ s_2'=x_2=y_2=t_2=s_2, \ s_3'=x_3=y_3+y_4=t_3+t_4=s_3,\\
&&s_4'=x_4+x_8=y_4{y_5+y_6\over y_3+y_4}={t_4t_5\over t_3+t_4}=s_4, \ 
s_6'={x_7 x_8\over x_4+x_8}={y_6y_7\over y_5+y_6}=t_6=s_6 , \\ 
&&s_7'={x_7 x_4\over x_4+x_8}={y_5y_7\over y_5+y_6}=t_7=s_7 ,
\end{eqnarray*}
where identifications are made by expressing the $x$'s, $s$'s and $t$'s in terms of the $y$'s.
As in the case of $\mu_2\mu_1 (G_3)$, we see that we must still apply the Lemma
\ref{secmov} to the first factor under $s_4'$ in \eqref{muthreetwo}, which corresponds to the
substructure attached to the initial vertex 4. This allows to rewrite
$$
x_5+{x_6\over 1-t x_7}= {s_5''\over 1-t {s_6''\over 1-t s_7''}}
$$
with
\begin{eqnarray*}
&&s_5''=x_5+x_6=y_3{y_5+y_6\over y_3+y_4}={t_3 t_5\over t_3+t_4}=s_5, \ \ 
s_6''={x_6x_7\over x_5+x_6}={y_6y_7\over y_5+y_6}=t_6=s_6\\
&&s_7''={x_5x_7\over x_5+x_6}={y_5y_7\over y_5+y_6}=t_7=s_7 .
\end{eqnarray*}
Substituting these into \eqref{muthreetwo} allows to identify it with \eqref{mutwothree}.

\noindent{\bf Mutation $\mu_6 \mu_2 \mu_3 (G_3)=H_3$.}
This is part of the sequence of graphs of Example \ref{gpathex}. It is obtained by applying the 
Lemma \ref{secmov} to the vertices $5,6,7$ of the graph $\mu_2\mu_3 (G_3)$ (see Fig.\ref{fig:slfour}). 
This gives
$\Phi^{(\mu_6\mu_2\mu_3 (G_3))}(t)=\Phi^{(\mu_2\mu_3( G_3))}(t)$, with
$$
\Phi^{(\mu_6\mu_2\mu_3 (G_3))}(t)={1\over 1-t r_1-t {r_2\over 1-t {r_3 \over 1-t 
{r_4 \over 1-t {r_5\over 1-t {r_6\over 1-t r_7}}}}}}
$$
with
$$
r_1=s_1,r_2=s_2,r_3=s_3,r_4=s_4,r_5=s_5+s_6,
r_6={s_6s_7\over s_5+s_6},r_7={s_5s_7\over s_5+s_6} .
$$
namely, upon using the $Q$-system relation $R_{3,1}R_{3,3}=R_{3,2}^2+R_{2,2}$:
$$
r_1={R_{1,1}\over R_{1,0}},r_2={R_{2,1}\over R_{1,0}R_{1,1}},r_3={R_{1,0}R_{2,2}\over R_{1,1}R_{2,1}},
r_4={R_{1,1}R_{3,2}\over R_{2,1}R_{2,2}},r_5={R_{2,1}R_{3,3}\over R_{2,2}R_{3,2}},
r_6={R_{2,2}\over R_{3,2}R_{3,3}},r_7={R_{3,2}\over R_{3,3}} .
$$
Note the relation $r_1r_3r_5r_7=1$.

\noindent{\bf Mutation $\mu_4 \mu_1 \mu_2 (G_3)$.} 
It is obtained by applying the 
Lemma \ref{secmov} to the vertices $1,2$ and the structure attached to $2$
of the graph $\mu_1\mu_2 (G_3)$ (see Fig.\ref{fig:slfour}). 
Writing $\Phi^{(\mu_1\mu_2 (G_3))}(t)=1+t u_1 \Psi^{(\mu_1\mu_2 (G_3))}(t)$, we have
$\Psi^{(\mu_4\mu_1\mu_2 (G_3))}(t)=\Psi^{(\mu_1\mu_2 (G_3))}(t)$, with
\begin{equation}\label{mostricky}
\Psi^{(\mu_4\mu_1\mu_2 (G_3))}(t)={1\over 1-t {v_1\over 1-t\left( {v_2+{v_9+{v_{10}\over 1-t v_7}\over
1-t v_5 -t {v_6\over 1-t v_7}}\over 1-t 
\left( v_3 +{v_4+{v_8\over 1-t v_7}\over 1-t v_5-t {v_6\over 1-t v_7}} \right) } \right)}}
\end{equation}
with
\begin{eqnarray*}
&&v_1=u_1+u_2,v_2={u_2u_3\over u_1+u_2},v_9={u_2 u_4\over u_1+u_2},
v_{10}={u_2u_8\over u_1+u_2}, v_3={u_1 u_3\over u_1+u_2},\\
&&v_4={u_1 u_4\over u_1+u_2},
v_5=u_5,v_6=u_6,v_7=u_7,v_8={u_1u_8\over u_1+u_2} .
\end{eqnarray*}
namely, upon using the $Q$-system relation $R_{1,1}R_{1,3}=R_{1,2}^2+R_{2,2}$:
\begin{eqnarray*}
&&v_1={R_{1,3}\over R_{1,2}},v_2={R_{2,2}^2\over R_{2,1}R_{1,2}R_{1,3}},
v_3={R_{1,2}R_{2,2}\over R_{2,1}R_{1,3}}, v_4={R_{1,2}^2R_{3,1}^2\over R_{3,0}R_{2,1}R_{2,2}R_{1,3}},
v_5={R_{2,1}R_{3,1}\over R_{3,0}R_{2,2}}\\
&&v_6={R_{2,1}^2\over R_{3,0}R_{3,1}R_{2,2}},v_7={R_{3,0}\over R_{3,1}},
v_8={R_{1,2}^2\over R_{3,0}R_{2,2}R_{1,3}},v_9={R_{3,1}^2\over R_{3,0}R_{2,1}R_{1,3}},
v_{10}={1\over R_{3,0}R_{1,3}} .
\end{eqnarray*}
Note that as we have 3 more $v$'s than $y$'s, we expect four relations between them.
These read:  $v_1v_3v_5v_7=1$, $v_2v_4=v_3v_9$, $v_8v_9=v_4v_{10}$ and
$v_2v_8=v_3v_{10}$.

We conclude that for $r=3$, $R_{1,n}$ is a
positive Laurent polynomial of all the mutations of the initial data $\bx_0$ in the fundamental
domain ${\mathcal F}_3$.


\label{appendixa}
\section{The case of $A_3$ in the path formulation.}\label{sectafour}

\begin{figure}
\centering
\includegraphics[width=13.cm]{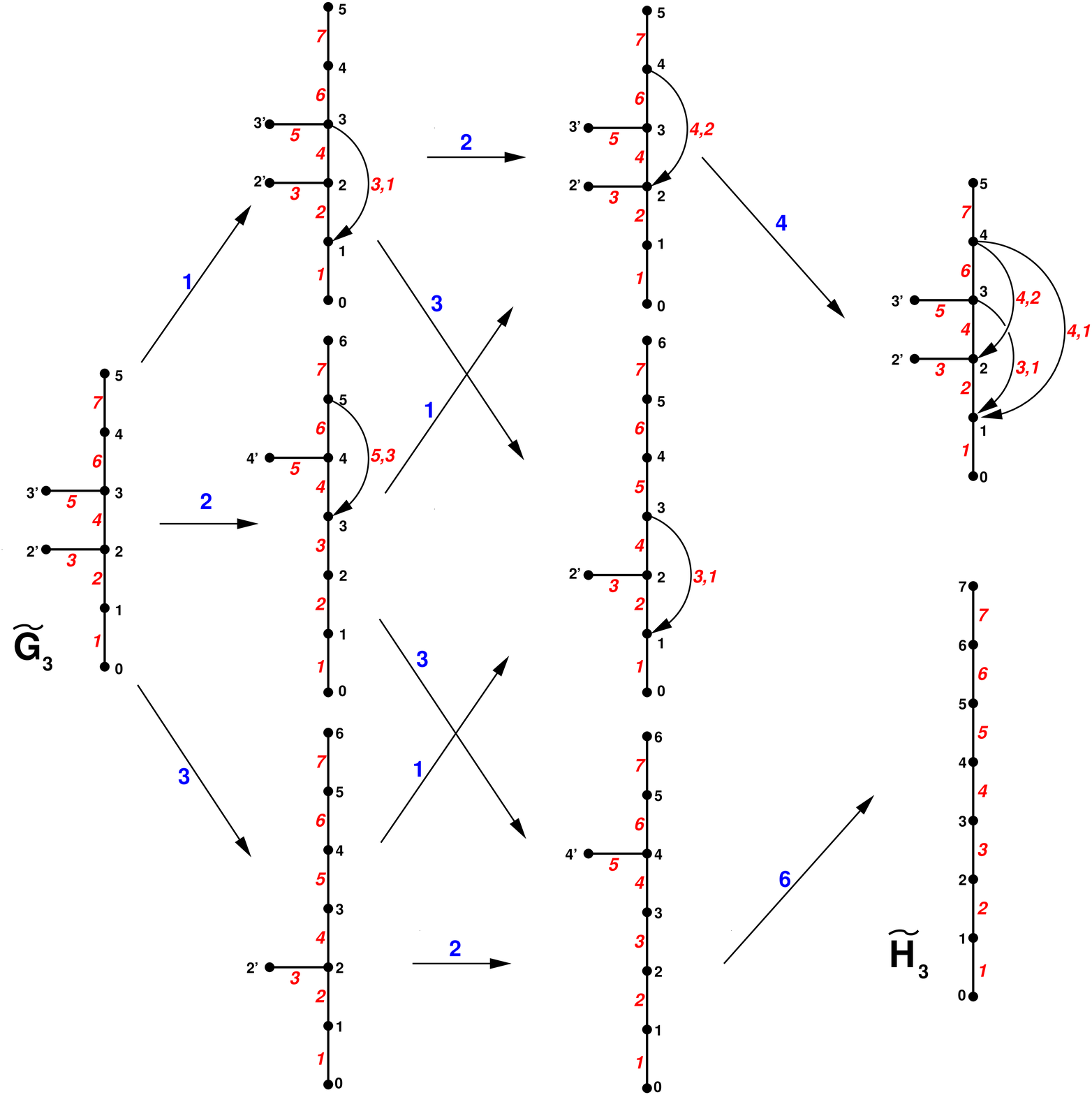}
\caption{\small The target graphs $\Gamma_\bM$ for the nine seeds of the case $A_3$
(labelled by the 9 Motzkin paths of Fig.\ref{fig:nine}), and the corresponding
mutations. We have indicated vertex and edge labels.}\label{fig:pathslfour}
\end{figure}

In this appendix, we detail the path interpretation of the solution $R_{1,n}$ of the $A_3$ $Q$-system,
as expressed in terms of the various initial data corresponding to the 9 cluster variables of the fundamental domain
${\mathcal F}_3$, labelled by the 9 Motzkin paths of Fig.\ref{fig:nine}.

As a direct illustration of the constructions of Section \ref{pathinter}, we list below for each of the
nine cluster variables the corresponding transfer matrix $T_\Gamma$, and the generating function
$\left((I-T_\Gamma)^{-1}\right)_{0,0}$, in terms of dummy weights $y_i$ and $y_{i,j}$,
that are actually shorthand notations for the weights $y_i(\bM)$ and their redundant
counterparts, to be extracted from Theorem
\ref{weightexpression} and eq.\eqref{solyij}. 
The corresponding graphs are represented in Fig.\ref{fig:pathslfour},
together with the mutations relating the associated cluster variables.

\noindent{\bf For $\bM_0=\bf \{(0,1),(0,2),(0,3)\}$:} The corresponding target 
graph is $\Gamma_{\bM_0}={\tilde G}_3$.
The transfer matrix and the associated generating function are:
$$
T_{\Gamma_{\bM_0}}={\small \left( 
\begin{matrix}
0&ty_1&0&0&0&0&0&0\\
1&0&ty_2&0&0&0&0&0\\
0&1&0&ty_3&ty_4&0&0&0\\
0&0&1&0&0&0&0&0\\
0&0&1&0&0&ty_5&ty_6&0\\
0&0&0&0&1&0&0&0\\
0&0&0&0&1&0&0&ty_7\\
0&0&0&0&0&0&1&0
\end{matrix} 
\right)},\quad \left((I-T_{\Gamma_{\bM_0}})^{-1}\right)_{0,0}=
{1\over 1-t {y_1\over 1-t {y_2\over 1-t y_3 -t {y_4\over 1-t y_5 -t{y_6\over 1-t y_7}}}}}
$$

\noindent{\bf For $\bM_1=\bf \mu_1\bM_0=\{(1,1),(0,2),(0,3)\}$,}
with $y_{3,1}=y_2y_4/y_3$:
$$
T_{\Gamma_{\bM_1}}={\small \left( 
\begin{matrix}
0&ty_1&0&0&0&0&0&0\\
1&0&ty_2&0& ty_{3,1}&0&0&0\\
0&1&0&ty_3&ty_4&0&0&0\\
0&0&1&0&0&0&0&0\\
0&0&1&0&0&ty_5&ty_6&0\\
0&0&0&0&1&0&0&0\\
0&0&0&0&1&0&0&ty_7\\
0&0&0&0&0&0&1&0
\end{matrix} 
\right)},\quad\left((I-T_{\Gamma_{\bM_1}})^{-1}\right)_{0,0}=
{1\over 1-t {y_1\over 1-t {y_2+{y_{3,1}\over 1-t y_5-t{y_6\over 1-t y_7}}\over
1-t y_3-t{y_4\over  1-t y_5-t{y_6\over 1-t y_7}}}}}.
$$

\noindent{\bf For $\bM_2=\bf \mu_2\bM_0=\{(0,1),(1,2),(0,3)\}$,} with $y_{5,3}=y_4y_6/y_5$:
$$
T_{\Gamma_{\bM_2}}={\small \left( 
\begin{matrix}
0&ty_1&0&0&0&0&0&0\\
1&0&ty_2&0&0&0&0&0\\
0&1&0&ty_3&0&0&0&0\\
0&0&1&0&ty_4&0&t y_{5,3}&0\\
0&0&0&1&0&ty_5&ty_6&0\\
0&0&0&0&1&0&0&0\\
0&0&0&0&1&0&0&ty_7\\
0&0&0&0&0&0&1&0
\end{matrix} 
\right)}\quad\left((I-T_{\Gamma_{\bM_2}})^{-1}\right)_{0,0}={1\over 1-t {y_1\over 1-t{y_2\over 1-t{y_3\over 
1-t{y_4+{y_{5,3}\over 1-t y_7}\over 1-t y_5-t{y_6\over 1-ty_7}}}}}}.
$$

\noindent{\bf For $\bM_3=\bf \mu_3\bM_0=\{(0,1),(0,2),(1,3)\}$:}
$$
T_{\Gamma_{\bM_3}}={\small \left( 
\begin{matrix}
0&ty_1&0&0&0&0&0&0\\
1&0&ty_2&0&0&0&0&0\\
0&1&0&ty_3&ty_4&0&0&0\\
0&0&1&0&0&0&0&0\\
0&0&1&0&0&ty_5&0&0\\
0&0&0&0&1&0&ty_6&0\\
0&0&0&0&0&1&0&ty_7\\
0&0&0&0&0&0&1&0
\end{matrix} 
\right)},\quad\left((I-T_{\Gamma_{\bM_3}})^{-1}\right)_{0,0}=
{1\over 1-t{y_1\over 1-t {y_2\over 1-t y_3-t{y_4\over 1-t {y_5\over 1-t {y_6\over 1-t y_7}}}}}}.
$$

\noindent{\bf For $\bM_{2,1}=\bf
  \mu_2\mu_1\bM_0=\{(1,1),(1,2),(0,3)\}$:} With $y_{4,2}=y_4y_6/y_5$,
$$
T_{\Gamma_{\bM_{2,1}}}={\small \left( 
\begin{matrix}
0&ty_1&0&0&0&0&0&0\\
1&0&ty_2&0&0&0&0&0\\
0&1&0&ty_3&ty_4&0&ty_{4,2}&0\\
0&0&1&0&0&0&0&0\\
0&0&1&0&0&ty_5&ty_6&0\\
0&0&0&0&1&0&0&0\\
0&0&0&0&1&0&0&ty_7\\
0&0&0&0&0&0&1&0
\end{matrix} 
\right)},\quad
\left((I-T_{\Gamma_{\bM_{2,1}}})^{-1}\right)_{0,0}=
{1\over 1-t{y_1\over 1-t {y_2\over 1-t y_3-t{y_4+{y_{4,2}\over 1-t y_7}\over 1-t y_5-t{y_6\over 1-ty_7}}}}}.
$$

\noindent{\bf For $\bM_{3,1}=\bf
  \mu_3\mu_1\bM_0=\{(1,1),(0,2),(1,3)\}$:} With $y_{3,1}=y_2y_4/y_3$, 
$$
T_{\Gamma_{\bM_{3,1}}}={\small \left( 
\begin{matrix}
0&ty_1&0&0&0&0&0&0\\
1&0&ty_2&0& ty_{3,1}&0&0&0\\
0&1&0&ty_3&ty_4&0&0&0\\
0&0&1&0&0&0&0&0\\
0&0&1&0&0&ty_5&0&0\\
0&0&0&0&1&0&ty_6&0\\
0&0&0&0&0&1&0&ty_7\\
0&0&0&0&0&0&1&0
\end{matrix} 
\right)},\quad
\left((I-T_{\Gamma_{\bM_{3,1}}})^{-1}\right)_{0,0}=
{1\over 1-t{y_1\over 1-t{y_2+{y_{3,1}\over 1-t{y_5\over 1-t{y_6\over 1-t y_7}}}\over
1-t y_3-t{y_4\over 1-t{y_5\over 1-t{y_6\over 1-t y_7}}}}}}
$$

\noindent{\bf For $\bM_{3,2}=\bf \mu_3\mu_2\bM_0=\{(0,1),(1,2),(1,3)\}$:}
$$
T_{\Gamma_{\bM_{3,2}}}={\small \left( 
\begin{matrix}
0&ty_1&0&0&0&0&0&0\\
1&0&ty_2&0&0&0&0&0\\
0&1&0&ty_3&0&0&0&0\\
0&0&1&0&ty_4&0&0&0\\
0&0&0&1&0&ty_5&ty_6&0\\
0&0&0&0&1&0&0&0\\
0&0&0&0&1&0&0&ty_7\\
0&0&0&0&0&0&1&0
\end{matrix} 
\right)},\quad
\left((I-T_{\Gamma_{\bM_{3,2}}})^{-1}\right)_{0,0}=
{1\over 1-t{y_1\over1-t{y_2\over 1-t{y_3\over 1-t{y_4\over 1-t y_5-t{y_6\over 1-t y_7}}}}}}.
$$

\noindent{\bf For $\bM_{4,2,1}=\bf \mu_4\mu_2\mu_1\bM_0=\{(2,1),(1,2),(0,3)\}$:}
With $y_{3,1}=y_2y_4/y_3$, $y_{4,2}=y_4y_6/y_5$, $y_{4,1}=y_2y_4y_6/(y_3y_5)$,
$$
T_{\Gamma_{\bM_{4,2,1}}}={\small \left( 
\begin{matrix}
0&ty_1&0&0&0&0&0&0\\
1&0&ty_2&0&t y_{3,1}&0&t y_{4,1}&0\\
0&1&0&ty_3&ty_4&0&ty_{4,2}&0\\
0&0&1&0&0&0&0&0\\
0&0&1&0&0&ty_5&ty_6&0\\
0&0&0&0&1&0&0&0\\
0&0&0&0&1&0&0&ty_7\\
0&0&0&0&0&0&1&0
\end{matrix} 
\right)},\quad\left((I-T_{\Gamma_{\bM_{4,2,1}}})^{-1}\right)_{0,0}=
{1\over 1-t {y_1\over 1-t{y_2+{y_{3,1}+{y_{4,1}\over 1-t y_7}\over 1-t y_5 -t{y_6\over 1-t y_7}}
\over 1-ty_3-t{y_4+{y_{4,2}\over 1-ty_7}\over  1-t y_5 -t{y_6\over 1-t y_7}}}}}
$$

\noindent{\bf For $\bM_{6,3,2}=\bf \mu_6\mu_3\mu_2\bM_0=\{(0,1),(1,2),(2,3)\}$:}
$$
T_{\Gamma_{\bM_{6,3,2}}}={\small \left( 
\begin{matrix}
0&t y_1&0&0&0&0&0&0\\
1&0&t y_2&0&0&0&0&0\\
0&1&0&t y_3&0&0&0&0\\
0&0&1&0&t y_4&0&0&0\\
0&0&0&1&0&t y_5&0&0\\
0&0&0&0&1&0&t y_6&0\\
0&0&0&0&0&1&0&t y_7\\
0&0&0&0&0&0&1&0
\end{matrix} 
\right)},\quad
\left((I-T_{\Gamma_{\bM_{6,3,2}}})^{-1}\right)_{0,0}=
{1\over 1-t{y_1\over 1-t{y_2\over 1-t {y_3\over 1-t {y_4\over 1-t {y_5\over 1-t {y_6\over 1-t y_7}}}}}}}.
$$

Note finally that the graphs in Figures \ref{fig:slfour} and \ref{fig:pathslfour} are duals of each other. 
This duality
is best seen by expressing a bijection between the paths (represented with their time extension)
and the corresponding heaps, by associating a disc of the heap to each descent of the path.
The heap graph is simply the graph whose vertices stand for the discs and whose edges indicate
the pairs of overlapping discs.\label{appendixb}


\begin{thebibliography}{10}

\bibitem{STIEL2} R. Beals, D. H. Sattinger and J. Szmigielski,
  \emph{Continued fractions and integrable systems},
  J. Comput. Appl. Math. {\bf 153} No.1-2 (2003) 47-60.

\bibitem{POSIT} P. Caldero and M. Reineke, \emph{On the quiver
    Grassmannian in the acyclic case}.  J. Pure Appl. Algebra   \textbf{212}
  (2008),  no. 11, 2369--2380. {\tt arXiv:math/0611074
    [math.RT]}.
\bibitem{DFK} P. Di Francesco and R. Kedem, \emph{Proof of the
    combinatorial Kirillov-Reshetikhin conjecture}. Int. Math. Res. Notices, 
 Vol. {\bf 2008} (2008), rnn006-57.  {\tt
    arXiv:0710.4415 [math.QA]}.
\bibitem{DFK08} P. Di Francesco and R. Kedem, \emph{Q-systems as
    cluster algebras II}. {\tt
    arXiv:0803.0362 [math.RT]}.
\bibitem{DFK09}P. Di Francesco and R. Kedem, \emph{$Q$-systems cluster algebras, paths and total positivity}, {\tt arXiv:0906.3421 [math.co]}.
\bibitem{DFK09a}P. Di Francesco and R. Kedem, \emph{Positivity of the
$T$-system cluster algebra}, preprint.
\bibitem{DFK09b}P. Di Francesco and R. Kedem, in progress.

\bibitem{FZposit} S. Fomin And A. Zelevinsky \emph{Total positivity:
    tests and parameterizations}, 
        Math. Intelligencer \textbf{22} (2000), 23-33.  {\tt
    	arXiv:math/9912128 [math.RA]}.

\bibitem{FZI} S. Fomin and A. Zelevinsky Cluster Algebras I.
  J. Amer. Math. Soc.   \textbf{15}  (2002),  no. 2, 497--529 {\tt
    arXiv:math/0104151 [math.RT]}.

\bibitem{FZLaurent} S. Fomin And A. Zelevinsky \emph{The Laurent
    phenomenon}.   Adv. in Appl. Math.   \textbf{28}  (2002),  no. 2,
  119--144. {\tt arXiv:math/0104241 [math.CO]}. 

\bibitem{LGV2} I. M. Gessel and X. Viennot, \emph{Binomial
    determinants, paths and hook formulae}, Adv. Math. \textbf{58}
  (1985) 300-321.  

\bibitem{AH} A. Henriqu\`es, \emph{A periodicity theorem for the
    octahedron recurrence}, 
             J. Algebraic Comb. \textbf{26} No1 (2007) 1-26. {\tt
          	arXiv:math/0604289 [math.CO]}.
	
\bibitem{KJ} K. Johansson, \emph{Non-intersecting paths, random tilings and random matrices},
Probability theory and related fields \textbf{123} (2002), no. 2, 225-280. {\tt
arXiv:math/0011250 [math.PR]}.

\bibitem{Ke07} R. Kedem, \emph{$Q$-systems as cluster algebras}.  J.
  Phys. A: Math. Theor. \textbf{41} (2008) 194011 (14 pages). {\tt
    arXiv:0712.2695 [math.RT]}.

\bibitem{KR} A.~N. Kirillov and N.~Yu. Reshetikhin,
  \emph{Representations of Yangians 
and multiplicity of occurrence of the irreducible components of the
tensor product 
of representations of simple {L}ie algebras}, J. Sov. Math. {\bf 52}
(1990) 3156-3164. 
\bibitem{Kon} M. Kontsevich, private communication.

\bibitem{KT}A. Knutson, T. Tao, and C. Woodward, 
\emph{A positive proof of the Littlewood-Richardson rule using the
  octahedron recurrence},   
Electr. J. Combin. \textbf{11} (2004) RP 61. {\tt
	     arXiv:math/0306274 [math.CO]}

\bibitem{HEAPS2} C. Krattenthaler, \emph{The theory of heaps and the
    Cartier-Foata monoid},  
appendix to the electronic republication of
\emph{Probl\`emes combinatoires de commutation et r\'earrangements}, 
by Pierre Cartier and Dominique Foata, available at
http://www.mat.univie.ac.at/~slc/books/cartfoa.html

\bibitem{LGV1} B. Lindstr\"om, \emph{On the vector representations of
	induced matroids}, Bull. London Math. Soc. \textbf{5} (1973)
	85-90. 

\bibitem{LP} A. Postnikov, \emph{Total positivity, Grassmannians, and
    networks}. {\tt	 
	arXiv:math/0609764v1 [math.CO]}.

\bibitem{RR} D. Robbins and H. Rumsey, \emph{Determinants and
    Alternating Sign Matrices}, 
             Advances in Math. \textbf{62} (1986) 169-184.

\bibitem{SPY} D. Speyer, \emph{Perfect matchings and the
        octahedron recurrence}, J. Algebraic Comb. \textbf{25} No 3
      (2007) 309-348. {\tt arXiv:math/0402452 [math.CO]}.

\bibitem{STIEL1} T.J. Stieltjes, \emph{Recherches sur les fractions
    continues}, in  
\emph{Oeuvres compl\`etes de Thomas Jan Stieltjes}, Vol. II,
No. LXXXI, P. Nordhoff, Groningen, 
(1918) 402-566 (see in particular eq.(1) p 402 and eq.(7) p 427).

\bibitem{HEAPS1} X. Viennot, \emph{Heaps of pieces I: Basic
        definitions and combinatorial lemmas}, in ``Combinatoire
      \'enum\'erative'', eds G. Labelle and P. Leroux, Lecture Notes
      in Math. \textbf{1234}, Springer-Verlag, Berlin (1986) 321-325.




\end{thebibliography}
\end{document}